\documentclass[11pt]{amsart}
\DeclareSymbolFont{bbold}{U}{bbold}{m}{n}
\DeclareSymbolFontAlphabet{\mathbbold}{bbold}
\usepackage{amsmath,amsfonts,amssymb,amsthm,color,tikz,xspace,euscript,nicefrac,mathrsfs,calligra,enumitem,extraipa,bbm,fancybox,todonotes}
\usepackage{quiver}
\usepackage[russian,english]{babel}
\usepackage{hyperref}
\newtheorem{Theorem}{Theorem}[section]
\newtheorem{Proposition}[Theorem]{Proposition}
\newtheorem{Lemma}[Theorem]{Lemma}
\newtheorem{Question}[Theorem]{Question}
\newtheorem{Corollary}[Theorem]{Corollary}
\newtheorem{Conjecture}[Theorem]{Conjecture}

\theoremstyle{definition}
\newtheorem{Definition}[Theorem]{Definition}
\newtheorem{Example}[Theorem]{Example}
\newtheorem{Remark}[Theorem]{Remark}
\usepackage[margin=3cm,footskip=25pt,headheight=20pt]{geometry}
\usepackage{hyperref}

\numberwithin{equation}{section}
\renewcommand{\theequation}{\arabic{section}.\arabic{equation}}

\tikzset{wei/.style={draw=red,double=red!40!white,double distance=1.5pt,thin}}
\newcounter{subeqn}
\renewcommand{\thesubeqn}{\theequation\alph{subeqn}}
\newcommand{\subeqn}{%
  \refstepcounter{subeqn}
  \tag{\thesubeqn}
}
\makeatletter
\@addtoreset{subeqn}{equation}
\newcommand{\newseq}{%
  \refstepcounter{equation}
}

\newcommand{\nc}{\newcommand}
\newcommand{\renc}{\renewcommand}
\newcommand{\arxiv}[1]{\href{http://arxiv.org/abs/#1}{\tt arXiv:\nolinkurl{#1}}}

\nc{\tFW}{\widetilde{\mathsf{FW}}}
\nc{\FW}{{\mathsf{FW}}}
\nc{\bla}{{bla}}
\nc{\Supp}{\operatorname{Supp}}
\nc{\mmod}{\operatorname{-mod}}
\nc{\dash}{\operatorname{-}}
\nc{\h}{\mathfrak h}
\nc{\g}{\mathfrak g}
\nc{\walg}{{\tilde{T}}}
\nc{\ft}{\mathfrak t}
\nc{\fM}{\mathfrak M}
\nc{\bM}{\mathbf M}
\nc{\bR}{\mathbf R}
\nc{\bm}{\mathbf m}
\nc{\bS}{\mathbf S}
\nc{\bT}{\mathbf T}
\nc{\bU}{\mathbf U}
\nc{\bV}{\mathbf V}
\nc{\bi}{\mathbf i}
\nc{\Bv}{\mathbf{v}}
\nc{\Bw}{\mathbf{w}}
\nc{\TL}{\tilde{\mathscr{T}}_{\mathcal{L}}}
\nc{\bs}{\mathbf s}
\nc{\Cx}{\C^\times}
\nc{\Sym}{\operatorname{Sym}}
\nc{\acham}{\eta}
\nc{\tU}{\mathcal{U}}

\nc{\longi}{\boldsymbol{\ell}}
\nc{\red}{{\operatorname{red}}}
\nc{\ind}{\operatorname{ind}}
\nc{\Res}{{\operatorname{Res}}}
\nc{\Ind}{\operatorname{Ind}}
\nc{\Coind}{\operatorname{Coind}}
\nc{\wgmod}{\operatorname{-wgmod}}
\nc{\coind}{\operatorname{coind}}
\nc{\ff}{\mathfrak{f}}
\nc{\Z}{{\mathbb Z}}
\renc{\C}{\mathbb C}
\nc{\R}{\mathbb R}
\nc{\N}{\mathbb N}
\nc{\B}{\mathcal B}
\nc{\M}{\mathcal M}
\nc{\cE}{\mathcal E}
\nc{\cF}{\mathcal F}
\nc{\fB}{\mathfrak B}
\nc{\fKLR}{\widetilde{\mathsf{f\! T}}}
\nc{\sfKLR}{\mathsf{f\!T}}
\nc{\con}{\sim}
\nc{\pgl}{\mathfrak{pgl}}
\nc{\Wei}{\EuScript{W}}
\nc{\ev}{\mathsf{ev}}
\nc{\Hom}{\operatorname{Hom}}
\nc{\End}{\operatorname{End}}
\nc{\res}{\operatorname{res}}
\nc{\al}{\alpha}
\nc{\vp}{\varphi}
\nc{\Cth}{S_h}
\nc{\cO}{\mathcal{O}}
\nc{\cK}{\mathcal{K}}
\nc{\fg}{\mathfrak{g}}
\nc{\one}{\mathbf{1}}
\nc{\bb}{\mathbf{b}}
\nc{\ext}{\operatorname{Ext}}
\nc{\out}{\operatorname{out}}
\nc{\FY}{FY}
\nc{\EP}{\mathcal{S}}
\nc{\ep}{\epsilon}
\nc{\bz}{{\mathbf z}}
\nc{\inn}{\operatorname{in}}
\nc{\BK}{\reflectbox{\rm R}}
\nc{\Bi}{\mathbf{i}}
\nc{\Bj}{\mathbf{j}}
\nc{\Ba}{\mathbf{a}}
\nc{\Bb}{\mathbf{b}}
\nc{\Bm}{\mathbf{m}}
\nc{\tGamma}{\tilde{\Gamma}}
\nc{\tGammabR}{\tGamma_{\bR}}
\nc{\GammabR}{\Gamma_{\bR}}
\nc{\diam}{\diamond}
\nc{\la}{\lambda}
\nc{\ul}{\underline{\la}}
\nc{\Yml}{Y_\mu^\lambda}
\nc{\bgam}{{\boldsymbol{\gamma}}}
\nc{\blam}{{\boldsymbol{\lambda}}}
\nc{\gr}{\operatorname{gr}}
\nc{\Spec}{\operatorname{Spec}}
\nc{\Stendhal}{Stendhal\xspace}
\nc{\MaxSpec}{\operatorname{MaxSpec}}
\nc{\Cartan}{\C[H_\bullet^{(\bullet)}]}
\nc{\tmetric}{\mathscr{\tilde T}}
\nc{\metric}{\mathscr{T}}
\nc{\pmmetric}{{}_{\pm}\mathscr{T}}
\nc{\pmetric}{{}_{+}\mathscr{T}}
\nc{\mmetric}{{}_{-}\mathscr{T}}
\nc{\Pol}{\mathsf{Pol}}
\nc{\hh}{h}
\nc{\wtmodY}{{Y^\la_\mu\text{-wt mod}}}
\nc{\wtmodFY}{{FY^\la_\mu\text{-wt mod}}}
\nc{\ftR}{\ft_\R}
\nc{\ftZ}{\ft_\Z}
\nc{\ftgen}{\ft^{\operatorname{gen}}}
\nc{\sfr}{\mathsf{r}}
\nc{\Wh}{\operatorname{Wh}}
\nc{\bWh}{\overline{\Wh}}
\nc{\cores}{\operatorname{cores}}
\nc{\op}{\operatorname{op}}
\nc{\vertex}{I}
\nc{\edge}{\mathsf{E}}
\nc{\GTc}{\mbox{\rm\foreignlanguage{russian}{ГЦ}}}
\nc{\Gplus}{G_+}
\nc{\Stein}{\mathbb{X}}
\nc{\pStein}{\mathbb{Y}}
\nc{\rif}{\mathscr{I}}
\nc{\corif}{co\mathscr{I}}
\nc{\MS}{\EuScript{M}}
\newcommand{\Sg}{\mathsf{G}}
\newcommand{\Sr}{\mathsf{R}}
\newcommand{\Sc}{\mathsf{C}}
\newcommand{\Sgr}{\mathsf{GR}}
\newcommand{\Sall}{\mathsf{CGR}}
\newcommand{\lx}{l}

\nc{\cA}{\mathcal A}
\nc{\cR}{\mathcal R}
\nc{\cU}{\mathcal U}
\nc{\fw}{\mathfrak w}
\nc{\anti}{\sigma}
\nc{\oW}{\overline{\mathcal W}}

\nc{\hooklongrightarrow}{\lhook\joinrel\longrightarrow}
\nc{\twoheadlongrightarrow}{\relbar\joinrel\twoheadrightarrow}
\nc{\hooklongleftarrow}{\longleftarrow\joinrel\rhook}
\nc{\twoheadlongleftarrow}{\twoheadleftarrow\joinrel\relbar}

\nc{\eps}{\varepsilon}
\nc{\Gr}{\mathsf{Gr}}
\nc{\CB}{\cA}
\nc{\CBs}[1]{\cA_{#1}}
\nc{\ICB}{\cA^B}		 
\nc{\ACB}{\cA^{\textsf{ab}}}	 
\nc{\PCB}{\cA^P}
\nc{\CBh}{\cA_{\hbar}}
\nc{\ICBh}{\cA^B_{\hbar}}		 
\nc{\ACBh}{\cA_{\hbar}^{\textsf{ab}}}	 
\nc{\PCBh}{\cA_{\hbar}^P}
\nc{\NH}{\mathcal{NH}_\fg}

\nc{\tG}{\tilde G}
\nc{\tL}{\tilde L}
\nc{\Vect}{\mathsf{Vect}}
\nc{\sslash}{\mathbin{/\mkern-6mu/}}
\nc{\Oof}[1]{{{#1}{\text{-}\cO}}}
\nc{\excise}[1]{}
\nc{\resBE}{\res_{\mathsf{BE}}}
\nc{\resCB}{\res_{\mathsf{CB}}}
\nc{\KZ}{\mathsf{KZ}}
\nc{\tP}{\tilde{P}}
\nc{\sph}{\operatorname{sph}}
\newcommand{\HBG}[1]{\mathbbold{\Lambda}_{#1}}

\nc{\GammaInt}{\widetilde{\Gamma}}
\nc{\vInt}{\widetilde{\Bv}}
\nc{\wInt}{\widetilde{\Bw}}
\nc{\varphiInt}{\widetilde{\varphi}}
\nc{\GammaO}{\Gamma_{\varphi}}

\nc{\wtAlg}{F}
\nc{\wtBim}{I}

\title[Lie algebra actions on module categories for truncated shifted Yangians]{Lie algebra actions on module categories for\\ truncated shifted Yangians}
\author[Kamnitzer]{Joel Kamnitzer}
\address{J.~Kamnitzer: Department of Mathematics, University of Toronto, Canada}
\email{jkamnitz@math.toronto.edu}

\author[Webster]{Ben Webster}
\address{B.~Webster: Department of Pure Mathematics, University of Waterloo \&
Perimeter Institute for Theoretical Physics, Canada}
\email{ben.webster@uwaterloo.ca}

\author[Weekes]{Alex Weekes}
\address{A.~Weekes: Department of Mathematics and Statistics, University of Saskatchewan, Canada}
\email{weekes@math.usask.ca}

\author[Yacobi]{Oded Yacobi}
\address{O.~Yacobi: School of Mathematics and Statistics, University of Sydney, Australia}
\email{oded.yacobi@sydney.edu.au}

\begin{document}
\begin{abstract}

We develop a theory of parabolic induction and
restriction functors relating modules over Coulomb branch algebras, in the sense of Braverman-Finkelberg-Nakajima.  Our functors generalize Bezrukavnikov-Etingof's
induction and restriction functors for Cherednik algebras, but their
definition uses different tools.

After this general definition, we focus on quiver gauge theories
attached to a quiver $\Gamma$.  The induction and restriction functors allow us to define a categorical action of the  corresponding
symmetric Kac-Moody algebra $\mathfrak{g}_{\Gamma}$
on category $ \mathcal O $ for these Coulomb branch algebras.   When $ \Gamma $ is of Dynkin type, the Coulomb
branch algebras are truncated shifted
Yangians and quantize generalized affine Grassmannian slices. Thus, we regard our action as a categorification of the geometric Satake correspondence.

To establish this categorical action, we define a new class of ``flavoured'' KLRW algebras, which are similar to the diagrammatic algebras originally constructed by the second author for the purpose of tensor product categorification. We prove an 
equivalence between the category
of Gelfand-Tsetlin modules over a Coulomb branch algebra and the modules over a flavoured KLRW algebra.  This equivalence relates
the categorical action by induction and restriction functors to the usual categorical action on modules over a KLRW
algebra.  
\end{abstract}
\maketitle

 \section{Introduction}

\subsection{Categorification and affine Grassmannian slices}
Let $ G_\Gamma$ be a semisimple complex group with Dynkin diagram $ \Gamma$.  In recent years, following ground-breaking ideas of Khovanov, there has been great interest in constructing categorifications of tensor product representations of $ G_\Gamma $.

Khovanov-Lauda \cite{KLI} and Rouquier \cite{Rou2KM} introduced a family of combinatorially-defined diagrammatic algebras for this purpose.  Their work was later extended by the second author, who defined KLRW algebras $   T^{\ul}_\mu $, for a list $\ul=(\la_1,\dots,\la_n)$ of dominant weights and a weight $ \mu$ of $ G_\Gamma $ \cite{Webmerged}.  The categories of modules over these algebras carry categorical $ \fg_\Gamma $ actions \cite[Th.~B.]{Webmerged}. 
Recall that a categorical $\fg_\Gamma $ action assigns a category to each weight space (in this case, $  T^{\ul}_\mu\mmod$) and a functor 
\begin{equation} \label{eq:EiKLR}
	\EuScript E_i :   T^{\ul}_\mu\mmod \rightarrow   T^{\ul}_{\mu + \alpha_i}\mmod \qquad \EuScript F_i :   T^{\ul}_\mu\mmod \rightarrow   T^{\ul}_{\mu - \alpha_i}\mmod
\end{equation}
for each
$ i \in \vertex $.   We require that $ \oplus_\mu K_{\C}( T^{\ul}_\mu \mmod) \cong V(\lambda_1) \otimes \cdots \otimes V(\lambda_n) $ as representations of $ \fg_\Gamma $, with $ \EuScript E_i, \EuScript F_i $ categorifying the Chevalley generators of $ \fg_\Gamma $.  Joint work of this author and Losev shows that these categories are the unique {\bf tensor product categorification} for the tensor product above \cite[Th.~A]{LoWe}.  Thus, other categorifications of tensor products appearing in the literature will typically be equivalent to $  T^{\ul}_\mu\mmod$.  The most notable of these is defined by translation functors on category $\mathcal{O}$, based on the original work of Bernstein-Frenkel-Khovanov \cite{BFK} in the case where $\fg_\Gamma=\mathfrak{sl}_2$.

In \cite{KWWY}, we began a project to construct categorifications using affine Grassmannian slices and their quantizations.  Affine Grassmannian slices $ \oW^\lambda_\mu $ are defined for pairs of dominant weights $\lambda,\mu$ where $\mu \leq \la$.  In \cite{BFNslices}, these were generalized to arbitrary weights $ \mu $.  These generalized affine Grassmannian slices are affine Poisson varieties with symplectic singularities \cite{KWWY, Weekes2,Yehao,Bellamy}, defined using the affine Grassmannian of the Langlands dual group $ G_\Gamma^\vee $.  The main feature of these varieties is that they contain the Mirkovi\'c-Vilonen cycles  as the attracting loci of their fixed points \cite{KWWY,Krylov}.

Generalized affine Grassmannian slices admit natural quantizations $ Y^\lambda_\mu $, called {\bf truncated shifted Yangians}, defined for dominant $ \mu $ in \cite{KWWY} and for general $\mu $ in \cite{BFNslices}.  Motivated by the geometric Satake correspondence of Mirkovi\'c-Vilonen \cite{MV} and the philosophy of Braden-Licata-Proudfoot-Webster \cite{BLPWgco}, we conjectured in \cite{KWWY} that category $\mathcal O $ for these algebras could be used to construct a categorification of tensor product representations.  More precisely, these algebras appear as a family over a space of quantization parameters, which in this paper are called {\bf flavours}; we let $\Oof{Y^\lambda_\mu}_\Z$ be the category $\mathcal{O}$ over this algebra for a generic integral flavour.

We expected the existence of exact functors
\begin{equation} \label{eq:EiY}
	\EuScript E_i : \Oof{Y^\lambda_\mu}_\Z \rightarrow \Oof{Y^\lambda_{\mu + \alpha_i}}_\Z  \qquad \EuScript F_i : \Oof{Y^\lambda_\mu}_\Z  \rightarrow \Oof{Y^\lambda_{\mu - \alpha_i}}_\Z
\end{equation}
along with isomorphisms $\oplus_\mu K_\C(\Oof{Y^\lambda_\mu}_{\Z} ) \cong V(\lambda_1) \otimes \cdots \otimes V(\lambda_n) $ of $ \fg_\Gamma $-representations, for $\la_k$ fundamental weights satisfying $\la=\la_1+\cdots +\la_n$.

In \cite{KTWWYO}, we made a decisive step in this direction by constructing equivalences of abelian categories $ \Oof{Y^\lambda_\mu}_{\Z} \cong   T^{\ul}_\mu \mmod $.  In this way, we are able to define the functors (\ref{eq:EiY}) using the functors (\ref{eq:EiKLR}).  However, this left the following question, which we will resolve in this paper:

\begin{Question} \label{qu:intro}
	Can we define the functors (\ref{eq:EiY}) directly using truncated shifted Yangians?
\end{Question}

\subsection{Restriction functors for Coulomb branches}

In this paper, we will construct these functors (\ref{eq:EiY}) by realizing them in the larger context of Coulomb branches.   For any reductive group $ \tilde{G} $, a representation $ N $ and a normal subgroup $G\subset \tilde{G}$ such that $F=\tilde{G}/G$ (the flavour group) is a torus, Braverman-Finkelberg-Nakajima \cite{BFN} constructed an affine Poisson variety $ M_C(G,N) $ and its deformation quantization $ \CB(G,N)$.  In the notation of \cite{BFN}, this second algebra would be denoted $ \CB_{\hbar}(G,N)/(\hbar-1)$, but since we will be interested primarily in this non-commutative deformation or specializations of it, we will drop the subscript. We note that $ \CB(G,N)$ carries a filtration whose associated graded $ \gr \CB(G,N) $ is isomorphic to the coordinate ring of $ \CB(G,N) $.

Following the physics literature:
\begin{itemize}
	\item  The variety $ M_C(G,N) $ is called the (flavour deformation of the) {\bf Coulomb branch} of the gauge theory defined by $ G, N$. 
	\item The algebra $ \CB(G,N) $ is called the {\bf Coulomb branch algebra.}
\end{itemize}
An important feature of the Coulomb branch construction is that it comes with a complete integrable system $ M_C(G,N) \rightarrow \tilde{\ft}/ W \cong \tilde{\fg}/\hspace{-1mm}/G$ where $\tilde{\ft}$ is a Cartan subalgebra of $\tilde{\fg}=\operatorname{Lie}(\tilde{G})$. The quantization of this integrable system yields a maximal commutative subalgebra 
$ (\Sym \tilde{\ft}^*)^W \subset \CB(G,N) $, which we call the {\bf Gelfand-Tsetlin subalgebra}, since it generalizes the usual Gelfand-Tsetlin subalgebra of $ U( \mathfrak{gl}_n )$.  

We consider modules over $\CB(G,N) $, which are locally finite as $(\Sym \tilde{\ft}^*)^W$-modules.  We call these {\bf Gelfand-Tsetlin modules}, and let $\CB(G,N)\dash{\GTc}$ be the category of Gelfand-Tsetlin modules (Definition \ref{def:Acat}).  Modules $M $ in this category have decompositions into generalized eigenspaces for the action of the Gelfand-Tsetlin algebra: \[ M = \bigoplus_{\gamma \in \tilde{\ft}/W}  \Wei_\gamma(M) .\]  Throughout this paper, we will only study modules on which the center of $\CB(G,N) $ acts semi-simply, though we include discussion of how our results can be modified to account for more general modules.

The category of Gelfand-Tsetlin modules contains category $\cO$ for $\CB(G,N)$-modules, denoted $\Oof{\CB(G,N)}$, consisting of modules satisfying a local-finiteness condition with respect to a $\C^\times$-action (cf. Definition \ref{def:catO}).  

When $ G, N $ are constructed using ADE quiver data, then by \cite{BFNslices,WeekesCoulomb}:
\begin{itemize}
	\item 
	The variety $ M_C(G,N) $ is a generalized affine Grassmannian slice.
	\item The algebra $ \CB(G,N) $ is a truncated shifted Yangian $Y^\lambda_\mu$ 
\end{itemize} See Section \ref{se:Ylamu} for the precise versions of these statements.  Our main insight in this paper is that the relationship between $ Y^\lambda_\mu$ and $ Y^\lambda_{\mu + \alpha_i} $ can be studied as a special case of a more general ``parabolic restriction'' of Coulomb branches.

Let $ \xi : \Cx \rightarrow T $ be a coweight of the group $ G$.  This determines a Levi subgroup $ L $ of $ G $ (its centralizer) and an $L$-subrepresentation $ N_0^\xi $ (the subspace of invariants for the cocharacter $\xi$).  When $ G, N $ are chosen so that $ \CB(G,N) \cong Y^\lambda_\mu $ (Section \ref{se:Ylamu}), there is a particular choice of $ \xi $ (see Theorem \ref{th:res-E_i}) yielding $ \CB(L/\Cx_\xi, N_0^\xi) = Y^\lambda_{\mu + \alpha_i}$.  This inspired us to examine the relation between $ \CB(G,N) $ and $ \CB(L / \Cx_\xi, N_0^\xi) $ in the more general setting.

Geometrically, $ \dim M_C(L/\Cx_\xi, N_0^\xi) = \dim M_C(G,N) - 2 $ and so one might expect to obtain $ M_C(L/\Cx_\xi, N_0^\xi) $ by Hamiltonian reduction of $ M_C(G,N) $ by the action of an additive group.  In the finite ADE quiver situation, it is possible to achieve this (see \cite{KPW} and Remark \ref{re:comult}), but not in a way compatible with the above mentioned integrable systems.  Thus, in this paper, we pursue a different approach.

Our first main result (Theorem \ref{thm:rel}) describes the relationship between the four algebras $ \CB(G,N), \CB(L,N), \CB(L, N_0^\xi)$ and  $\CB(L/\Cx_\xi, N_0^\xi) $.  The exact statement is a bit complicated, but most importantly for our purposes, the maps between these algebras are compatible with their Gelfand-Tsetlin subalgebras, and allow us to prove the following result (Theorem \ref{th:res-weight}).

\begin{Theorem} \label{th:IntroRes}
	There is a restriction functor $ \CB(G,N)\dash\GTc \xrightarrow{\res} \CB(L, N_0^\xi)\dash\GTc $ such that	$$
	\Wei_\nu(\res(M)) = \Wei_\nu(M)
	$$
	for all $ \nu \in \tilde{\ft}/W_L $ satisfying a condition called $ \xi$-negative (which can be achieved by subtracting a sufficiently large integral multiple of $ \xi $).
\end{Theorem}

Finally, the algebra  $\CB(L/\Cx_\xi, N_0^\xi) $ is the  quantum Hamiltonian reduction of $ \CB(L, N_0^\xi)$  by the ``monopole'' operator $ r_\xi $.  The Hamiltonian reduction functor also preserves the category of Gelfand-Tsetlin modules.  The effect of this functor on weight spaces is described in Proposition \ref{prop: qhr and gt}.

\subsection{Relating the functors}

In this paper we will consider any quiver $ \Gamma$ with vertex set $\vertex$ and edge set $\edge$. If $\Gamma$ has no edge loops, then it has an associated symmetric Kac-Moody Lie algebra $\fg_\Gamma$.  For the time being, we include the case where $\Gamma$ has edge loops, but of course any statement that uses $\fg_\Gamma$ can only apply in the case of no edge loops.  Note that this is more general than our previous work \cite{KTWWYO}, where only simply-laced types with bipartite Dynkin diagrams are studied.  We will write $M_C(\Bv,\Bw) $ for the Coulomb branch, and $ \CB(\Bv,\Bw) $ for the Coulomb branch algebra, defined using $\Gamma$ and dimension vectors $ \Bv,\Bw \in \Z^{\vertex} $.   This means that we are using the gauge group and representation defined by
$$
G=\prod_i GL(v_i) \qquad N=\bigoplus_{i\to j\in \edge }\Hom(\C^{v_i},\C^{v_j})\oplus
\bigoplus_{i}\Hom(\C^{v_i},\C^{w_i}).
$$
We can consider the extended group $\tilde{G}=G\times (\Cx)^{\edge}\times \prod_i (\Cx)^{w_i}$ where the second and third factors act by scaling on the summands of $N$.  
When $ \Gamma $ is of ADE type, then $ \CB(\Bv,\Bw) $ is isomorphic to $ Y^\lambda_\mu $ where we define $\lambda=\sum w_i \varpi_i$ and $\mu=\lambda-\sum v_i \alpha_i$ (see Section \ref{se:Ylamu} for details).

We introduce a family of diagrammatic algebras which we call \textbf{flavoured KLRW algebras} $ \fKLR_{\Bv}^{\varphi} $, generalizing the metric KLRW algebras studied in \cite{KTWWYO}, and closely related to the weighted KLRW algebras introduced in \cite{WebwKLR}. These depend on a choice of dimension vectors $\Bv,\Bw$, and also a choice of flavour $\varphi $: a choice of complex number for each edge of the Crawley-Boevey quiver $\Gamma^\Bw $.  We write $ \fKLR_{\Bv}^{\varphi}\wgmod$ for the category of weakly gradable modules.

By the Coulomb branch construction, we can also think of $ \varphi $ as a central character of $ \CB(\Bv,\Bw) $.  We generalize the main result from \cite{KTWWY} to give a precise description of the category $\CB_\varphi(\Bv,\Bw)\dash\GTc_{\Z}$ of Gelfand-Tsetlin modules over $ \CB(\Bv,\Bw) $ with integral weights and integral central character $ \varphi $.  (In the main body of the paper, we do not restrict to integral weights, but we do so in the introduction to make the statements simpler.)

\begin{Theorem}[Corollary \ref{co:wKLR}] \label{th:CBfKLR}
	There is an equivalence of categories 
	$$  
	\CB_\varphi(\Bv,\Bw)\dash{\GTc}_{\Z}\  \cong  \fKLR_{\Bv}^{\varphi}\wgmod.
	$$ 
\end{Theorem}

Our answer to Question \ref{qu:intro} is given by the following theorem.
\begin{Theorem}[Theorem \ref{th:KLR-restriction} \& Theorem \ref{thm:divpowers}]\label{th:res-E_i}
Assume that $\Gamma$ has no edge loops.  
	\begin{enumerate}
		\item 
		If we choose $ \xi $ to be the first fundamental coweight of $GL(v_i) $, the restriction functor from Theorem \ref{th:IntroRes}, combined with a Hamiltonian reduction, gives a functor
		$$
		\CB_{\vp}(\Bv,\Bw)\operatorname{-}{\GTc}_\Z \xrightarrow{\res_i} \CB_{\vp}(\Bv - e_i,\Bw)\dash{\GTc}_\Z $$
		\item We have a commutative diagram:
		\begin{equation}\label{eq:intro-square}
			\begin{tikzcd}
				\CB_\varphi(\Bv,\Bw)\dash{\GTc}_\Z \arrow{r} {\res_i} \arrow{d} &   \CB_{\varphi}(\Bv - e_i,\Bw)\dash{\GTc}_\Z  \arrow{d}  \\
				\fKLR_{\Bv}^{\varphi}\wgmod \arrow{r}[below]{\EuScript E_i}   & \fKLR_{\Bv-e_i}^{\varphi}\wgmod
			\end{tikzcd}
		\end{equation}
		where the vertical equivalences come from Theorem \ref{th:CBfKLR}, and $ \EuScript E_i $ is a version of the functor (\ref{eq:EiKLR})
		\item The functors $\res_i$ preserve $\Oof{\CB_{\vp}(\Bv,\Bw)}_\Z$, and with their adjoints $\ind_i $ give functors as in (\ref{eq:EiY}) which define a categorical $\fg_\Gamma$-action on $\bigoplus_{\Bv}\Oof{\CB_{\vp}(\Bv,\Bw)}_\Z$.  
\end{enumerate} \end{Theorem}

\subsection{Cherednik algebras}
\label{sec:intro-cherednik}

Similar induction and restriction functors were defined by Bezru\-kav\-nikov and Etingof \cite[\S 3.5]{BEind}  in the context of category $\cO$ for rational Cherednik algebras.  It's natural to compare these with the restriction and induction functors we define, since in the case where  $G=GL(n)$ and
$N=\Hom(\C^n,\C^n)\oplus \Hom(\C^n, \C^\ell)$, Kodera and Nakajima \cite{KoNa} (see
also \cite{BEF, Webalt}) have shown that the Coulomb branch algebra $\CB(n,\ell)=\CB(G,N)$ is isomorphic to
\begin{itemize}
	\item the spherical trigonometric Cherednik algebra $\mathsf{H}({n,0})$ of $S_n$ if $\ell=0$, and
	\item the spherical rational Cherednik algebra $\mathsf{H}({n,\ell})$ of
	$G(\ell, 1, n)=S_n\wr \Z/\ell\Z$ otherwise.
\end{itemize}

In this paper, we will only consider the case $\ell>0$.  
Let $ \xi $ denote the first fundamental coweight of $GL(n)$ as in Theorem \ref{th:res-E_i}. Then Theorem \ref{thm:rel} yields a restriction functor $ \res : \Oof{\CB(n,\ell)}\to \Oof{\CB(n-1,\ell)}$, while Bezrukavnikov and Etingof have defined a similar functor between category $\cO$'s for the full Cherednik algebra $\mathsf{\tilde{H}}({n,\ell})$.  As shown in \cite[\S 3.5.2]{GoLo}, these preserve the subcategory of aspherical modules and thus induce functors $\mathsf{res}_{\mathsf{BE}}\colon \Oof{\mathsf{H}({n,\ell})}\to \Oof{\mathsf{H}({n-1,\ell})}$, thought of as quotient categories of category $\cO$ for $\mathsf{\tilde{H}}({n,\ell})$.

Note that unlike in Theorem \ref{th:CBfKLR}, we are not assuming any integrality of parameters, since the most interesting cases for the Cherednik algebra involve rational, but non-integral parameters.

\begin{Theorem}\label{th:Cherednik1}
	There is an equivalence $\Oof{\mathsf{H}({n,\ell})}\to \Oof{\CB(n,\ell)}$ for each $n$ which intertwines the functors $\res$ and $\res_{\mathsf{BE}}$.  
\end{Theorem}
The proof of this will be given in Section \ref{sec:cherednik-algebras}.
Given the discussion of the isomorphism $\mathsf{H}({n,\ell})\cong \CB(n,\ell)$ above, the reader might imagine that this is how the equivalence above is constructed.  That is not, in fact, the case.  Rather, the equivalence of Theorem \ref{th:Cherednik1} is constructed by comparing both categories to a flavoured KLRW algebra.  This equivalence is far from unique; it depends on several choices, and at the moment it is not clear how to tweak all the choices involved to assure that we obtain the equivalence induced by the isomorphism of \cite{KoNa}.  This is a question we will return to in future work.

\subsection{Generalized geometric Satake}

In \cite[Conj.~3.25]{BFNslices}, Braverman, Finkelberg and Nakajima propose a generalization of the geometric Satake theorem to all symmetric Kac-Moody types, which is further developed in the affine type A case in \cite{NakSatake}.  For each character $G\to \Cx$, there is an induced Hamiltonian $\Cx$ action on $M_C(\Bv,\Bw)$.  Choose this character to be given the product of the determinants, and let $\mathfrak{A}(\Bv,\Bw)\subset M_C(\Bv,\Bw)$ be the attracting locus for this $ \Cx$ action.  If $ \Gamma $ is of finite ADE type, then $ \mathfrak A(\Bv,\Bw) $ is a Mirkovi\'c-Vilonen locus in the affine Grassmannian of $ G_\Gamma^\vee $ by \cite[Lemma 4.4]{Krylov}.  
\begin{Conjecture}[\mbox{\cite[Conj. 3.25(3)]{BFNslices}}] \label{conj:nakajima}
	The sum of the top Borel-Moore homologies 
	$$\bigoplus_{\Bv} H^{BM}_{\operatorname{top}}(\mathfrak{A}(\Bv,\Bw))$$
	carries an action of $\mathfrak{g}_{\Gamma}$, making it isomorphic to the irreducible representation with highest weight $ \sum w_i \varpi_i$.  
\end{Conjecture}
In the finite-type case, this conjecture follows from the geometric Satake correspondence of Mirkovi\'c-Vilonen \cite{MV}. It also holds in affine type A by the results of \cite{NakSatake}.  In general, we expect that $\dim \mathfrak{A}(\Bv,\Bw)=d$ where we define $d= \tfrac{1}{2} \dim M_C(\Bv,\Bw)$; this is again known to hold in finite type by \cite[Thm.~3.2]{MV}, and affine type $A$ by \cite[Prop.~7.33]{bowvarieties}.

This conjecture was an important motivation for us, since there is a close relationship between the category $\cO$ of $\CB(\Bv,\Bw)$ and the top homology appearing here.  For any module $M$ in category $\cO$, there is an associated characteristic cycle class $\operatorname{CC}(M)\in H^{BM}_{2d}(\mathfrak{A}(\Bv,\Bw))$.   As in \cite[Prop. 6.13]{BLPWquant}, we can define this cycle by taking any good filtration of $M$, (compatible with the above mentioned filtration on $ \CB(G,N)$) and then summing the Borel-Moore classes of the $d$-dimensional components of $\gr M$, weighted by the generic rank of $\gr M$ on the component.  This is well-defined by a standard argument of Bernstein \cite[Lec. 2.8]{BernD}; see also \cite[Th. 1.1.13]{Ginzburg-D-modules}.  
Since we have fixed the degree of the Borel-Moore class here, this class will be insensitive to the multiplicity of our characteristic cycle on components of complex dimension $<d$. 

This induces a map $K_\C(\Oof{\CB_\varphi(\Bv,\Bw)})\to H^{BM}_{2d}(\mathfrak{A}(\Bv,\Bw))$.  This map is  not an isomorphism in most cases, since if the Gelfand-Kirillov dimension of $M$ is less than $d$, then $\operatorname{CC}(M)=0$ by Lemma \ref{lem:GKdim}.  The kernel of this map also depends in a sensitive way on the choice of $ \varphi $.

Assume that $ \varphi $ is integral and let $\cO_{\operatorname{top}}(\Bv,\Bw)$ be the quotient of $\Oof{\CB_\varphi(\Bv,\Bw)}$ by the subcategory of objects with GK dimension  $<d$. By Lemma \ref{lem:GKdim}, the characteristic cycle map descends to a map $ K_\C(\cO_{\operatorname{top}}(\Bv,\Bw)) \rightarrow  H^{BM}_{2d}(\mathfrak{A}(\Bv,\Bw))$.

On the other hand, by Corollary \ref{co:irrep}, the sum $\bigoplus_{\Bv}K_\C(\cO_{\operatorname{top}}(\Bv,\Bw))$ is an irreducible representation of $\mathfrak{g}_{\Gamma}$, with the action induced by the induction and restriction functors $\EuScript E_i, \EuScript F_i$. Note, there is no dependence on $\vp$ in this result, beyond requiring it to be integral.  This shows that unlike the rest of category $\mathcal{O}$, this quotient is not sensitive to $\vp$.  
Thus, Conjecture \ref{conj:nakajima} reduces to a proof that:
\begin{Conjecture}
	We have an equality $\dim \mathfrak{A}(\Bv,\Bw)=d$ and for any integral $\varphi$, the characteristic cycle map $\bigoplus_{\Bv}K_\C(\cO_{\operatorname{top}}(\Bv,\Bw))\to \bigoplus_{\Bv} H^{BM}_{2d}(\mathfrak{A}(\Bv,\Bw))$ is an isomorphism of vector spaces.
\end{Conjecture}
In finite-type ADE cases and in affine type A, the domain and codomain of this map both carry $\mathfrak{g}_{\Gamma}$-actions, but it is not clear that this map intertwines them. We know that the two sides are isomorphic as irreducible $\mathfrak{g}_{\Gamma}$-modules, so if the characteristic cycle map is equivariant, it must be an isomorphism.  On the other hand, outside of the finite-type and affine type A cases, we have no pre-existing action on the codomain, and thus we wish to use this conjecture to define one.

\newcommand{\redu}[1]{[#1]}
\newcommand{\rb}{\Upsilon}

\subsection*{Acknowledgements}
We thank Alexander Braverman and Justin Hilburn for helpful conversations.  J.K.  and B.W. were supported by NSERC through Discovery Grants.
O.Y. was supported in part by the ARC through grant DP180102563.  
This research was supported in part by Perimeter
Institute for Theoretical Physics.
Research at Perimeter Institute is supported in part by the Government of Canada through the Department of Innovation, Science and Economic Development Canada and by the Province of Ontario through the Ministry of Colleges and Universities.

\section{Relating various Coulomb branch algebras}

\subsection{Coulomb branch algebras and their partial flag versions}

Let $ G $ be a reductive group.  Let $ T\subset B$ denote a fixed maximal torus and Borel subgroup, and let $W$ denote the Weyl group.
We write  $\mathcal K = \C((z)), \mathcal O = \C[[z]] $ and we
consider the groups $ G(\cK) = G((z))$ and $G(\cO) = G[[z]] $.   We will
study the affine Grassmannian $ \Gr_G = G(\cK)/G(\cO)$.  Recall that for any dominant coweight $ \lambda $, the $ G(\cO) $-orbits $ \Gr^\lambda = G(\cO) z^\lambda $ partition $\Gr_G $.

Let $ P$ denote any  parabolic subgroup of $ G $ containing $T$, and $L$ its Levi subgroup.  Let $ W_L $ be the Weyl group of $ L$.
There is a corresponding parahoric subgroup $ I_P \subset G(\cO) $
which is defined as the preimage of $ P $ under the evaluation map $
G(\cO) \rightarrow G $.  In particular, if $ P = B $, then $ I_P $ is
the usual Iwahori; on the other hand, if $ P = G $, then $ I_P =
G(\cO) $.
We have the corresponding partial affine flag variety $
G(\cK) / I_P $.

We fix an extension
$$
1 \rightarrow G \rightarrow \tG \rightarrow F \rightarrow 1
$$
where $ F $ is a torus.  Following the physics literature, we call $ F $ the \textbf{flavour torus} of this extension.

Let $ \tilde T $ denote the maximal torus of $ \tG $ containing $ T $.  Similarly, for any parabolic subgroup $ P $ of $ G $, let $\tilde{P}$ be the
parabolic subgroup in $\tG$ such that $G\cap \tilde{P}=P$
and $\tilde{I}_P$ its preimage in $\tG(\cO)$.
Note that the Levi $\tilde{L}$ of $\tilde{P}$ is an
extension of $F$ by $L$.  We also have that $ L(\cO) \subset I_P $.

Throughout, we will regularly need to use the cohomology rings with complex coefficients of the classifying spaces of these groups, and so we write 
\[\HBG{G}=H^*(BG)=H_G^*(pt)=(\Sym\ft^*)^W, \]  
and similarly for $ L, F, \tG, T $, etc.

We define the group $\tG_\cK^\cO$ to be the preimage of $ F(\cO) $
under the map $ \tG(\cK) \rightarrow F(\cK)$.  This subgroup contains
$\tG(\cO)$, and the quotient $\tG_\cK^\cO/\tG(\cO)$ is isomorphic to
the affine Grassmannian $\Gr_G$.  We have an action of $ \Cx $ on $
\tG_\cK^\cO $ by loop rotation, given by $(t\cdot g)(z) = g(tz)$ for $t \in \Cx$ and $g(z) \in \tG_\cK^\cO$ . There is an action of the semi-direct
product $  \tG_\cK^\cO \rtimes \Cx $ on $\Gr_G$.

Let $ N $ be a representation of $ \tG $.
We will typically construct
$\tG$ by starting with a representation of $G$, and letting $\tG$ be
the product of $G$ and a torus which acts on $N$ commuting with $G$; for convenience later, we will not assume this action is faithful.

Following Braverman-Finkelberg-Nakajima \cite{BFN}, we will now define the spaces used to construct the Coulomb branch.
Let $ N(\cO) = N \otimes \C[[z]] $ and $ N(\cK) = N \otimes \C((z))$;
these are naturally representations over $\tG(\cO) \rtimes \Cx$ and $
\tG_\cK^\cO \rtimes \Cx$, respectively.  We can consider the vector bundle
$$
T_{G,N} := G(\cK) \times_{G(\cO)} N(\cO) =  \tG_\cK^\cO \times_{\tG(\cO)} N(\cO) \rightarrow Gr_G
$$
Note that the subscript $G(\cO)$ (resp.~$\tG(\cO)$) indicates the quotient by a natural group action. There is a projection map $ p: T_{G,N} \rightarrow N(\cK) $ given by $ [g,v] \mapsto gv $.

Let
$$
R_{G,N} = p^{-1}(N(\cO)) = \{([g], w) \in Gr_G \times N(\cO) :  w \in g N(\cO) \}
$$
Following Braverman-Finkelberg-Nakajima \cite{BFN}, we  form the convolution algebra
$$ \CBh(G,N) :=   H_*^{\tG(\cO) \rtimes \Cx}(R_{G,N}) $$
We will use $\hbar$ throughout for the equivariant parameter corresponding to the loop $\Cx$-action.  We include it in the notation here to emphasize that we have left it as a free variable, whereas later, we will usually consider the quotient where we set $\hbar=1$.  A great deal of care is required to define the equivariant Borel-Moore homology for an infinite-dimensional space and
group; see \cite[\S 2(ii)]{BFN} for a detailed discussion.
We will refer to $\CBh(G,N)$ as the \textbf{spherical Coulomb branch algebra}.  

\begin{Remark}
	If we specialize $\hbar$ to 0, then the resulting algebra is
	commutative, and is the coordinate ring on a Poisson variety $
	M_C(G,N)$.  The inclusion $ H^*_F(pt) \rightarrow \CBh(G,N) $ gives rise a morphism $ M_C(G,N) \rightarrow \mathfrak f$.  The fibre over $0\in \mathfrak{f}$ is usually called the {\bf Coulomb branch}, and this family provides a Poisson deformation.
\end{Remark}

We will also need the
partial flag variety version of the BFN algebra.
\begin{Definition}
	Let $ P $ be a parabolic subgroup of $ G$.   We define
	\begin{align*}      T_{G,N}^P& := G(\cK) \times_{I_P} N(\cO) =  \tG_\cK^\cO \times_{\tilde{I}_P} N(\cO)\xrightarrow{p_P} N(\cK)\\
		R_{G,N}^P &:= p_P^{-1}(N(\cO)) = \{([g], w) \in G(\cK)/ I_P \times N(\cO) :  w \in g N(\cO) \}
	\end{align*}
	The {\bf parabolic Coulomb branch algebra} is the convolution
	algebra
	\begin{equation*}
		\PCBh(G,N) := H_*^{\tilde{I}_P \rtimes \Cx}(R_{G,N}^P).
	\end{equation*}
	
\end{Definition}

These algebras depend on the choice of $\tG$, but
we leave this implicit in the notation.  In the special case where $ P = B$, then we will call $ \ICBh $ the
\textbf{Iwahori Coulomb branch algebra}.  The spherical Coulomb
branch algebra is an example of a {\bf
	principal Galois order} as defined e.g. in \cite{FGRZGalois}.  The idea of replacing principal Galois
orders by flag orders (defined in \cite[Lem. 2.5]{WebGT}) played an important role in
\cite{WebGT,Webdual};   the algebra $\ICBh$ is an example of a flag
order in this case.

Note that the space $  R_{G,N}^P$ contains a copy of $G(\cO)/I_P\times
N(\cO)\cong G/P\times N(\cO)$.  We have vector space isomorphisms
\begin{equation}
    \label{eq:nilHecke}H_*^{{\tilde{I}_P \times
		\Cx}}(G/P\times N(\cO))\cong H_*^{\tP \times \Cx}(G/P)\cong H_*^{\tG
	\times \Cx}((G/P)^2).
\end{equation}
(Here we use $ G/P = \tG / \tilde P $ to get the action of $ \tG $ on $G/P $.)

Now, $  H_*^{\tG \times \Cx}((G/P)^2) $ has a convolution structure of its own, and Poincar\'e duality shows that it is a matrix algebra:
\[H^{\tG \times \Cx}_*((G/P)^2)\cong \End_{\HBG{\tG \times \Cx}}(H_*^{\tG \times \Cx
}(G/P))\cong \End_{\HBG{\tG \times \Cx}}(\HBG{\tL \times \Cx}).\]
When $ P = B $, then this is the nilHecke algebra of $ W$.
\begin{Lemma}
	The inclusion $ (G(\cO)/I_P)\times N(\cO) \hookrightarrow R_{G,N}^P $ induces an algebra map $$H_*^{\tG \rtimes \Cx }((G/P)^2)\to \PCBh(G,N).$$
\end{Lemma}

Let $e' \in H_*^{\tG \rtimes \Cx }((G/P)^2)=H_*^{\tP \rtimes \Cx}(G/P)$ be the primitive
idempotent in this matrix algebra that projects to the
$W$-invariants.  We can formulate the usual abelianization isomorphism
as the statement  that for any $\tG$-space $X$,
we have $e'H_*^{\tP}(X)\cong H_*^{\tG}(X)$ where $H_*^{\tP \rtimes
	\Cx}(G/P)$ acts by convolution.

Applying this fact, we find that:
\begin{Proposition} \label{th:Morita}
	For any parabolic $ P \subset G$, we have isomorphisms
	\begin{equation}\label{eq:Morita}
	   \PCBh e' \cong H_*^{\tilde{I}_P \rtimes \Cx}(R_{G,N}) \qquad
	e'\PCBh \cong H_*^{\tG(\cO) \rtimes \Cx}(R_{G,N}^P)\qquad \CBh \cong e'\PCBh e' 
	\end{equation}
	The bimodules $ \PCBh e'  $ and $ e'\PCBh $ define a Morita equivalence between $ \PCBh $ and $\CBh$.
\end{Proposition}
The first two isomorphisms of \eqref{eq:Morita} are related by the map from $\tilde{I}_P$-orbits on $R_{G,N}$ to $\tG(\cO)$ orbits on $R_{G,N}^P$ sending $(\redu{g},v)\mapsto(\redu{g^{-1}},gv)$.  
\begin{proof}
	This follows the same proof as \cite[Th. 2.6]{WeekesCoulomb}.
\end{proof}

The algebra $\CBh(G,N)$ contains a subalgebra given by the equivariant
parameters 
$$H_{\tilde G(\cO) \rtimes \Cx}^\ast(pt) \cong \HBG{\tG \times \Cx}.$$
Braverman, Finkelberg, and Nakajima \cite[Section 3(vi)]{BFN} call this the Cartan
subalgebra, but we prefer to call this the {\bf Gelfand-Tsetlin subalgebra}, since it is a generalization of this subalgebra
in $U(\mathfrak{gl}_n)$.  Similarly $ \CBh^P(G,N) $ contains $H^\ast_{\tilde{I}_P \rtimes \Cx}(pt)\cong \HBG{\tL \times \Cx}$.

The Gelfand-Tsetlin algebra $\HBG{\tG\times \Cx}$ of $ \CBh(G,N) $ contains the subalgebra
$Z:=\HBG{F\times \Cx}\cong \Sym \mathfrak{f}^*[\hbar]$, which is central in $\CBh(G,N)$ (here
$ \mathfrak f $ is the Lie algebra of our flavour torus $ F $).  In
fact, an application of \cite[Th. 4.1(4)]{FOgalois}, using the fact
that $\CBh(G,N)$ is a Galois order as shown in \cite[Th. B]{WebGT},
shows that $ Z $ is the full center of $\CBh(G,N)$.

\subsection{Inclusion of Coulomb branch algebras}
One advantage of considering these parabolic Coulomb branch algebras is
that they allow us to study the relation with Coulomb branch algebras
defined by Levi subgroups.

The inclusion $ L(\cK) \subset G(\cK)$ gives rise to an inclusion $ \Gr_L \hookrightarrow G(\cK) / I_P $.  Moreover it is easy to see that the restriction of $R^P_{G,N} $ to $ \Gr_L $ coincides with $ R_{L,N}$.  This leads to the following result.

\begin{Proposition} \label{th:inclusion}
	There is an inclusion of algebras $ \CBh(L, N) \rightarrow \PCBh(G,N)
	$, which respects their Gelfand-Tsetlin subalgebras $\HBG{\tL \times \Cx}$.
\end{Proposition}
\begin{proof}
	The inclusion $ R_{L,N} \subset R^P_{G,N} $ gives an inclusion
	
	\begin{equation}
		\iota \colon H_*^{L(\cO)\rtimes \Cx}(R_{L,N}) \hookrightarrow
		H_*^{L(\cO)\rtimes \Cx}(R_{G,N}^P).\label{eq:L-inc}
	\end{equation}
	To see the compatibility of this inclusion with multiplication, we use
	an argument similar to that in \cite[5(ii)]{BFN}. Consider
	the analogue of the diagram \cite[(3.2)]{BFN}:
	\begin{equation*}
		\tikz[->,thick]{
			\matrix[row sep=12mm,column sep=12mm,ampersand replacement=\&]{
				\node (a) {$ R_{G,N}^P \times R_{L,N}  $}; \& \node (b)
				{$p^{-1}(R_{G,N}^P \times R_{L,N}  )$}; \&\node (d) {$q(p^{-1}(R_{G,N}^P \times R_{L,N} ))$}; \& \node (e)
				{$R_{G,N}^P $}; \\
			};
			\draw (b) -- (a) node[above,midway]{$p$};
			\draw (b) -- (d) node[above,midway]{$q$} ;
			\draw (d) -- (e) node[above,midway]{$m$};
		}
	\end{equation*}
	where $ p : R^P_{G,N} \times L(\cK) \rightarrow  R^P_{G,N} \times T_{L,N} $ is given by $ ([g_2], w, g_1) \mapsto ([g_2], w, [g_1, w])$.
	This diagram defines a right action of $\CBh(L,N)$ on $\PCBh(G,N)$.
	In fact, this diagram is a special case of the auxilliary action diagram (27) from \cite{BFNSpringer}, where $ Z = R^P_{G,N} $ and $ G = L$, except with all factors reversed.  Thus from \cite[Proposition 4.13]{BFNSpringer}, we can deduce that it defines a
	right module structure of $ \CBh(L,N)$ on $\PCBh(G,N)$ commuting with left multiplication.
	
	As in \cite[5.7(1)]{BFN}, we can see that $1\star b =\iota(b)$.
	More generally, we must have $a\star b=a\iota(b)$.  This shows that
	for any $b,b\in \CBh(L,N)$, we have:
	\[\iota(bb')=1\star(bb')=(1\star b)\star b'=\iota(b)\star b'=\iota(b)\iota(b').\]
	Thus, \eqref{eq:L-inc} is an algebra map.
	
	The inclusion $ L(\cO) \subset I_P $ leads to an isomorphism
	$$
	H_*^{L(\cO)\rtimes \Cx}(R_{G,N}^P) \cong H_*^{I_P \rtimes \Cx}(R_{G,N}^P).
	$$ Composing this with \eqref{eq:L-inc} gives the
	desired injection.
\end{proof}

\subsection{Abelian theories and monopole operators}
\label{section: abelian theories and monopoles}
Let $ \nu: \C^\times \rightarrow T $ be a central coweight. Since $ \nu $ is central, $ z^\nu \in \Gr_G$ is a $ G(\cO)$ orbit.  We let
$r_{\nu}\in \CBh(G,N)$ be the homology class of the preimage of $z^\nu $ in $R_{G,N}$.  These elements $r_\nu$ are called \textbf{monopole operators}.

When $G$ is abelian all coweights are central, and the elements $r_\nu$ form a basis for $\CBh(G, N)$ as a left (or right) module over $\HBG{\tG \times \Cx}$, where $\nu$ runs over all coweights of $G$.  Their relations are known explicitly by \cite[Section 4(iii)]{BFN}:
\begin{equation}
	\label{eq: monopole relations}
	r_\xi r_\nu = \prod_{\langle\mu,\xi\rangle > 0 > \langle \mu,\nu\rangle} \prod_{j=1}^{d(\langle \mu, \xi\rangle, \langle \mu,\nu\rangle)} \big(\mu + (\langle\mu, \xi\rangle - j) \hbar \big ) \prod_{\langle\mu,\xi\rangle < 0 < \langle \mu,\nu\rangle} \prod_{j=0}^{d(\langle \mu, \xi\rangle, \langle \mu,\nu\rangle)-1} \big(\mu + (\langle\mu, \xi\rangle + j) \hbar \big ) r_{\xi+\nu},
\end{equation}
Here the first and third products range over weights $\mu$ of $N$,
with multiplicity.  These are weights for the action of  $ \tG$, and the products above lie in the Gelfand-Tsetlin subalgebra $ \HBG{\tG\times \Cx}$.  Also $d: \Z \times \Z \rightarrow \Z_{\geq 0}$ is the function defined by
$$
d(a,b) = \left\{ \begin{array}{cl}  0, & \text{if } a,b \text{ have the same sign,} \\ \min\{|a|, |b|\}, &  \text{if } a,b \text{ have different signs.} \end{array} \right.
$$
It will also be useful to have an inverted version of these formulas:

\begin{equation}\label{eq: inverse monopole relations}
	r_\xi ^{-1}r_{\nu}=\prod_{\langle\mu,\xi\rangle > 0 > \langle \mu,\nu-\xi\rangle} \prod_{j=1}^{d(\langle \mu, \xi\rangle, \langle \mu,\nu-\xi\rangle)} \frac{1}{\mu - j \hbar} \prod_{\langle\mu,\xi\rangle < 0 < \langle \mu,\nu-\xi\rangle} \prod_{j=0}^{d(\langle \mu, \xi\rangle, \langle \mu,\nu-\xi \rangle)-1} \frac{1}{\mu + j \hbar }  r_{\nu-\xi} ,	
\end{equation}
This version follows by rearranging \eqref{eq: monopole relations} for the product $r_\xi r_{\nu-\xi}$, using in addition the fact that for any  weight $\mu\in \HBG{\tG \times \Cx}$ we have
$$
r_\xi \mu = (\mu + \langle \mu,\xi \rangle \hbar) r_\xi
$$
by \cite[(4.8)]{BFN}. Note that this implies $r_\xi^{-1} \mu = (\mu - \langle \mu,\xi\rangle \hbar) r_\xi^{-1}$.

\begin{Remark}\label{rem:BFNhalf}
	Note that equation (\ref{eq: monopole relations}) differs slightly  from \cite{BFN}, as we do not follow the convention from \cite[Section 2(i)]{BFN} of shifting the weights of the loop rotation action on $N(\cK)$ by $1/2$.
\end{Remark}

We now recall some algebra homomorphisms defined in \cite{BFN}, which
will play an important role in this paper.

For general $G$ there is an abelianization map $(\iota_\ast)^{-1} : \CBh(G,N) \hookrightarrow \CBh(T,N)_{loc}$ described in  \cite[Remark 5.23]{BFN}.  Here $\CBh(T,N)_{loc}$ denotes the localization of $\CBh(T,N)$ at the multiplicative set generated by  $\hbar, \alpha + m\hbar$, where $\alpha$ runs over roots of $G$ and $m\in \Z$.

Next, recall from \cite[Remark 5.14]{BFN} that for any two representations $N_1, N_2$ there is a natural injective map $ \CBh(G,N_1 \oplus N_2) \hookrightarrow \CBh(G, N_2)$. For $ G $ abelian, this map is given by \cite[Section 4(vi)]{BFN}:
\begin{equation}
	r_\nu \mapsto \prod_{\langle \mu, \nu\rangle < 0} \prod_{j=\langle \mu,\nu\rangle}^{-1} (\mu + j \hbar) \ r_\nu
\end{equation}
Here the product is over the weights of $ N_1 $.  Also, since it is
potentially confusing, we note that we have written $r_\nu$ for the
monopole operators in $\CBh(G, N_1\oplus N_2)$ and $\CBh(G,N_2)$,
respectively.

Finally, assume there is a coweight $\wp\colon \C^*\hookrightarrow
\tilde{T}$ which acts on $N_2$ by scalar multiplication of weight 1 and on $N_1$ with weight $0$; this is always possible if we
extend the flavour torus $F$. 

By \cite[6(viii)]{BFN}, there  is a ``Fourier transform'' isomorphism
\begin{equation} \label{eq:Fourier0}
	\CBh(G, N_1 \oplus N_2) \stackrel{\sim}{\longrightarrow} \CBh(G, N_1 \oplus N_2^\ast),
\end{equation}
which is the identity on the Gelfand-Tsetlin subalgebra $\HBG{\tG\times \Cx}$.
In the abelian case, this isomorphism is defined by the map \begin{equation}
	r_{\nu}\mapsto (-1)^{\delta(\nu)}r_{\nu}\qquad
	\mu \mapsto \mu+\hbar \langle \mu, \wp\rangle. \label{eq:Fourier}
\end{equation} 
where $\delta(\nu)=\sum_{\langle
	\mu,\nu\rangle>0}\langle \mu,\nu\rangle$.  In this sum $\mu$ ranges over
weights of $N_2$.  (If we use the conventions of \cite{BFN} (as discussed in Remark \ref{rem:BFNhalf}), then this shift by $\wp$ is unnecessary.)

For general $G$, the isomorphism is defined using the abelian case, via the abelianization map \cite[Lemmas 5.9--5.10]{BFN}.

\subsection{Passing to invariants}

For $ \lambda $ a dominant coweight of $ G$, let $ R_{G,N}^\lambda $ denote the preimage of the $G(\cO)$-orbit closure $ \overline{\Gr^\lambda} $ under the map $ R_{G,N} \rightarrow \Gr_G $.
We write $ \CBh(G, N)^\lambda $ for the subspace of the Coulomb branch algebra coming from the homology of $ R_{G,N}^\lambda $.  We can restrict the injective algebra map $ \CBh(G,N_1 \oplus N_2) \hookrightarrow \CBh(G, N_2)$ to an injective linear map $ \CBh(G,N_1 \oplus N_2)^\lambda \hookrightarrow \CBh(G, N_1)^\lambda$.

\begin{Lemma}
	\label{lem: fibre condition}
	Let $ \lambda $ be a dominant coweight for $ G $. Suppose that for all $ [g] \in \overline{\Gr^\lambda} $, we have $ g (N_2 \otimes \cO) \subseteq N_2 \otimes \cO $. Then we have $ \CBh(G,N_1 \oplus N_2)^\lambda = \CBh(G, N_1)^\lambda$.
\end{Lemma}

\begin{proof}
	
	The map $\CBh(G,N_1 \oplus N_2)^\lambda \hookrightarrow \CBh(G, N_1)^\lambda$ comes from pullback along the map $$ R_{G,N_1 \oplus N_2}^\lambda \rightarrow R_{G,N_1}^\lambda. $$  However, by the hypothesis, this map is a vector bundle (with fibre over $([g],w) \in R_{G,N_1}^\lambda$ being $g(N_2 \otimes \cO) $).
\end{proof}

\begin{Definition}\label{def:Npm}
	For any coweight $\xi \colon \Cx\to T$, we let $N_0^{\xi},
	N_\pm^{\xi}, N_{\geq}^{\xi}, N_{\leq}^{\xi}$ be the sum of weight
	spaces where the weight of $\xi$ is zero, positive, negative,
	non-negative or non-positive, respectively.
	
	Note that we always have
	\[N = N^{\xi}_- \oplus N^{\xi}_0 \oplus N^{\xi}_+= N^{\xi}_{\leq}
	\oplus N^{\xi}_+= N^{\xi}_- \oplus N^{\xi}_{\geq}.\]
\end{Definition}

We will assume that the coweight $\xi$ is central.  Recall the monopole operator $r_\xi$ from Section \ref{section: abelian theories and monopoles}.

\begin{Theorem} \label{th:invertrxi}
	Assume that $ N^{\xi}_- = 0 $.
	The natural map $ \CBh(G, N) \hookrightarrow \CBh(G, N^{\xi}_0) $ gives an isomorphism
	$$ \CBh(G, N)[r_\xi^{-1}] \cong \CBh(G, N^{\xi}_0)$$
\end{Theorem} 

For the proof, we appeal to the following basic fact about Ore localizations:

\begin{Lemma}
	Let $S$ be a multiplicative set in a domain $A$.  Suppose that $\varphi: A \hookrightarrow B$ is an injective homomorphism into a ring $B$, such that:
	\begin{enumerate}
		\item[(i)] for all $s \in S$, the image $\varphi(s)$ is invertible in $B$,
		
		\item[(ii)] for all $b \in B$, there exist $s\in S$ and $a \in A$ so that $b = \varphi(s)^{-1} \varphi(a)$.
	\end{enumerate}
	Then $S$ satisfies the left Ore condition in $A$, and $\varphi$ gives an isomorphism $S^{-1} A \cong B$.
\end{Lemma}

\begin{proof}[Proof of Theorem \ref{th:invertrxi}]
	We will verify the stipulations of the lemma. 	
	Since $ \C^\times_\xi  $ acts trivially on $ N^{\xi}_0$,  the
	element $ r_\xi $
	is  invertible in $ \CBh(G, N^{\xi}_0) $ (with inverse $ r_{-\xi}$).
	Thus, it remains to check condition (ii) above. This will be implied by the following statement:
	\begin{enumerate}
		\item[$(\dagger)$] 
		\label{claim}	For all dominant $ \lambda $, there exists $ m \in \Z_{\geq 0} $
		such that  $r_\xi^m \CBh(G, N^{\xi}_0)^\lambda \subset \CBh(G,
		N)^{\lambda + m \xi}.$
	\end{enumerate} 		
	Given a dominant coweight $ \lambda$, let $ m = \max \langle \mu, \lambda \rangle $ where $ \mu $ ranges over all weights of the representation $ N^{\xi}_+ $.  A standard reasoning shows that for all $ [g] \in \overline{\Gr^\lambda} $, we have $ g (N^{\xi}_+ \otimes \cO) \subset z^{-m} (N^{\xi}_+ \otimes \cO)$.  Thus $ z^{m\xi} g (N^{\xi}_+ \otimes \cO) \subset N^{\xi}_+ \otimes \cO $ (as the weights of $ \Cx_\xi $ on $N^{\xi}_+ $ are positive).  So by  Lemma \ref{lem: fibre condition}, $  \CBh(G,N)^{\lambda + m\xi} = \CBh(G, N^{\xi}_0)^{\lambda + m\xi}$.
	
	Now we observe that $ r_\xi^m \CBh(G, N^{\xi}_0)^\lambda = \CBh(G, N^{\xi}_0)^{\lambda + m \xi}$ and thus we have established $(\dagger)$.
\end{proof}

\subsection{Hamiltonian reduction}\label{sec:ham-red}
Note that the operations we've discussed thus far only relate groups
with the same rank, so the dimension of the Coulomb branch will be unchanged.  When we change the matter from $N$ to $N^{\xi}_0$,  the action of the central subgroup $ \Cx_\xi $ becomes trivial.  So we can consider $N^\xi_0 $ as a representation of the quotient group $ G/\Cx_\xi$.  This invites us to consider the relationship between
$\CBh(G, N^{\xi}_0)$ and $\CBh(G/\Cx_\xi, N^{\xi}_0)$, which decreases
the dimension of the Coulomb branch by 2.

Let $A$ be an algebra, and $b\in A$. Recall that the \textbf{quantum Hamiltonian reduction} of $ A$ by $b$ (at level 1) is defined in two stages.  First, we form the right $A$--module $ A / (b -1) A$.  Then we construct the quantum Hamiltonian reduction by considering
$$ A \sslash_1 b :=  \operatorname{End}_A\big(A / (b-1)A\big) \cong \big\{ [a] \in A / (b-1)A \ \mid \  a(b-1) \in  (b-1)A \big\}
$$
Thought of as an endomorphism ring, $A \sslash_1 b$ has a natural algebra structure, which may equivalently be defined on equivalence classes by $[a_1]\cdot [a_2] = [a_1 a_2]$.  Geometrically, this algebra is the
quantization of the operation of passing to the level set $b=1$, and then dividing by the flow of the Hamiltonian vector field associated to $ b $.

\begin{Theorem} \label{th:Reduction}
	Assume that $ \xi : \C^\times \rightarrow G $ is central and primitive (i.e.~it is not an integer multiple of any other cocharacter), and that $ \Cx_\xi $
	acts trivially on $ N $. Then we have
	$$
	\CBh(G, N) \sslash_1 r_\xi \cong \CBh(G / \Cx_\xi, N)
	$$
\end{Theorem}

\begin{proof}
	For the purposes of this proof, let us write $ G' = G/\Cx_\xi $.	Consider $ R_{G', N}$.  This space has an action of $ G(\cO) $ which factors through the map $ G(\cO) \rightarrow G'(\cO)$.  On the one hand,
	$$
	H_*^{G(\cO) \rtimes \Cx}(R_{G'}, N) \cong \CBh(G, N) / (r_\xi - 1) \CBh(G,N) $$
	as a module over $ \CBh(G, N) $.  To see this, we note that $ R_{G', N} = R_{G,N} / \Z $ where the generator of $ \Z $ acts by translation by the element $ z^\xi $, since $\xi$ is primitive.
	
	On the other hand,
	$$H_*^{G(\cO) \rtimes \Cx}(R_{G', N}) \cong  H_*^{G'(\cO) \rtimes \Cx}(R_{G'}, N) \otimes_{\HBG{G'}}\HBG{G}
	$$
	as a $ \CBh(G', N) $ module.	
	
	Therefore we have that
	$$
	\CBh(G, N) /(r_\xi - 1)  \CBh(G,N) \cong \CBh(G', N) \otimes_{\HBG{G'}} \HBG{G}.
	$$
	Note that $\HBG{G}\cong \HBG{G'}[a]$, where $a$ is a character of the Lie algebra $\mathfrak{g}$ splitting the derivative of $\xi$.  The action of $r_\xi$ on the RHS above commutes with $\HBG{G'}$, and satisfies  $[r_\xi,a]=r_{\xi}$.  This shows that the kernel of $r_{\xi}-1$ is $\CBh(G', N)$.
\end{proof}
\begin{Remark}
	The hypothesis of primitivity is needed here.  For example, if we reduce $\CBh(\C^{\times},0)$ by the square $r_2=r_1^2$, then the result will be $\C[r_1]/(r_1^2-1)\ncong \C$.  \end{Remark}

\begin{Remark}
	One can also consider the ``right'' quantum Hamiltonian reduction
	$$
	\operatorname{End}_A\big( A/ A(b-1)\big)^{op} \ \cong \ \big\{ [a] \in A / A (b-1) \ \mid \  (b-1)a \in  A(b-1) \big\}
	$$
	With the same assumptions as in the previous theorem, the right Hamiltonian reduction of $\CBh(G,N)$ by $r_\xi$ is also isomorphic to $\CBh(G / \Cx_\xi, N)$. 
\end{Remark}

\subsection{Combining all the steps}
Now, we will see how to combine the above results.

Let $ G $ be a reductive group, $ N $ a representation, $ \xi : \Cx \rightarrow T $ any coweight.
Let $ P, L $ be the parabolic and Levi subgroups corresponding to $ \xi $ and let $ N^{\xi}_0 $ be the invariants for the action of $ \Cx_\xi $ on $ N $.

\begin{Theorem}\label{thm:rel}
	The algebras
	$$ \CBh(G,N),\ \PCBh(G,N),\ \CBh(L,N),\ \CBh(L, N^{\xi}_0), \text{ and }\ \CBh(L/\Cx_\xi, N^{\xi}_0) $$
	are related as follows.
	\begin{enumerate}
		\item There is a Morita equivalence between
		$$
		\CBh(G,N) \text{ and } \PCBh(G,N).
		$$
		\item There is an inclusion of algebras
		$$ \CBh(L,N) \hookrightarrow \PCBh(G,N).
		$$
		\item There is an isomorphism
		$$
		\CBh(L,N)[r_\xi^{-1}] \cong \CBh(L, N^{\xi}_0).
		$$
		\item If $\xi$ is primitive, there is an isomorphism
		$$
		\CBh(L,N^{\xi}_0) \sslash_1 r_\xi \cong \CBh(L/\Cx_\xi, N^{\xi}_0).
		$$
	\end{enumerate}
	All these maps are compatible with the maps between the Gelfand-Tsetlin subalgebras
	$$
	\HBG{\tG} \hookrightarrow\HBG{\tL}\hookleftarrow \HBG{\tL/\Cx_{\xi}}
	$$
\end{Theorem}

\begin{proof}\hfill
	\begin{enumerate}
		\item This follows immediately from Proposition \ref{th:Morita}.
		\item This follows immediately from Proposition \ref{th:inclusion}.
		\item 
		Note that because $ \xi $ is central in $ L$,  all
		the subspaces from Definition \ref{def:Npm} are invariant subspaces for the action of $L $.  By  \eqref{eq:Fourier0}, we have an isomorphism
		\begin{equation}\label{eq:Fourier2}
			\CBh(L, N) \cong \CBh(L, (N^{\xi}_-)^* \oplus N^{\xi}_{\geq})	
		\end{equation}

		Since $ {N^{\xi}_-}^* \oplus N^{\xi}_{\geq} $ has non-negative weight vectors for the action of $ \Cx_\xi $, we can apply Theorem \ref{th:invertrxi} to give an isomorphism
		$$
		\CBh(L, {N^{\xi}_-}^* \oplus N^{\xi}_{\geq})[r_\xi^{-1}] \cong \CBh(L, N^{\xi}_0)
		$$
		Finally, we recall that map \eqref{eq:Fourier2} is not
		the identity on the integrable system $\HBG{\tL}$;
		it involves a shift by a  cocharacter $\wp$.  However,
		this cocharacter acts trivially on $N^{\xi}_0$, and so
		$\CBh(L, N^{\xi}_0)$ carries an automorphism shifting
		the integrable system $\HBG{\tL}$ by $-\hbar \wp$, and
		leaving all monopole operators unchanged.
		Thus, if consider the composition of the map \eqref{eq:Fourier2} with the map of Theorem \ref{th:invertrxi}, and then finally the automorphism discussed above, the composition will give the desired map.  
		\item This follows immediately from Theorem \ref{th:Reduction}.\qedhere
	\end{enumerate}
\end{proof}

\begin{Remark}
	If we swap $\xi$ and $-\xi$, then the
	algebras appearing in Theorem \ref{thm:rel} are unchanged. However the subspaces $ N^{\xi}_-, N^{\xi}_+ $ are swapped, and so we will use a different map $\CBh(L,N)\to \CBh(L,N^{\xi}_0)$, which will induce a different isomorphism
	$\CBh(L,N)[r_{-\xi}^{-1}]\cong \CBh(L,N^{\xi}_0)$.
\end{Remark}
\begin{Remark}
	The assumption that $\xi$ is primitive appears only in part (4) of the theorem.  In fact, we may safely make this assumption throughout: for any $\xi$ and any integer $k\geq 1$, the coweights $\xi$ and $k \xi$ determine the same subgroups $P,L$, and determine the same subspaces $N_0^{\xi} = N_0^{k \xi}$ and $N_\pm^{\xi} = N_\pm^{k \xi}$.   Moreover $r_{k \xi} = r_\xi^{\phantom{.}k}$ by (\ref{eq: monopole relations}), so $\CBh(L,N)[r_{k\xi}^{-1}] = \CBh(L,N)[r_\xi^{-1}]$.  It follows that parts (1)--(3) of the theorem are identical for both $\xi$ and $k\xi$, and for this reason we could assume that $\xi$ is primitive throughout.
\end{Remark}

\section{Quiver gauge theories}
In this section we focus on quiver gauge theories, and identify in important special cases the algebras which appear in Theorem \ref{thm:rel}.

\label{sec:quiver}
\subsection{Quiver data} \label{se:quiverdata}
Let
$\Gamma$ be a quiver with vertex set $\vertex$ and edge set $\edge(\Gamma)$, and dimension vectors
$\mathbf{v},\mathbf{w}\colon \vertex\to \Z_{\geq 0}$ (we allow loops and multiple edges).  For an edge $ e \in \edge(\Gamma)$, we write $ h(e) $ for its head and $ t(e) $ for its tail.

Consider the group and representation
\[N=\bigoplus_{e \in \edge(\Gamma)}\Hom(\C^{v_{t(e)}},\C^{v_{h(e)}})\oplus
\bigoplus_{i\in I}\Hom(\C^{v_i},\C^{w_i})\qquad G=\prod_{i\in I} GL(v_i).\]
We consider also the larger group
$$\tG :=  G \times \prod_{i\in I} (\Cx)^{w_i} \times (\Cx)^{\edge(\Gamma)}$$
where the middle factor is product of the diagonal matrices inside each $ GL(w_i) $.
We have a (non-faithful) action of $\tG$
on $N$.  We'll use $\CBh(\Bv,\Bw) := \CBh(G, N)$ to denote the Coulomb branch
algebra attached to these dimension vectors.

This is slightly cleaner if we use the ``Crawley-Boevey trick'' of adding a new vertex
$\infty$ and $w_i$ new edges from $i$ to $\infty$ to make a larger quiver $\Gamma^{\Bw}$.  We extend $ \Bv $ to this new vertex by setting $ v_\infty = 1 $.  Then, we can think of an element
of $N$ as a representation of $\Gamma^{\Bw}$, using $ \Hom(\C^{v_i},
\C^{w_i}) = \Hom(\C^{v_i}, \C^{v_\infty})^{\oplus w_i} $, so 
$ N = \bigoplus_{e \in \edge(\Gamma^\Bw)} \Hom(\C^{v_{t(e)}},\C^{v_{h(e)}}) $.
Also, from this perspective, $\tG$ is obtained from
$G$ by  adding a
scaling along each edge of $\Gamma^{\Bw}$.

\subsection{Relating algebras} \label{se:Quiverxi}
As in the general case, we want to consider a coweight $\xi \colon
\Cx\to T$.  This is the same as choosing a $ \Z$-grading on the
vector spaces $\C^{v_i}$; to correct some later sign issues, we let the degree $k$ elements be those with weight $-k$.  Thus, for each $p\in \Z$, we have a
dimension vector $v_i^{(p)}$ such that $v_i=\sum_{p\in \Z}v_i^{(p)}$.

In this case, $L$ is the set of grading
preserving automorphisms of these vector spaces,  so $L=\prod_{p\in \Z} L^{(p)}$
where $L^{(p)}=\prod_{i\in \vertex}GL(v_i^{(p)})$.
The subspace $N^{\xi}_0$ is just the grading
preserving quiver representations, where $\C^{w_i}$ is given degree 0. That is, $N^{\xi}_0=\bigoplus_{p \in \Z} N^{(p)}$ where \[N^{(p)}=
\begin{cases}
	\displaystyle
	\bigoplus_{e \in \edge(\Gamma)}\Hom(\C^{v_{t(e)}^{(p)}},\C^{v_{h(e)}^{(p)}}) & p\neq 0\\
	\displaystyle\bigoplus_{e \in \edge(\Gamma)}\Hom(\C^{v_{t(e)}^{(p)}},\C^{v_{h(e)}^{(p)}})\oplus
	\bigoplus_{i}\Hom(\C^{v_i^{(p)}},\C^{w_i}) & p=0.
\end{cases}\]
By \cite[(3(vii).(a))]{BFN}, we immediately deduce the following.
\begin{Proposition}\label{prop:relate} We have an isomorphism of algebras
	\[  \CBh(L, N^{\xi}_0)\cong \cdots \otimes \CBh(\Bv^{(-1)}, 0)\otimes
	\CBh(\Bv^{(0)}, \Bw)\otimes \CBh(\Bv^{(1)}, 0)\otimes\cdots,\]
	where
	each factor on the RHS is again the Coulomb branch of a quiver gauge theory of
	the same underlying quiver, all but one of them unframed.
\end{Proposition}
Thus for a quiver gauge theory, Theorem \ref{thm:rel}
relates the algebra
$\CBh(\Bv,\Bw)$ to the tensor product appearing in this proposition.

\subsection{Truncated shifted Yangians} \label{se:Ylamu}
Assume that $\Gamma$ is a Dynkin quiver and let $ \fg_\Gamma$ be the corresponding simply-laced simple Lie algebra.  The dimension vectors $\mathbf w, \mathbf v$, encode a pair of coweights $\lambda, \mu $, for $ \fg_\Gamma$, where $ \lambda = \sum w_i \varpi_i $ and $ \lambda - \mu = \sum v_i \alpha_i $.

Let $ Y^\lambda_\mu $ denote the corresponding truncated shifted Yangian (with formal parameters and with $ \hbar $) as defined in appendix B(viii) of \cite{BFNslices}.  In \cite{WeekesCoulomb}, the fourth author proved that there is an isomorphism $ \CBh(\Bv, \Bw) \cong Y^\lambda_\mu $,  building on the results in appendix B of \cite{BFNslices}.

The splitting $v_i=\sum_{p\in \Z} v_i^{(p)} $ gives us a list of $ \mu^{(p)} $, where $ \mu^{(p)} = - \sum v_i^{(p)} \alpha_i$ for $ p \ne 0 $ and $ \lambda - \mu^{(0)} = \sum v_i^{(0)} \alpha_i $ (in particular we have $ \mu = \sum \mu^{(p)} $).  Thus, Theorem \ref{thm:rel} and Proposition \ref{prop:relate} relate  $Y^\lambda_\mu $ with $ Y^\lambda_{\mu^{(0)}} \otimes \bigotimes_{p \ne 0 } Y^0_{\mu^{(p)}} $.  For example, we can relate $ Y^\lambda_{\mu^{(0)} + \mu^{(1)}} $ with $ Y^\lambda_{\mu^{(0)}} \otimes Y^0_{\mu^{(1)}}$.

An important special case is when we take $ \mu^{(1)} = -\alpha_i $ for some $ i$ and all other $ \mu^{(p)} = 0 $.  This corresponds to taking $ \xi $ to be the $i$th standard coweight $ \xi = \varpi_{i,1} $ (i.e. we map $ \Cx $ into $ G $ by sending it to the upper left matrix slot of $ GL(v_i)$).  In this case, $( L/\Cx_\xi, N^{\xi}_0) $ is the quiver gauge theory corresponding to $ \lambda$ and $\mu + \alpha_i $, so $ \CBh(L/\Cx_\xi, N^{\xi}_0) \cong Y^\lambda_{\mu + \alpha_i}$.  Thus Theorem \ref{thm:rel} relates $Y^\lambda_\mu $ and $Y^\lambda_{\mu + \alpha_i}$.  In section \ref{se:FunctorQuiver}, this will allow us to construct functors between category $\mathcal O $ for these algebras giving a categorical $ \fg_\Gamma$-action.

\begin{Remark} \label{re:comult}
	In \cite{FKPRW}, half of the authors along with Finkelberg, Pham, and
	Rybnikov, constructed a comultiplication map $ Y_{\mu^{(0)} +
		\mu^{(1)}} \rightarrow Y_{\mu^{(0)}} \otimes Y_{\mu^{(1)}} $.  We
	believe that the comultiplication descends to maps $ Y_{\mu^{(0)} +
		\mu^{(1)}}^{\lambda^{(0)} + \lambda^{(1)}} \rightarrow
	Y_{\mu^{(0)}}^{\lambda^{(0)}} \otimes Y_{\mu^{(1)}}^{\lambda^{(1)}}$,
	and in particular to a map $Y_{\mu^{(0)}+\mu^{(1)}}^\lambda
	\rightarrow Y_{\mu^{(0)}}^\lambda \otimes Y_{\mu^{(1)}}^0$.  In
	\cite{KPW}, we partially proved this statement in the case that $
	\mu^{(1)} = -\alpha_i$.  However, we don't currently understand the relationship between comultiplication and Theorem \ref{thm:rel}.  In particular, the comultiplication map is not compatible with the integrable systems.
\end{Remark}

\section{Modules for Coulomb branch algebras}
\label{sec:modul-coul-branch}

\subsection{Gelfand-Tsetlin modules}

We
wish to understand the representation theory of the algebra
$\CBh(G,N)$, following the approach of \cite{Webdual,KTWWYO,WebGT}.  In particular we will focus on the Gelfand-Tsetlin modules.

Now and for the remainder of the paper we specialize to $ \hbar = 1 $, and we will 
write $ \CB(G,N) $ for the result of this specialization\footnote{In \cite{KTWWY,KTWWYO}, we considered the specialization $ \hbar = 2 $ in order to match the conventions of
	some earlier papers.  In any case, all non-zero $\hbar $ give isomorphic specializations.}.  Later in the paper, we will further specialize at a point $ \varphi \in \mathfrak f = \MaxSpec Z $, and consider $  \CB_\varphi(G,N) = \CB(G,N)/ \CB(G,N) \mathsf m_\varphi$, where $\mathsf{m}_\varphi\subset Z$ is the corresponding maximal ideal.

Given $\gamma\in \tilde{\ft} / W = \MaxSpec (\HBG{\tG})$, let $\mathfrak{m}_\gamma\subset \HBG{\tG}$
be the corresponding maximal ideal.
Consider the weight functors $ \Wei_\gamma: \CB(G,N)\mmod \rightarrow \Vect $ defined by
\[ \Wei_\gamma(M)=\{ m\in M \mid
\mathfrak{m}_\gamma^Nm=0 \quad\text{ for some }  N\gg 0\}.\]  The reader might reasonably be
concerned about the fact that this is a {\it generalized} eigenspace.
In this paper, we will always want to consider these, and thus will
omit ``generalized'' before instances of ``weight.''  These spaces $ \Wei_\gamma(M) $ are called  {\bf Gelfand-Tsetlin (GT) weight spaces}.

\begin{Definition} \label{def:GT}
	An $\CB(G,N)$-module $M$ is called a {\bf ($Z$-semisimple) Gelfand-Tsetlin module}
	if it is finitely generated, $M=\bigoplus_{\gamma\in \tilde{\ft}/W} {\Wei_{\gamma}}(M)$,  and the center $ Z $ acts semisimply on $ M $.  For the rest of the paper, we let $\Wei_\gamma$ denote the restriction of the weight functor to the category of $Z$-semisimple Gelfand-Tsetlin modules.

	The {\bf support} of a Gelfand-Tsetlin module is the set  \[\Supp(M)= \{ \gamma
	\in \tilde{\ft} / W \mid \Wei_{\gamma} (M) \ne 0 \} .\]
\end{Definition}

\begin{Remark}
	In this paper, we will always assume that Gelfand-Tsetlin modules are $Z$-semisimple, but will periodically add comments about how our results would change if we allowed the center $Z$ to act with non-trivial nilpotent part. 
\end{Remark}

\begin{Remark}
	For the parabolic Coulomb branch algebra $ \CB^P(G,N)$, we also have a Gelfand-Tsetlin subalgebra $ \HBG{\tL} $.  In a similar way, we can speak about Gelfand-Tsetlin weight functors $ {}^P \Wei_\nu $ and Gelfand-Tsetlin modules.  In this case, the weight $ \nu $ lies in $ \tilde{\ft} / W_L=\MaxSpec( \HBG{\tL} )$.
\end{Remark}

\subsection{Morphisms between weight functors} \label{se:FS}
Gelfand-Tsetlin modules can be classified using the approach of \cite{FOD},
which is based on understanding the space of natural transformations $\Hom(\Wei_\gamma,\Wei_{\gamma'})$.  Observe that
\begin{equation*}
	\Hom(\Wei_\gamma,\Wei_{\gamma'})=\varprojlim
	\CB/(\CB\mathfrak{m}_\gamma^N+\mathfrak{m}_{\gamma'}^N\CB+\CB\mathsf{m}_\varphi)=\varprojlim
	\CBs{\varphi}/(\CBs{\varphi}\mathfrak{m}_\gamma^N+\mathfrak{m}_{\gamma'}^N
	\CBs{\varphi}),
\end{equation*}
for $ \gamma, \gamma' \in \tilde{\ft}/W $ with common image $ \varphi
\in \mathfrak f $.  If $ \gamma, \gamma' $ don't have the same image in $\mathfrak{f}$, then $\Hom(\Wei_\gamma,\Wei_{\gamma'})=0$.

\begin{Remark}
	Note that this calculation of natural transformations is only valid for the restriction of $\Wei_{\gamma}$ to the category of $Z$-semisimple modules; on the full category of $\CB(G,N)$-modules, each element $z\in Z$ defines a natural transformation.  Under our convention of $Z$-semisimplicity, this element $ z $ acts on $\Wei_\gamma$ by the scalar $\varphi(z)$, but there are Gelfand-Tsetlin modules which are not $Z$-semisimple where this does not hold.
\end{Remark}

The space $ \Hom(\Wei_\gamma,\Wei_{\gamma'}) $ has a natural weak
operator topology, where a sequence converges if and only if it is
eventually constant  on $\Wei_\gamma(M)$ for all $M$.
This is the same as the inverse limit topology.

These spaces $\Hom(\Wei_\gamma, \Wei_{\gamma'}) $ can be organized into the algebra $\End(\oplus \Wei_{\gamma})$,
or alternatively, into the category whose objects are weight functors. 
Since $\tilde{\ft} / W$ is an uncountably infinite set, $\End(\oplus
\Wei_{\gamma})$ is a very large algebra.  Luckily it naturally
decomposes into summands: consider the partition of
$\tilde{\ft} / W$ into the disjoint union of the images of
orbits of the extended affine Weyl group $\widehat{W}=N_{G(\mathcal{K})}(T)/T\cong W\ltimes \ft_{\Z}$ on $ \tilde{\ft}$.  As
discussed in \cite[\S 2.5]{WebGT}, the
space of natural transformations $\Hom(\Wei_\gamma,\Wei_{\gamma'})$ is
non-zero if and only if $\gamma$ and $\gamma'$ lie in the image $\bar{\mathscr S} \subset \tilde{\ft} / W$ of a $\widehat{W}$--orbit 
$\mathscr{S} \subset \tilde{\ft}$.  Thus, an indecomposable Gelfand-Tsetlin module must
have weights concentrated on the image of single orbit $\mathscr{S}$ and the category of all Gelfand-Tsetlin modules decomposes as a direct sum subcategories of modules with  weights concentrated on the image of single orbit.

Consider a set $\mathsf{S}\subset \tilde{\ft}$, and its image $\bar{\mathsf{S}} \subset \tilde{\ft} / W$. We define\footnote{Since English-speaking readers
	may not be familiar with these letters, the cyrillic $\GTc$ is
	pronounced roughly ``geh-tseh.'' These are the first letters of the
	names ``Gelfand'' and ``Tsetlin'' in Russian.} $\CB(G,N)\dash\GTc_\mathsf{S}$ to be
the category of  Gelfand-Tsetlin
modules modulo the subcategory of modules killed by $\Wei_{\gamma}$
for all $\gamma\in \bar{\mathsf{S}}$. In the case where
$\mathsf{S}=\mathscr{S}$ is a single $\widehat{W}$-orbit, $\CB(G,N)\dash\GTc_\mathsf{S}$ is just the subcategory of modules with
support in $\bar{\mathscr{S}}$.

Consider the algebra  $\wtAlg(\mathsf{S})=\End(\oplus_{\gamma\in
	\bar{\mathsf{S}}} \Wei_{\gamma})$.  More generally, given two sets $ \mathsf S, \mathsf S' \subset \tilde{\ft}$, we let
\begin{equation} \label{eq:defFSS}
	\wtAlg(\mathsf{S},
	\mathsf{S}')=\Hom\Big(\bigoplus_{\gamma\in \bar{\mathsf{S}}}\Wei_{\gamma}, \bigoplus_{\gamma'\in \bar{\mathsf{S}'}}\Wei_{\gamma'}\Big),
\end{equation}
which is an $ \wtAlg(\mathsf S'), \wtAlg(\mathsf S) $ bimodule.

\begin{Proposition}[\mbox{\cite[Th. 17]{FOD}; \cite[Th. 2.23]{WebGT}}]\label{thm:GTc-F}
	The functor $\Wei_{\mathsf S} := \oplus_{\gamma\in \bar{\mathsf{S}}} \Wei_{\gamma}$ gives an
	equivalence of categories between $\CB(G,N)\dash\GTc_\mathsf{S}$ and
	modules over $\wtAlg(\mathsf{S})$ continuous in the discrete topology.
\end{Proposition}

We recall the following definition from \cite[Def. 2.24]{WebGT}.

\begin{Definition} \label{de:CompleteSet}
	Let $ \mathscr S \subset \tilde{\ft} $ be a $ \widehat W$-orbit.  A finite set $\mathsf{S} \subset \mathscr S $ is called a {\bf complete set} for $\mathscr{S}$, if for every $\gamma\in
	\bar{\mathscr{S}}$, there is a $\gamma'\in \bar{\mathsf{S}}$ such that
	$\Wei_{\gamma}\cong \Wei_{\gamma'}$.
\end{Definition}

By \cite[Cor. 4.16]{WebGT}, a complete set $ \mathsf S $ always exists.  In this case, 
$\CB(G,N)\dash\GTc_\mathsf{S}$ is equal to $\CB(G,N)\dash\GTc_\mathscr{S}$.  Moreover, if $\mathsf{S},\mathsf{S}'$ are both complete sets for $ \mathscr S $, then $\wtAlg(\mathsf S, \mathsf S') $ gives
a Morita equivalence between $ \wtAlg(\mathsf S) $ and $ \wtAlg(\mathsf S') $.  

Note that the existence of a complete set shows that $\CB(G,N)\dash\GTc_\mathscr{S}$ is Artinian, since the functor $\Wei_{\mathsf S}$ is an equivalence and any module in its image is finite-dimensional.

\subsection{Category $\cO$}
\label{sec:category-co}

The Coulomb branch $\CB(G,N)$ has a Hamiltonian action of the torus
$K=(G/[G,G])^\vee$, so we have a Hamiltonian $\Cx$-action on $ \CB(G,N)$ for each character $\chi$ of
$G$. 

More precisely, the map $\chi\colon G\to \Cx$ induces a map $ \Gr_G \rightarrow \Gr_{\Cx}$.  Since $\Gr_{\Cx} = \Z$, this gives a
$\Z$-grading to $ \CB(G,N) $, and thus an action of $\Cx$.  By abuse of notation, we let $\chi$
also denote the derivative of the character $ \chi$, interpreted as an element of $(\ft^*)^W$.  Then by \cite[Lemma 3.19]{BFN}, we have the
quantum moment map relation: for $a \in \CB(G,N)$ an element of weight $k \in \Z$ (with respect to $ \chi$),
we have $[\chi,a]=ka$.

\begin{Definition}\label{def:catO}
	The {\bf category $\Oof{\CB(G,N)}$} for $\chi$ is the full subcategory of
	finitely generated
	$\CB(G,N)$-modules on which
	\begin{itemize}
		\item $Z$ acts semisimply, and 
		\item $\chi$ acts locally finitely, with finite dimensional generalized eigenspaces with eigenvalues whose real parts are bounded above.
	\end{itemize}  
\end{Definition}
\begin{Remark}
	A module in $ \Oof{\CB(G,N)}$ is necessarily a Gelfand-Tsetlin module, so this is equivalent to
	asking that elements of non-negative weight for the grading induced by $ \chi $ act locally finitely, see \cite[Lemma 3.13]{BLPWgco}.
\end{Remark}

As discussed in \cite[Th. 4.10]{Webdual},  we can realize category
$\Oof{\CB(G,N)}$ using the same algebraic approach that we used to
understand the category
$\CB(G,N)\dash\GTc_\mathscr{S}$.  The set $\bar{\mathscr{S}}$ is divided into finitely
many equivalence classes by the relation $\gamma\sim \gamma'$ if
$\Wei_{\gamma}\cong \Wei_{\gamma'}$ as functors.  On each equivalence
class, either:
\begin{itemize}
	\item Thought of as a function on $\bar{\mathscr{S}}$, the function $\gamma \mapsto \Re(\chi(\gamma))$ attains a
	maximum, and attains it at a finite number of points.  In this case,
	we call the equivalence class {\bf bounded}.
	\item  On $ \bar{\mathscr S}$, the function $\gamma \mapsto \Re(\chi(\gamma))$ either has no  maximum, or it attains its
	maximum on an infinite set.  In this case, we call the equivalence
	class {\bf  unbounded}.
\end{itemize}

\begin{Proposition} \label{prop:GTinO}
	Every module in category $\Oof{\CB(G,N)}$ is a Gelfand-Tsetlin module.  Conversely,
	a  Gelfand-Tsetlin module $M\in \CB(G,N)\dash\GTc_\mathscr{S}$ lies in category $\Oof{\CB(G,N)}$ if
	and only if $\Wei_{\gamma}(M)=0$ for all $\gamma$ in an unbounded equivalence class. 
\end{Proposition}

Thus, let $\Oof{\CB(G,N)}_\mathscr{S}$ denote the subcategory of
$\CB(G,N)\dash\GTc_\mathscr{S}$ given by objects in $\Oof{\CB(G,N)}$.  For any $ \mathsf S \subset \tilde{\ft}$, let $\mathscr{I}\subset
\wtAlg(\mathsf{S})$ be the two-sided ideal generated by the
identity on $\Wei_{\gamma}$ for all $\gamma\in \bar{\mathsf{S}}$ whose
equivalence class is
unbounded. (Alternatively, this is the span
of the natural transformations that factor through such $\Wei_{\gamma}$.)

This gives us the following modification of Proposition
\ref{thm:GTc-F}:
\begin{Proposition}\label{prop:O-F}
	Let $\mathsf{S}\subset \mathscr{S}$ be a finite
	complete set.  The functor $\oplus_{\gamma\in \bar{\mathsf{S}}} \Wei_{\gamma}$ gives an
	equivalence of categories between $\Oof{\CB(G,N)}_\mathscr{S}$ and
	modules over $\wtAlg(\mathsf{S}) /\mathscr{I}$ continuous in the discrete topology.
\end{Proposition}
As mentioned above, this is a rephrasing of  \cite[Th. 4.10]{Webdual}: in the notation of that paper, the algebra $F(\mathsf{S})$ would be denoted $A_{\mathsf{S}}$ and the ideal $\mathscr{I}$ by $\mathscr{I}_{-\chi}$.  

\subsection{Gelfand-Kirillov dimension}
\label{sec:gelf-kirill-dimens}

Let $\mathbbm{k}$ be a field.  
Consider a $\mathbbm{k}$-algebra $A$   which
is generated by a finite dimensional subspace $A_0$, and a left $A$-module
$M$ which is finitely generated by a finite dimensional subspace
$M_0$.  In this context, the {\bf Gelfand-Kirillov dimension}
$\operatorname{GKdim}_A(M)$ is defined by:
\begin{equation}
	\label{eq:1}
	\operatorname{GKdim}_A(M)=\limsup_{n\to \infty }\log_n\dim_{\mathbbm{k}}(A_0^nM_0)
\end{equation}
It's a standard result that this number is independent of the choice of
$A_0$ and $M_0$, and only depends on the structure of $M$ as an
$A$-module.   


Given a Gelfand-Tsetlin module $M$ for $\CB(G,N)$, let $M^B$ denote the corresponding module of the Morita equivalent algebra $\CB^B(G,N)$ (see Theorem \ref{thm:rel}(1) and  Section \ref{sec:res-ind} below). Recall that $\CB^B(G,N)$ contains $\CB(T,N)$ as a subalgebra, by Proposition \ref{th:inclusion}.

\begin{Lemma}\label{lem:GKdim}
	Let $M$ be a Gelfand-Tsetlin module for $\CB(G,N)$.  Consider:
 \begin{enumerate}
     \item The Gelfand-Kirillov dimension of $M$ over $\CB(G,N)$.
     \item The Gelfand-Kirillov dimension of $M^B$ over $\CB(T,N)$.
 \item The dimension of the Zariski closure of the support of $M$ in
	$\tilde{\ft}/W$.
 \item The Krull dimension of $\gr M$ as a $\gr \CB(G,N)$ module, for any good filtration of $M$.  
 \end{enumerate}
 Then $(1) = (2) = (3) \geq (4)$.
\end{Lemma}
\begin{proof}
{$(1) \geq (2)$:} This is a consequence of:
$$
\operatorname{GKdim}_{\CB(G,N)}(M) = \operatorname{GKdim}_{\CB^B(G,N)}(M^B) \geq \operatorname{GKdim}_{\CB(T,N)}(M^B)
$$
Here, the equality on the left is due to the Morita equivalence of $\CB(G,N)$ and $\CB^B(G,N)$, while the inequality on the right comes simply from the fact that $\CB(T,N) \subset \CB^B(G,N)$ is a subalgebra.

{$(2) = (3)$:} This follows from the same argument as  \cite[Prop. 8.2.3]{MVdB}. In fact, if the action of $T$ on $N$ is faithful, then by \cite[\S 4(vii)]{BFN}, the algebra $\CB(T,N)$ is precisely one of the algebras $B^{\chi}$ considered by \cite{MVdB}.

{$(3)\geq (1)$:}	As noted above, we may equivalently look at the Gelfand-Kirillov dimension of $M^B$ over $\CB^B(G,N)$.  Now, if $a \in \CB^B(G,N)$ is any element, then there exist elements $w_1,\ldots, w_p\in \widehat{W}$ of the extended affine Weyl group such that
\begin{equation}
\label{eq: effect of elements on GT weights}
a \cdot \Wei_\lambda(M^B) \ \subseteq \ \sum_{i=1}^p \Wei_{w_i(\lambda)}(M^B)
\end{equation}
for every weight $\lambda \in \tilde\ft$. This claim follows from the results of \cite{WebGT}, see also \cite[Lemma 3.1(4)]{FOfibers}. Explicitly, the localization of $\CB^B(G,N)$ at the fraction field of $\HBG{\tilde{T}}$ is isomorphic to the smash product $\operatorname{Frac}(\HBG{\tilde{T}})\# \widehat{W}$. The image $a \mapsto  \sum_{i=1}^p a_i w_i$ in this localization is then a finite sum where $a_i \in \operatorname{Frac}(\HBG{\tilde{T}})$ and $w_i \in \widehat{W}$. These are precisely the desired elements $w_i$.

Choose any finite dimensional generating subspace $A_0$ of $\CB^B(G,N)$, which exists by \cite[Prop. 6.8]{BFN}. Then we can choose $w_1,\ldots, w_p \in \widehat{W}$ such that (\ref{eq: effect of elements on GT weights}) holds uniformly for all $a \in A_0$.

 Now, fix any $W$-invariant norm on $ \tilde\ft$. Note that each $w_i$ above is given by an element of the finite Weyl group times a translation $\nu_i$. Let $\epsilon$ be the maximum of the norms of the these elements $\nu_i$.  Thus, by the triangle inequality the ball
	$B(t)$ of radius $t$ around $0$ satisfies $w_i\cdot B(t)\subset B(t+\epsilon)$ for any  $t >0$.  Consequently, by (\ref{eq: effect of elements on GT weights}) we have \[A_0 \cdot \bigoplus_{\lambda\in
		B(t)} \Wei_{\lambda}(M^B)\subset \bigoplus_{\lambda\in
		B(t+\epsilon )}\Wei_{\lambda}(M^B).\]

 Next, choose any finite dimensional generating set $ M_0$ for  $M^B$.  It is contained in a sum of finitely many weight spaces so we can  choose $ t_0$ so that $M_0 \subset \oplus_{\lambda\in
		B(t_0)} \Wei_{\lambda}(M^B)$.
Thus, we find that
$$
\dim A_0^n M_0 \le \sum_{\lambda \in B(t_0 + n\epsilon)} \dim \Wei_{\lambda}(M^B) 
$$
We conclude that the Gelfand-Kirillov dimension satisfies:
	\[ \operatorname{GKdim}_{\CB^B(G,N)}(M^B)\leq \lim_{t\to \infty}\frac{\log \sum_{\lambda\in
			B(t)} \dim \Wei_{\lambda}(M^B)}{\log t} \]

	Consider the equivalence relation
	$\lambda\sim \lambda'$ iff $\Wei_{\lambda}\cong
	\Wei_{\lambda'}$.  By \cite[Cor. 4.16]{WebGT}, each $ \widehat{W}$-orbit only contains finitely many equivalence classes, called clans, and each is the $W$-orbit of an intersection of a $\ft_{\Z}$-coset with a convex polyhedron. 
  Each indecomposable Gelfand-Tsetlin module is supported on a single $ \widehat{W}$ orbit (this follows from (\ref{eq: effect of elements on GT weights}); see also the discussion before Proposition \ref{thm:GTc-F} ), so by finite generation, $\Supp(M^B)$  is contained in a finite union of $\widehat{W}$-orbits. Thus there are
	only finitely many equivalence classes $\{U_1,\dots,
	U_r\}$ in $\Supp(M^B)$.

	The
	Zariski closure $\overline{\Supp(M^B)}$ is thus the union of the Zariski
	closures $\overline{U_i}$.  Thus, if we define $ d= \dim \overline{\Supp(M^B)} $, then	$d=\max_{i=1}^r\dim(\overline{U_i})$.
	Similarly, we have an equality \[\lim_{t\to \infty}\frac{\log \dim \bigoplus_{\lambda\in
		B(t)}\Wei_{\lambda}(M^B)}{\log t}=\max_{i=1}^r\lim_{t\to \infty}\frac{ \log |U_i\cap
		B(t)| + \log  m_i}{\log t}.\]
  where $ m_i = \dim \Wei_\lambda(M^B)$ for any $ \lambda\in U_i$ (these dimensions are constant by the definition of $ U_i$).
  
	Since $U_i$ is a finite union of the intersections of $\ft_{\Z}$-cosets with convex polyhedra,
	this latter growth rate is exactly the dimension of its Zariski
	closure.  This shows that
	\[ \lim_{t\to \infty}\frac{\log \dim \bigoplus_{\lambda\in
			B(t)}\Wei_{\lambda}(M^B)}{\log t} =d\]
	which completes the proof that $(3) \geq (1)$.

Finally, $(1) \geq (4)$ is a general property of passing to associated graded algebras and modules.  Indeed, if the Krull dimension of $\gr M$ is $q$, then by Noether normalization, there are $q$ algebraically independent elements $\bar a_1,\dots, \bar a_q$ of $\gr \CB(G,N)$, and which act on some element  $\bar{v}\in \gr M$ with no relations beyond commutativity.  Thus, if $v,a_1,\dots, a_q$ are preimages of these elements, and $A_0$ the span of $a_1,\dots, a_q$, then we have $\dim A_0^n\cdot v\geq \binom{n+q-1}{q-1}$, so this shows that $\operatorname{GKdim}_{\CB(G,N)}(M)\geq q$.  
\end{proof}

Let us include here a closely related result:

\begin{Lemma}
	Let $M$ be a Gelfand-Tsetlin module for $\CB(G,N)$.  Then 
	\[\operatorname{GK-dim}(\CB(G,N)/\operatorname{ann}(M))=2\operatorname{GK-dim}(M).\]
\end{Lemma}
\begin{proof}
	In the case $G=T$, this proven by Musson and van der Bergh in  \cite[Cor. 8.2.5]{MVdB}.  
	
	Thus, in the general case, exactly as in the proof of Lemma \ref{lem:GKdim}, we can consider $M^B$ over the Morita equivalent
	algebra $\CB^B(G,N)$.  Let $A=\CB^B(G,N)/\operatorname{ann}_{\CB^B(G,N)}(M^B)$ and $A_T=\CB(T,N)/\operatorname{ann}_{\CB(T,N)}(M^B)$. The abelian case proves that 	\[\operatorname{GK-dim}(A_T)=2\operatorname{GK-dim}(M^B)\]
where the latter is regarded as an $\CB(T,N) $ module. By the previous Lemma, we have $ \operatorname{GK-dim}(M)=\operatorname{GK-dim}(M^B)$ and so it suffices to prove that
	\[\operatorname{GK-dim}(A) = \operatorname{GK-dim}(A_T).\]

	Since $\CB(T,N)/\operatorname{ann}_{\CB(T,N)}(M^B)$ is a subalgebra of $\CB^B(G,N)/\operatorname{ann}(M^B)$, we thus have 
	\[\operatorname{GK-dim}(A)\geq \operatorname{GK-dim}(A_T).\]
	
Thus, we need only prove the opposite inequality.  For $w\in \widehat{W}$, let   $R_{G,N}^B(\leq w)$ be the preimage of the Schubert variety $\overline{IwI}/I$ in $R_{G,N}^B$. Let $\CB^B(\leq w)$ be the homology of this subspace; this is a $\HBG{T}\operatorname{-}\HBG{T}$-subbimodule of $\CB^B(G,N)$.  Consider $\CB^B(\leq w)/\CB^B(< w)$.  This quotient is the homology $H_*^{\tilde{I}\rtimes\C^{\times}}(IwI/I)$, which is a free module of rank 1 over $\HBG{T}$ as a left module or as a right module, with the two actions differing by the action of $w$.  

The image of $\HBG{T}$ in $A$ is a quotient of this polynomial ring.  Note that for a simple $\CB(T,N)$ GT-module $M'$, the kernel of the map  $\HBG{T}\to \operatorname{ann}_{\CB(T,N)}(M')$ is the intersection of the maximal ideals associated to the non-zero weight spaces, and is thus the radical ideal of polynomials vanishing on $\overline{\Supp(M')}$;  for example, this follows from \cite[Prop. 3.1.7]{MVdB}.  

This implies that the kernel for an arbitrary $\CB(T,N)$ GT-module lies in the product of finitely many such ideals.  In particular, the image of the map $\HBG{T}\to A$
is a not-necessarily-radical quotient  whose support as a coherent sheaf is $\overline{\Supp(M^B)}$.

Thus, the same is true of any subquotient of $A$ as a left or right $\HBG{T}$-module, considered as a quasi-coherent sheaf.  In particular, taking the corresponding quotient $A(\leq w)/A(<w)$, we thus  obtain a $\HBG{T}\operatorname{-}\HBG{T}$-bimodule whose support as a left and a right module must be in $\overline{\Supp(M^B)}$.  Since these actions differ by $w$, the support as a left $\HBG{T}$ module must lie in $\overline{\Supp(M^B)}\cap w\cdot \overline{\Supp(M^B)}$.  
Now, note that all the components of $\overline{\Supp(M^B)}$ are affine subspaces which as $\leq d$ dimensional, since the Zariski closure of a clan has this form.  

Now, for any two affine subspaces $E_1,E_2$ of an $n$-dimensional affine space,  modeled on vector subspaces $V_1,V_2$ which are  $m_1$ and $m_2$-dimensional.  The set of elements $x$ such that $E_1\cap E_2\neq \emptyset$ is an affine subspace itself, modeled on the span $V_1+V_2$, and this intersection is an affine subspace modeled on $V_1\cap V_2$.  
If we apply this with $E_1$ a component of $\overline{\Supp(M^B)}$, and $E_2$ the image of such a component under an element of the finite Weyl group, we find that the intersection $\overline{\Supp(M^B)}\cap w\cdot \overline{\Supp(M^B)}$ is $\geq k$ dimensional only for extended affine Weyl group elements whose translation parts lie in a $2d-k$ dimensional variety where $d=\dim \overline{\Supp(M^B)}$. 

Choose a generating set of the extended affine Weyl group containing the simple reflections, let $\ell_0(w)\geq \ell(w)$ be the length function of with respect to these generators.  (Since $G $ is not usually semisimple, we will need extra generators.)  Let $\CB^B(\leq n)$ be the span of $\CB^B(\leq w)$  for all elements $w$ with $\ell_0(w)\leq n$.  Now, consider the span  $\CB_0$  of
\begin{enumerate}
	\item the degree 1 elements $\ft^*\subset \HBG{T} $ and 
	\item generators of $ \CB^B(\leq n)$ as a left $\HBG{T}$-module for some large $n$
\end{enumerate} 
For $n$ sufficiently large, this will be a set of generators of $\CB^B(G,N)$ as an algebra.  Let $A_0$ be the image of $\CB_0$ in $A$. We need to show that the dimension of $A_0^m$ does not grow too quickly in terms of $m$; in order to obtain this bound, we use the filtration $A(\leq w)$ discussed above, and the loop filtration $F_p\CB$ induced by the homological grading before we specialize $\hbar$ to 1.  

In particular, we let $D$ be the smallest integer such that $\CB_0\subset F_{D}\CB^B$.  

Thus, we have  $\CB_0^m\subset \CB^B(\leq nm)\cap F_{Dm}\CB$.

If $\dim \overline{\Supp(M^B)}\cap w\cdot \overline{\Supp(M^B)} \leq k$, then we must have that the dimension of $F_pA(\leq w)/F_pA(<w)$ must be bounded by $D'' p^k$ for all $p$ for some constant $D''$; since this intersection is a union of affine spaces, whose number of components is bounded by the number of pairs of components, we can choose one $D''$ which works for all $w$. 
Combining our estimates, we find:
\begin{enumerate}
	\item The number of $w\in \widehat{W}$ with  $\ell_0(w)\leq nm$ such that $\dim \overline{\Supp(M^B)}\cap w\cdot \overline{\Supp(M^B)}=k$ is bounded above by $D'm^{2d-k}$.
	\item The dimension of $(A(\leq w)\cap  A_0^m)/(A(< w) \cap A_0^m)$ is bounded above by $D''(Dm)^k$ if $\ell_0(w)\leq nm$ .  
\end{enumerate}
Thus, summing over $k=1,\dots, 2d$,  we have  $\dim A_0^m\leq D'D''D^{2d}m^{2d}$.  Thus, we have
\[\log_m(\dim A_0^m)\leq 2d +\frac{\log(2dD'D'')D^{2d}}{\log m} \]
so taking the limit, we have $\operatorname{GK-dim}(A)\leq 2d$, 
completing the proof.   
\end{proof}

\section{Functors from parabolic restriction}
\label{sec:functors}

\subsection{Restriction and induction functors}
\label{sec:res-ind}

The relations of Theorem \ref{thm:rel} induce a number of functors.
\begin{enumerate}
	\item The Morita equivalence of Proposition \ref{th:Morita} gives us an equivalence of
	categories
	\begin{equation*}
		\CB(G,N)\mmod \cong  \PCB(G,N)\mmod.
	\end{equation*}
	This functor $
	\PCB(G,N)\mmod \rightarrow \CB(G,N)\mmod $ can be written as $ M
	\mapsto e'M $.  Let $M^P=\PCB e'\otimes_{\CB}M$ be the inverse
	equivalence.

	\item Restriction induces a functor
	\begin{equation*}
		\PCB(G,N)\mmod \to  \CB(L,N)
		\mmod.
	\end{equation*}
	This functor has a left adjoint given by
	$\PCB(G,N)\otimes_{\CB(L,N)}-$ and a right adjoint
	$\Hom_{\CB(L,N)}(\PCB(G,N), -)$.
	\item Tensor product $\CB(L, N^{\xi}_0)\otimes_{\CB(L, N)} -$ induces an
	exact functor
	\begin{equation*}
		\CB(L,N)\mmod \to \CB(L, N^{\xi}_0)\mmod.
	\end{equation*}
	This has a
	right adjoint given by restriction.
\end{enumerate}

\begin{Definition} \label{def:res}
	We define the functor $ \res_\xi =\res \colon \CB(G,N)
	\mmod\to  \CB(L, N^{\xi}_0)\mmod$ as the composition of the above three functors
	$$
	\CB(G,N) \mmod \rightarrow \CB^P(G,N) \mmod \rightarrow \CB(L,N) \mmod \rightarrow \CB(L, N^\xi_0) \mmod
	$$
\end{Definition}

\begin{Lemma}
	The functor $\res\colon 	\CB(G,N)
	\mmod\to  \CB(L, N^{\xi}_0)\mmod$ is exact and admits a right adjoint $\coind\colon
	\CB(L, N^{\xi}_0)\mmod\to 	\CB(G,N) \mmod.$
\end{Lemma}
\begin{proof}
	The right adjoint is just defined by composing the right adjoints of the functors (1-3) above.  The functor (1) above is exact since it is a Morita equivalence.  The functor (2) is exact since it is the same underlying vector space, and the functor (3) is exact since it is localization of the action of a polynomial ring.
\end{proof}

\subsection{Functors on Gelfand-Tsetlin modules}

\begin{Lemma} \label{le:preserveGT}
	The functor $\res$ 
	preserves Gelfand-Tsetlin modules.
\end{Lemma}
\begin{proof}
	It's clear that the Morita equivalence (1) preserves Gelfand-Tsetlin modules.
	Similarly, the restriction functor in (2) above clearly does so by compatibility
	with the Gelfand-Tsetlin subalgebras.
	
	The functor of tensor product with $\CB(L, N^{\xi}_0)\otimes_{\CB(L, N)}
	-$ preserves local finiteness under $\HBG{\tL}$ since $\CB(L, N^{\xi}_0)$
	has the Harish-Chandra property as a bimodule over $\HBG{\tL}$: the
	sub-bimodule generated by any finite set is finitely generated as a
	right module (and also as a left module, though that is not relevant
	here).
	For any element $v\in \CB(L, N^{\xi}_0)\otimes_{\CB(L, N)}
	M$, the subspace $\HBG{\tL}\cdot v$ lies in the image of a
	tensor
	product $B\otimes_{\HBG{\tL}} C$ of a finitely generated
	sub-$\HBG{\tL} \operatorname{-}\HBG{\tL}$-bimodule $B$ and a finite
	dimensional $\HBG{\tL}$-submodule $C$.  Since $B$ is finitely generated
	as a right module, $ B\otimes_{\HBG{\tL}} C$ is finite-dimensional,
	showing the local finiteness.
	\excise{
		Finally when we consider $\PCB(G,N)$ as a left module over
		$\CB(L,N)$, it is finitely generated.  This generating set generates
		a subbimodule $S$ over $\HBG{\tL}$ which is finitely generated as a left
		and a right module over $\HBG{\tL}$.  Thus, any homomorphism over
		$\CB(L,N)$ maps $S$ into a sum of finitely many weight spaces, and
		for any fixed set of weight spaces, the set of $\HBG{\tL}$-module
		maps of $S$ to this sum of weight spaces is finite-dimensional, and
		closed under the action of $S$.  This shows the local finiteness of
		$\Hom_{\CB(L,N)}(\PCB(G,N), M)$ under the action of $S$.  }
\end{proof}

Let us examine a little more carefully the effect of $\res$ on weight
spaces.

\subsubsection{Morita equivalence on GT weight spaces}
The effect of the Morita equivalence from Proposition \ref{th:Morita} on weight spaces is
covered in \cite[Lem. 2.8]{WebGT}.  Let $\nu\in \tilde{\ft}/W_L$,
and $\gamma$ its image in  $ \tilde{\ft}/W$.  The preimage of $\gamma$ in
$\tilde{\ft}$ is a single orbit of the Weyl group $W$, and $\nu$ corresponds
to an orbit of $W_L$ within this larger orbit.  Generically both
these orbits are free, but there are degenerate cases where they are not.  The effect of the Morita equivalence on weight spaces is a bit subtle in the latter case.
Let $\la \in \tilde{\ft}$ be a preimage of $\nu $, let $W^{\la}$ be its
stabilizer in $W$ and $W_L^{\la}$ its stabilizer in $W_L$.

\begin{Lemma}\label{lem:change-Weyl} Let $ \lambda, \nu, \gamma $ correspond as above.  Let $ M $ be an $\CB(G,N)$-module, and $M^P$ the
	corresponding $\PCB(G,N)$-module.
	If $W^\la= W_L^\la$, then ${}^P\Wei_{\nu}(M^P)=\Wei_{\gamma}(M)$; more generally,
	we have a functorial isomorphism
	${}^P\Wei_{\nu}(M^P)\cong \Wei_{\gamma}(M)^{\oplus k } $, where $ k = [W^\la: W^{\la}_L]$.
\end{Lemma}

\begin{proof}
	Let $e'({\la})$ be the symmetrizing idempotent for $W^\la$ and
	$e'_L({\la})$ be the symmetrizing idempotent for $W^\la_L$.  Let
	$M^B$ be the corresponding module over $\CB^B$.  By
	\cite[Lem. 3.2]{WebGT}, we have
	that \[\Wei_{\gamma}(M)=e'({\la})\cdot {}^B\Wei_{\lambda}(M^B)\qquad
	\Wei_{\nu}^P(M^P)=e'_L({\la})\cdot {}^B\Wei_{\lambda}(M^B)\] and
	${}^B\Wei_{\lambda}(M^B)$ is a free module over $\C[W^\la]$.  Since
	$e'_L({\la})\C[W^\la]\cong \C^k$, the result follows.
\end{proof}

\subsubsection{Restriction on weight spaces}
The inclusion of $\CB(L,N)$ into $\PCB(G,N)$ from Proposition \ref{th:inclusion} identifies their Gelfand-Tsetlin subalgebras $  \HBG{\tL} $. In particular, restriction from $\PCB(G,N)$ to $\CB(L,N)$ leaves weight spaces unchanged.

\subsubsection{Inverting $ r_\xi$ and weight spaces}
Finally, we wish to consider the effect of inverting $r_\xi$. This can be computed separately on the summands of any
decomposition into $\C[r_{\xi}]$-submodules.  A Gelfand-Tsetlin module has a
natural such decomposition, given by the sum of weight spaces in a
single $\Z\xi$-coset.

Let $ \nu \in \tilde{\ft}/W_L$. Note that $ \nu + \xi $ is well-defined since $ \xi $ is invariant under the action of $ W_L $.  Let $M $ be a $ \CB(L,N)$ module.  We consider the directed system
\begin{equation} 
	\cdots \overset{r_{\xi}}\longrightarrow \Wei_{\nu+\xi}^{L,N}(M) \overset{r_{\xi}}\longrightarrow \Wei_{\nu}^{L,N}(M) \overset{r_{\xi}}\longrightarrow  \Wei_{\nu-\xi}^{L,N}(M) \overset{r_{\xi}}\longrightarrow\cdots\label{eq:dir-system}
\end{equation}
Let $\overrightarrow{\Wei}^{L,N}_{[\nu]}(M) $ denote the direct limit
$\varinjlim  \Wei^{L,N}_{\nu-k\xi}(M) $ of this system. (Here $[\nu]$ denotes the image of $ \nu $ in $ \left( \tilde \ft / W_L \right) / \Z\xi $.)

\begin{Lemma}\label{lem:weight-direct-limit} Let $ M $  be a Gelfand-Tsetlin $\CB(L,N)$-module.  Then $ \CB(L,N^{\xi}_0) \otimes_{\CB(L,N)} M $ is also a Gelfand-Tsetlin module and for any $ \nu \in \tilde{\ft}/W_L $, we have
	$$
	\Wei^{L,N^\xi_0}_\nu\big(\CB(L,N^{\xi}_0) \otimes_{\CB(L,N)} M\big) \ \cong \ \overrightarrow{\Wei}^{L,N}_{[\nu]}(M)
	$$
	In particular, this functor is exact on the category of Gelfand-Tsetlin modules.
\end{Lemma}

\begin{proof}
	Let $M'= \CB(L,N^{\xi}_0) \otimes_{\CB(L,N)} M $. We have an obvious map
	$p\colon M\to M'$ which leads to maps
	$r_{\xi}^{-k}p\colon \Wei^{L,N}_{\nu-k\xi}(M)\to \Wei^{L,N^\xi_0}_{\nu}(M')$, which are
	compatible with the directed system \eqref{eq:dir-system}.  This
	induces a map $\overrightarrow{\Wei}^{L,N}_{[\nu]}(M)\to  \Wei^{L,N_0^\xi}_{\nu}(M')$.
	
	Since $ \CB(L,N^\xi_0) = \CB(L,N)[r_\xi^{-1}]$, for any $v\in \Wei^{L,N_0^\xi}_{\nu}(M')$, we have
	that $r_\xi^{k}v$ is in the image of $M$ for $k$ sufficiently large.
	This shows that the map $\overrightarrow{\Wei}^{L,N}_{[\nu]}(M)\to
	\Wei^{L,N_0^\xi}_{\nu}(M')$ is surjective.
	
	On the other hand, if there is an element of
	the kernel, it is represented by some $w\in  \Wei^{L,N}_{\nu-k\xi}(M)$ for
	some $k$.  Since this element is killed by $p$, we have that
	$r_{\xi}^{k'}w=0$ for some $k'$.  Thus, $w$ has trivial image in the
	directed limit.  This shows injectivity.
\end{proof}

However, the limit $ \overrightarrow{\Wei}^{L,N}_{[\nu]}(M)$ is already isomorphic to a weight space of $M$, but possibly for a different $\nu$.  To see this, recall that in $ \CB(L,N)$, we have by (\ref{eq: monopole relations}) that
\begin{equation} \begin{aligned} \label{eq:rxi}
		r_{-\xi}  r_\xi & = \left( \prod_{\langle\mu,\xi \rangle >0} \prod_{j = 1}^{\langle \mu, \xi \rangle}( \mu - j) \right)   \left(\prod_{\langle\mu,\xi \rangle <0} \prod_{j = 0}^{-\langle \mu, \xi \rangle-1}( \mu + j )\right) \\
		r_{\xi}  r_{-\xi} & = \left( \prod_{\langle\mu,\xi \rangle <0} \prod_{j = 1}^{-\langle \mu, \xi \rangle}( \mu - j ) \right)   \left(\prod_{\langle\mu,\xi \rangle >0} \prod_{j = 0}^{\langle \mu, \xi \rangle-1}( \mu + j )\right) 
	\end{aligned}
\end{equation}
where the products both range over subsets of the weights $ \mu $ of the representation $ N $, counted with multiplicity.

These formulas motivate the following definition.
\begin{Definition} \label{def:xineg}
	$\lambda \in \tilde{\ft} $ is called {\bf
		$\xi$-negative}, if \begin{enumerate}
		\item there is no weight $ \mu $ of $ N $ such that $ \langle \mu, \xi \rangle>0  $ and  $ \langle \mu, \lambda \rangle$ is a positive integer and
		\item there is no weight $ \mu $ of $ N $ such that $
		\langle \mu, \xi \rangle<0  $ and  $ \langle \mu,
		\lambda \rangle$ is a non-positive integer.
		\item The stabilizer $ W^\lambda $ lies in $W_L$, that is
		$W^{\lambda}=W_L^\lambda $.
	\end{enumerate}
\end{Definition}

Since $ \xi $ is invariant under $ W_L $ and the set of weights of $ N $ is invariant under $ W_L $, the set of $\xi$-negative elements of $ \tilde{\ft} $ is invariant under $W_L $.  Thus, it makes sense to speak of $\xi$-negative elements of $ \tilde{\ft} / W_L $.

\begin{Lemma} \label{le:xineg1}
	Assume that $\nu \in \tilde{\ft} / W_L $ is $\xi$-negative.
	\begin{enumerate}
		\item  For all $ k \in \Z_{\geq 0} $, $r_\xi $ gives an isomorphism of functors
		$$
		\Wei^{L,N}_{\nu - k\xi} \to \Wei^{L,N}_{\nu - (k+1)\xi}
		$$
		\item The natural map $\Wei^{L,N}_\nu \to \overrightarrow{\Wei}^{L,N}_{[\nu]}$ is an isomorphism of functors.
	\end{enumerate}
\end{Lemma}

\begin{proof}
	Let $ M $ be a Gelfand-Tsetlin $\CB(L,N)$-module. Let $ k \in \Z_{\ge 0} $. We consider the composition $ r_{-\xi} r_\xi $ as a linear operator on $ \Wei_{\nu - k\xi}(M)$.  By (\ref{eq:rxi}), the eigenvalues of this operator on this weight space are given by the set
	$$
	\big\{ \langle \mu, \nu \rangle - k\langle \mu, \xi \rangle - j \ :\ \langle\mu, \xi\rangle > 0,\  1 \leq j \leq \langle\mu,\xi\rangle \big\}
	$$
	$$
	\cup\ \big\{ \langle \mu, \nu \rangle - k\langle \mu, \xi \rangle +j \ : \ \langle\mu, \xi\rangle < 0,\ 0 \leq j < - \langle\mu,\xi\rangle\big\}
	$$
	The hypothesis of $ \xi$-negativity ensures that none of these eigenvalues vanish.  Thus, $ r_{-\xi} r_\xi $ is an isomorphism.  Similarly, we see that $ r_\xi r_{-\xi} $ is an isomorphism on $\Wei_{\nu - (k+1)\xi}(M)$.
	
	So we conclude that $ r_\xi $ is an isomorphism.  Then the second part follows immediately.
\end{proof}

\subsubsection{Combined effect of $ res$ on weight spaces}

Combining the results above, we conclude the following.
\begin{Theorem}\label{th:res-weight}
	Consider a Gelfand-Tsetlin $\CB(G,N)$ module $M$.   Let $ \nu \in
	\tilde{\ft}/W_L$ and let $  \gamma \in \tilde{\ft}/W $ denote the image of $ \nu $. Assume that $ \nu $ is $\xi$-negative.
	Then there is a natural isomorphism
	$$\Wei_{\nu}^{L}(\res(M)) \cong \Wei_{\gamma}(M).$$
\end{Theorem}

\begin{Corollary} \label{co:mapWei}
	Let $ \nu, \nu'\in \tilde{\ft}/W_L$ be $ \xi$-negative and let $ \gamma, \gamma' \in \tilde{\ft}/W $ be their images.  There is a morphism
	$$
	\Hom(\Wei_\nu^L, \Wei_{\nu'}^L) \rightarrow \Hom(\Wei_{\gamma}, \Wei_{\gamma'})
	$$
\end{Corollary}

\begin{proof}
	By Theorem \ref{th:res-weight}, we have
	$$
	\Wei_{\nu}^L(\res(M))\cong \Wei_{\gamma}(M) \qquad \Wei_{\nu'}^L(\res(M))\cong \Wei_{\gamma'}(M)
	$$
	Given any $ x \in \Hom(\Wei_\nu^L, \Wei_{\nu'}^L) $, the map $ \Wei_{\nu}^L(\res(M)) \xrightarrow{x} \Wei_{\nu'}^L(\res(M)) $ gives us our desired map $ \Wei_{\gamma}(M) \rightarrow \Wei_{\gamma'}(M)
	$.
\end{proof}

\subsection{Algebraic description of the res functor} \label{se:FSSL}
We will now combine Proposition \ref{thm:GTc-F} with Corollary \ref{co:mapWei} to obtain an algebraic description of the functor res.

Consider a
$\widehat{W}$-orbit $\mathscr{S}\subset \tilde{\ft}$. As before, we let
$\CB(G,N)\dash\GTc_\mathscr{S}$ be the category of Gelfand-Tsetlin modules
supported on the image of this
orbit in $ \tilde{\ft}/W$.  Since $\mathscr{S}$ is closed under addition by $\xi$, Lemma \ref{lem:weight-direct-limit} shows that if $M$ is supported on $\bar{\mathscr{S}}$, then
$\res(M)$ is supported on $\bar{\mathscr{S}}$ as well.  

The set  $\mathscr{S}$ is a finite
union of $ \widehat{W}^L$-orbits.	
We will fix attention on a single
one of these $\widehat{W}^L$-orbits, which we denote $\mathscr{S}^L$.
Let $\res^{\mathscr{S}}_{\mathscr{S}^L}\colon \CB(G,N)\dash\GTc_\mathscr{S}\to
\CB(L,N_0^\xi)\dash\GTc_{\mathscr{S}^L}$ be the functor given by applying $\res$ and then taking the summand supported on $\bar{\mathscr{S}^L}$.

Let
$\mathsf{S},\mathsf{S}^L$ be finite complete sets in $\mathscr{S}, \mathscr{S}^L$
respectively, and without loss of generality, assume that
$\mathsf{S}^L\subset \mathsf{S}$.

Since $ r_\xi $ is invertible in $ \CB(L, N^\xi_0) $ as in Theorem \ref{th:invertrxi}, we have $ \Wei^{L,N_0^\xi}_\nu \cong \Wei^{L, N_0^\xi}_{\nu + \xi} $ for any $ \nu \in \tilde{\ft}/W_L $.  Thus, without loss of generality, we can choose all elements of
$\mathsf{S}^L$ to be $\xi$-negative.

By Proposition \ref{thm:GTc-F}, we
have equivalences
\[\CB(G,N)\dash\GTc_\mathscr{S}\cong \wtAlg(\mathsf{S})\mmod\qquad
\CB(L,N_0^\xi)\dash\GTc_{\mathscr{S}^L}\cong \wtAlg^L(\mathsf{S}^L)\mmod\]

We can define natural $\wtAlg^L(\mathsf{S}^L), \wtAlg(\mathsf{S})$-bimodules corresponding to the restriction and induction functors:
\begin{equation}\label{eq:I-def}
	\wtBim(\mathsf{S}^L, \mathsf{S})=
	\bigoplus_{\substack{\gamma'\in \bar{\mathsf{S}}\\\nu\in
			\bar{\mathsf{S}^L}}}\Hom(\Wei_{\nu}^L\circ \res,\Wei_{\gamma'})\qquad \wtBim(\mathsf{S}, \mathsf{S}^L)=
	\bigoplus_{\substack{\gamma'\in \bar{\mathsf{S}}\\\nu\in
			\bar{\mathsf{S}^L}}}\Hom(\Wei_{\gamma'}, \Wei_{\nu}^L\circ \res)
\end{equation}
where $ \bar{\mathsf S} $ denotes the image of $ \mathsf S $ in $\tilde \ft / W $ and $ \bar{\mathsf S^L} $ denotes the image of $ \mathsf S^L $ in $\tilde \ft / W_L $.

Recall that by construction, each $\nu \in \mathsf{S}^L$ is
$\xi$-negative.  Thus, by Theorem \ref{th:res-weight}, we can choose
an isomorphism $\Wei_{\nu}^L\circ \res\cong \Wei_{\gamma}$ where
$\gamma$ is the image of $\nu$ in $\tilde{\ft}/W$.  This induces a vector space isomorphism  $\wtAlg(\mathsf S, \mathsf S^L)  \cong \wtBim(\mathsf S, \mathsf S^L )$ (see (\ref{eq:defFSS}) for the definition of $\wtAlg(\mathsf S, \mathsf S^L)$).   The right action of $\wtAlg(\mathsf S^L)$ on
$ \wtAlg(\mathsf S, \mathsf S^L) $, and the right action of
$\wtAlg^L(\mathsf S^L) $ on $\wtBim(\mathsf S, \mathsf S^L )$ are related by the
homomorphism of Corollary \ref{co:mapWei}.

Let $e^L\in \wtAlg(\mathsf{S})$ be the idempotent obtained by summing the identities on elements of $\mathsf{S}^L$. In this case,
$e^L\wtAlg(\mathsf{S}) e^L=\wtAlg(\mathsf{S}^L)$ and $\wtBim(\mathsf{S}, \mathsf{S}^L)=e^L\wtAlg(\mathsf{S}) $.  Note that for $ M \in \wtAlg(\mathsf S)\mmod$ we have functorial
isomorphisms
\begin{equation}
	\wtBim(\mathsf{S}, \mathsf{S}^L)
	\otimes_{\wtAlg(\mathsf S)} M\cong e^LM\cong
	\Hom_{\wtAlg(\mathsf{S})}(\wtBim(\mathsf{S}^L, \mathsf{S}),M).\label{eq:both-functors}
\end{equation}

\begin{Theorem}\label{thm:Splus}
	We have a commutative diagram
	\[\tikz[->,thick]{
		\matrix[row sep=12mm,column sep=30mm,ampersand replacement=\&]{
			\node (d) {$\wtAlg(\mathsf{S})\mmod$}; \& \node (e)
			{$ \wtAlg^L(\mathsf{S}^L)\mmod$}; \\
			\node (a) {$\CB(G,N)\dash\GTc_\mathscr{S}$}; \& \node (b)
			{$\CB(L,N_0^\xi)\dash\GTc_{\mathscr{S}^L}$}; \\
		};
		\draw (a) -- (b) node[below,midway]{$\res$};
		\draw (a) -- (d) node[left,midway]{$\Wei_{\mathsf S}$} ;
		\draw (b) -- (e) node[right,midway]{$\Wei^L_{\mathsf S^L}$};
		\draw (d) -- (e) node[above,midway]{$\wtBim(\mathsf{S}, \mathsf{S}^L)
			\otimes_{\wtAlg(\mathsf S)} -$};
	}\]
\end{Theorem}

Note that \eqref{eq:both-functors} allows us to construct left and
right adjoints to $\res$ on the category of GT modules:
\[\tikz[->,thick]{
	\matrix[row sep=12mm,column sep=30mm,ampersand replacement=\&]{
		\node (d) {$\wtAlg(\mathsf{S})\mmod$}; \& \node (e)
		{$ \wtAlg^L(\mathsf{S}^L)\mmod$}; \\
		\node (a) {$\CB(G,N)\dash\GTc_\mathscr{S}$}; \& \node (b)
		{$\CB(L,N_0^\xi)\dash\GTc_{\mathscr{S}^L}$}; \\
	};
	\draw[<-] (a) -- (b) node[below,midway]{$\ind$};
	\draw (a) -- (d) node[left,midway]{$\Wei$} ;
	\draw (b) -- (e) node[right,midway]{$\Wei^L$};
	\draw[<-] (d) -- (e) node[above,midway]{$\wtBim(\mathsf{S}^L, \mathsf{S})
		\otimes_{\wtAlg^L(\mathsf {S}^L)} -$};
}\]
\[\tikz[->,thick]{
	\matrix[row sep=12mm,column sep=30mm,ampersand replacement=\&]{
		\node (d) {$\wtAlg(\mathsf{S})\mmod$}; \& \node (e)
		{$ \wtAlg^L(\mathsf{S}^L)\mmod$}; \\
		\node (a) {$\CB(G,N)\dash\GTc_\mathscr{S}$}; \& \node (b)
		{$\CB(L,N_0^\xi)\dash\GTc_{\mathscr{S}^L}$}; \\
	};
	\draw[<-] (a) -- (b) node[below,midway]{$\operatorname{coind}$};
	\draw (a) -- (d) node[left,midway]{$\Wei$} ;
	\draw (b) -- (e) node[right,midway]{$\Wei^L$};
	\draw[<-] (d) -- (e) node[above,midway]{$\Hom_{\wtAlg^L(\mathsf {S}^L)}(\wtBim(\mathsf{S}, \mathsf{S}^L),
		-)$};
}\]
The right adjoint $\coind$ is well-defined on all modules, based on the definition in Section \ref{sec:res-ind}.  Note that it's not clear that $\ind$ is well-defined for all modules, though it
seems likely that it agrees with $\coind$ for the coweight $-\xi$.
We can also construct $\ind$ from $\coind$ by using duality on the
category of Gelfand-Tsetlin modules.

\subsection{Hamiltonian reduction} \label{sec:Hamiltonian reduction}
In this section we consider part (4) of Theorem \ref{thm:rel}, and the corresponding functors on modules induced by quantum Hamiltonian reduction. 

Let $A$ be an algebra and $b\in A$, and recall that $A \sslash_1 b = \operatorname{End}_A( A / (b-1) A)$. There is a natural $A \sslash_1 b, A$--bimodule structure on $A /  (b-1) A$. In particular, there is a natural functor $A\mmod \rightarrow A\sslash_1 b \mmod$ defined by
$$
M\ \mapsto \ A / (b-1) A \otimes_A M \cong M / (b-1) M,
$$
which has right adjoint $\Hom_{A \sslash_1 b}( A / (b-1) A, -)$.

An instructive example to think about is when $G=\Cx$ and the action on $N$ is trivial (in this case, the algebra $\CB(\Cx,N)$ does not depend on $N$).  By \cite[4(iv)]{BFN}, the resulting algebra is generated by $x=r_{-1},x^{-1}=r_{1}$ and $a$ with the relation $[a,x]=x$, where $ a $ is the equivariant parameter coming from $\HBG{\Cx_{\xi}}$.  We can either think of this algebra as a ring of difference operators on the polynomial ring $\C[a]$, with $x$ acting by translation,
or as differential operators on the torus $\Cx$ with coordinate $x$, with $a=x\frac{\partial}{\partial x}$ (the isomorphism relating these two realizations is called the {\it Mellin transform}).  In this case, the quantum Hamiltonian reduction is simply $\CB(\Cx, N) \sslash_1 x  \cong \C$.

The representation theory of $\CB(\Cx,N)$ is not trivial, but it is simple.  A Gelfand-Tsetlin module $M$ over $\CB(\Cx,N)$ is one on which $a$ acts locally finitely, and can equivalently be thought of as a D-module on $\Cx$ on which $a=x\frac{\partial}{\partial x}$ acts locally finitely; this implies that $M$ corresponds to a regular local system. Thus, a simple module of this type must be a 1-dimensional local system with monodromy $c \in \C^\times$.  This module is isomorphic to $ \CB / \CB(a-m) $ for any $m$ such that $ e^{2\pi i m} = c $.

For a general module $ M$,  $M/(x-1)M$ is the fibre of the local system at $x=1$.  In particular, for a Gelfand-Tsetlin module, the number of simple composition factors is the dimension of this fibre. In order to reconstruct $M$, we need to also remember the action of the monodromy map $\exp(2 \pi i a)$ on this quotient (which is well-defined).

Now, consider the situation of Theorem \ref{th:Reduction}:  Assume that $\C^{\times}_{\xi}\subset G$ is the image of a primitive cocharacter $\xi\colon \C^{\times}\to G$ which acts trivially on a representation $N$.  Recall from Theorem \ref{th:Reduction} that in this case 
$\CB(G,N) \sslash_1 r_\xi\ \cong \ \CB(G/\C^\times_\xi, N)$. 

Let $\tilde{\ft}'=\tilde{\ft}/\C \xi$ where we use $\xi$ to denote the derivative of the cocharacter $\xi$. That is, $\tilde \ft'$ is the Lie algebra of a maximal torus of $\tilde{G}'=\tilde G/\C^{\times}_{\xi}$.  Thus, we have a map $\C[ \tilde{\ft}']^W\subset \C[\tilde{\ft}]^W\to \CB(G,N)$;  the composed map commutes with $r_{\xi}$, so  $\C[\tilde{\ft}']^W$ maps into the Hamiltonian reduction.  This induces the usual map $\HBG{G'}\cong\C[ \tilde{\ft}']^W\to \CB(G',N)$.  

We have a Hamiltonian reduction functor on left modules:
$$M \mapsto M/(r_{\xi}-1)M$$
\begin{Proposition}
	\label{prop: qhr and gt}
	Let $ M \in \CB(G,N)\dash\GTc $.  Then $ M / (r_\xi - 1) M $ is a Gelfand-Tsetlin module for $\CB(G', N) $, and for $ \gamma' \in \tilde \ft' /W$ we have
	$$
	\Wei_{\gamma'}\big(M/(r_{\xi}-1)M\big)\ \cong \  \bigoplus_{[\gamma]\in (\gamma'+\C\xi)/\Z\xi}  \Wei_{\gamma}(M)
	$$
	where the direct sum ranges over a set of representatives modulo $ \Z\xi $ for the preimage of $ \gamma' $ in $ \tilde \ft/W $.  
\end{Proposition}

For $\gamma$ a $\Z\xi$-coset representative, the weight space
$\Wei_{\gamma}(M)$ can be written more canonically as the limit of the directed system (\ref{eq:dir-system}). Since $\C^{\times}_{\xi}$ acts trivially on $N$, every map of this directed system will be an isomorphism.

\begin{proof}
	For any $\gamma \in \ft/W$ the weight spaces $\Wei_{\gamma}(M)$ and $\Wei_{\gamma+k\xi}(M)$ are canonically isomorphic by $r_{\pm k\xi}$. 
	
	Since all weight spaces of $M$ are finite dimensional, it follows that $\bigoplus_{k \in \Z} \Wei_{\gamma+k\xi}(M)$ is a free module of finite rank over the Laurent polynomial ring $\C[r_{\xi},r_{-\xi}]$, which is freely generated by $\Wei_\gamma(M)$.  Since $M$ is finitely generated, the support of $M$ is a finite union of orbits of the extended affine Weyl group of $G$, and so there are only finitely many cosets of $\Z\xi$ in any given coset of $\C\xi$ that lie in the support of $M$.
	
	Given $\gamma' \in \tilde{\ft}' / W$, we have that $M'_{\gamma'}:=\bigoplus_{\gamma\in \gamma'+\C\xi}\Wei_{\gamma}(M)$  is also a free module of finite rank over $\C[r_{\xi},r_{-\xi}]$: there are only finitely many cosets $[\gamma] \in (\gamma'+\C\xi)/\Z\xi$ for which $\Wei_{\gamma}(M) \neq 0$, and we may apply the above argument on each coset. We can choose free generators for $M'_{\gamma'}$ by fixing one representative $\gamma$ from each such coset $[\gamma]$, and thus
	\begin{equation*}
		M / (r_\xi -1)M = \bigoplus_{\gamma'} M'_{\gamma'} / (r_\xi-1) M'_{\gamma'} = \bigoplus_{\gamma'} \Big(\bigoplus_{[\gamma]\in (\gamma'+\C\xi)/\Z\xi}  \Wei_\gamma(M)\Big).  \qedhere 
	\end{equation*}

\end{proof}

Since $\xi$ is primitive, we can find $ a \in \tilde \ft^*_{\Z} $
with $ \xi(a) = 1 $.  
Choose a logarithm map $ \log : \Cx \rightarrow \C $ inverse to $ m \mapsto e^{2 \pi i m} $. For simplicity, we can uniquely fix this by requiring its image to lie in $[0,1)+i\R$.

The adjoint action of $a$ on $\CB(G,N)$
integrates to a $ \C^\times$ action and thus has only
integer eigenvalues.  Thus the module $M$ is the direct sum of submodules
\begin{equation}\label{eq:c-splitting}
	M_c := \bigoplus_{k\in \Z}\Wei^a_{\log c+k}(M) 
\end{equation}
where $ \Wei^a_x(M) $ denotes the generalized $ x$-eigenspace for the action of $ a $ on $ M $, and where $ c $ ranges over $ \Cx$.

We use $ a $ to split $ \tilde \ft = \tilde \ft' \oplus \C \xi $.  
For $ c \in \Cx $, define 
$$ (M/(r_\xi - 1)M))_c := \Wei^a_{\log c}(M) = \bigoplus_{\gamma' \in \tilde \ft'/W} \Wei_{\gamma' + (\log c) \xi}(M)
$$  Via Proposition
\ref{prop: qhr and gt}, we can identify $ (M/(r_\xi - 1)M))_c $ with a
subspace of $ M/(r_\xi-1)M $.

This gives a decomposition 
\begin{equation} \label{eq:monodromy decomposition}
	M/(r_\xi-1)M = \bigoplus_{c \in \Cx} (M/(r_\xi - 1)M))_c  =\bigoplus_{c \in \Cx}M_c/(r_{\xi}-1)M_c
\end{equation} 
into $ \CB(G',N)$-submodules by reducing the decomposition \eqref{eq:c-splitting}.

To connect this to the $D$-modules on the punctured line described above, note that $ a, r_\xi^{\pm} $ generate a copy of $ D(\Cx) $ inside of $ \CB(G,N) $.  The module  $ M $ can be viewed as a $D$-module on $ \Cx $ via this isomorphism.  Then the left hand side of (\ref{eq:monodromy decomposition}) is the fibre of this $D$-module at $ 1 $ and the right hand side is the  decomposition of the fibre (or equivalently, the nearby cycles at the origin) into generalized eigenspaces for the monodromy around the origin.

\section{Geometric description of weight modules and functors}

\subsection{Recollection on earlier results} \label{se:recollect}
By Proposition \ref{thm:GTc-F}, spaces of natural transformations between weight functors control the category of Gelfand-Tsetlin modules.
In the papers \cite{WebGT,Webdual}, the third author gave a geometric description of these spaces.  We will now recall this description; it will be phrased as an equivalence of categories.  We begin by defining the two categories involved.

In this section we will only work with integral weights.  Here $ \tilde{\ft}_\Z = \Hom(\C^\times, \tilde T) \subset \tilde \ft $.

\subsubsection{Category of weight functors}
\begin{Definition} \label{def:Acat}
	We call a Gelfand-Tsetlin module over $\CB(G,N)$ or $\CB^B(G,N)$ {\bf integral} if its support lies in $\tilde{\ft}_\Z$.  
	
	Let $ \widehat{\mathscr A}_\Z(G,N) $ be the category with objects $ \tilde{\ft}_\Z / W$ and morphisms given by
	$$
	\Hom_{\widehat{\mathscr A}_\Z(G,N)}(\gamma, \gamma') = \Hom(\Wei_\gamma, \Wei_{\gamma'})
	$$
	Similarly, let $\widehat{\mathscr A}_\Z^B(G,N) $ be category with objects $ \tilde{\ft}_\Z  $ and morphisms given by
	$$
	\Hom_{\widehat{\mathscr A}_\Z^B(G,N)}(\lambda, \lambda') = \Hom({}^B\Wei_\lambda, {}^B\Wei_{\lambda'})
	$$
	where as before, ${}^B \Wei_\lambda $ is a weight functor on the category of $ \mathcal A^B(G,N)$-modules.
\end{Definition}

\begin{Remark} \label{re:Karoubi}
	The reader might naturally wonder about the relationship between $\widehat{\mathscr A}_\Z(G,N) $ and $\widehat{\mathscr A}_\Z^B(G,N)$.  They are not equivalent: for example, in the pure case
	where $N=0$, if $\gamma$ is a $W$-orbit with a single element, the
	endomorphism algebra of $\gamma$ in
	$\widehat{\mathscr A}(G,N) $ has a single 1-dimensional discrete irreducible representation,
	and no object in  $\widehat{\mathscr A}^B(G,N) $ has this
	property.  On the other hand, by Lemma \ref{lem:change-Weyl},  there
	is an equivalence between the Karoubian envelopes of these
	categories, sending $\lambda \in \tilde{\ft}_\Z  $ to the direct sum of $W^\la$ copies of
	its image in $\tilde{\ft}_\Z/ W$. The inverse functor thus
	sends an element of $\tilde{\ft}/W$ corresponding to a
	non-free orbit to the image of the symmetrizing idempotent $e'(\lambda)$ in $W^\la$
	acting on a pre-image $\la$.
\end{Remark}

\subsubsection{Steinberg category}
Next, we will define certain Steinberg-type varieties which will be the building blocks of our second category.

First, given coweights $ \gamma, \gamma' \in \tilde{\ft}_\Z /W $, as before, let $ \lambda, \lambda' $ be the antidominant lifts of $ \gamma, \gamma' $.  As above, $N^{\lambda}_{\leq}$ is the subspace in $N$ on which $\lambda$
has non-positive weight, and let $P_\lambda \subset G$ be the parabolic subgroup on whose Lie
algebra $\lambda$ has non-positive weight (and similarly for $\lambda'$).  Let:
\begin{align}
\label{Ygamma-def}	Y_\gamma&=(G\times N^{\lambda}_{\leq})/{P}_\lambda\cong \{(gP_{\lambda},n) \mid n\in gN^{\lambda}_{\leq}\}\\
    {}_\gamma\pStein_{{\gamma'}}&=Y_\gamma\times_{N} Y_{\gamma'}=\{(g_1 P_\lambda,g_2P_{\lambda'}, n) \mid n\in
	g_1N^{\lambda}_{\leq}\cap g_2N_{\leq}^{\lambda'}\}.
\end{align}
Let $\widehat{H}^{G}(  {}_\gamma\pStein_{{\gamma'}})$ denote the
completion of the equivariant Borel-Moore homology of $
{}_\gamma\pStein_{{\gamma'}}$ with respect to its grading.

We will also need a full flag version of this construction.  Let:
\begin{equation}\label{eq:Xdef}
	X_\lambda =(G\times N^{\lambda}_{\leq})/{B} \qquad
	{}_\lambda\Stein_{{\lambda'}}=X_\lambda\times_{N} X_{\lambda'}=\{(g_1 B,g_2 B, n) \mid n\in
	g_1N^{\lambda}_{\leq}\cap g_2N_{\leq}^{\lambda'}\}.
\end{equation}
Note that $ B \subset P_\lambda $ and if $ \lambda $ is the antidominant lift of $ \gamma$, we have a natural morphism $ X_\lambda \rightarrow Y_\gamma $ which is a $ P_\lambda / B $ bundle.  

We can easily extend the definition of these spaces when $\lambda, \lambda'$ are not anti-dominant.  To this end, we define $ B_\la $ to be the Borel subgroup whose Lie algebra consists of those root spaces $ \fg_\alpha $, for all $ \alpha $ such that either $ \langle \lambda, \alpha \rangle < 0 $, or $ \langle \lambda, \alpha \rangle = 0 $ and $ \alpha $ is negative.  Then $N^{\lambda}_{\leq}$ will be invariant under $B_{\la}$ and we define $X_\la=(G\times N^{\lambda}_{\leq})/{B_{\la}}$.   Note that this space would be the same for any other Borel contained in $P_{\la}$; any two such Borels are conjugate in $P_\la$, and any element in $P_{\la}$ conjugating between them induces an isomorphism by $(g,n)\mapsto (gp^{-1},pn)$.  More generally, this assignment of spaces to coweights is equivariant for the action of the Weyl group.  Even though $wB_{\la}w^{-1}\neq B_{w\la}$ in some cases, we still have $wB_{\la}w^{-1}\subset P_{w\la}$, and so $X_{\la}\cong X_{w\la}$.  Finally, we extend the definition \eqref{eq:Xdef} of ${}_\lambda\Stein_{{\lambda'}}$ to the case of arbitrary $ \lambda, \lambda' \in \tilde{\ft}_\Z $.

Let $\widehat{H}^{G}(  {}_\lambda\Stein_{{\lambda'}})$ denote the
completion of the equivariant Borel-Moore homology of $
{}_\lambda\Stein_{{\lambda'}}$ with respect to its grading.

\begin{Remark}
	When there is an ambiguity, we will write $ Y_\gamma^{G,N}$ etc. to keep track of the group $ G $ and representation $ N $.
\end{Remark}

\begin{Definition}
	Let $ \widehat{\mathscr X}(G,N) $ be the category with objects $ \tilde{\ft}_\Z/W $ and morphisms
	$$\Hom_{\widehat{\mathscr X}(G,N)}(\gamma, \gamma') = \widehat{H}^{G}(
	{}_\gamma\pStein_{{\gamma'}})$$
	Similarly, let  $ \widehat{\mathscr X}^B(G,N) $ be the category with objects $ \tilde{\ft}_\Z $ and morphisms
	$$\Hom_{\widehat{\mathscr X}^B(G,N)}(\lambda, \lambda') = \widehat{H}^{G}(  {}_\lambda\Stein_{{\lambda'}})$$
	Composition in these categories is defined by convolution, as defined
	in \cite[(2.7.9)]{CG97}, viewing $
	{}_\lambda\Stein_{{\lambda'}}\subset X_\lambda\times X_{\lambda'}$ (in
	the notation of \cite{CG97}, $M_1=X_\lambda, M_2=X_{\lambda'}$). The
	associativity of the convolution product is given by \cite[\S
	2.7.18]{CG97}.
\end{Definition}

The category $ \widehat{\mathscr X}^B(G,N) $ is a variation of the
Steinberg category defined in  \cite[\S 2.4]{Webdual}. The definition from \cite{Webdual} involves assigning a subspace in $N$ to each element of  $
\tilde{\ft}_\Z $, which we have done here by the
assignment $\la\mapsto N^{\lambda}_{\leq}$; in \cite[\S 2.4]{Webdual},
this subspace is encoded as a sign sequence on a basis of $N$
representing which basis vectors lie in $ N^{\lambda}_{\leq}$ and
which do not.

\subsubsection{The equivalences}
Following \cite[Def 4.2]{Webdual}, we will now construct equivalences between the weight functor categories and the Steinberg categories.  Our strategy will be to begin with $ \widehat{\mathscr X}(T,N), \ \widehat{\mathscr A}_\Z(T,N)$, then study $ \widehat{\mathscr X}^B(G,N),\  \widehat{\mathscr A}_\Z^B(G,N)$, and finally $\widehat{\mathscr X}(G,N), \  \widehat{\mathscr A}_\Z(G,N)$.

Recall from Section \ref{section: abelian theories and monopoles}, that the algebra $\CBs{\varphi}(T,N)$ is generated over $\HBG{{\tilde T}}$ by the elements $r_{\nu}$ (for $ \nu \in \ft_\Z$), with relations given by (\ref{eq: monopole relations}).

For $ \la, \la' \in \ft_\Z$, we define $ \Phi_0(\la, \la')  \in  \HBG{T} $ by the following formula, where the product ranges over the weights $ \mu $ of the representation $ N $, counted with multiplicity: 	
\begin{equation}\label{eq:Phi0}
	\Phi_0(\la,\la')	 = \prod_\mu \prod_{\substack{j = 1,\dots, -\langle \mu,  \la-\la'\rangle\\j\neq\langle\mu,\la' \rangle }}(  \mu - j ).
\end{equation}
Note that the polynomial $\Phi_0(\la,\la')$ acts invertibly on the functor $\Wei_{\la'}$, so we can define morphisms in $\widehat{\mathscr A}_\Z(T,N)$ by \[\mathbbm{w}(\la,\la')=\frac{1}{\Phi_0(\la,\la')}r_{\la-\la'}\colon \Wei_{\la}\to \Wei_{\la'}.\]

Note that $ {}_{\la'}\Stein^{T,N}_{{\lambda}} $ is a vector space.  We also let  $\mathbbm{w}(\la,\la')$ denote its fundamental class  $[{}_{\la'}\Stein^{T,N}_{{\lambda}}] $, which is a morphism in $\widehat{\mathscr X}(T,N)$. 

A special case of \cite[Thm 4.3]{Webdual} is that:
\begin{Theorem} \label{th:XAT}
	There is an equivalence $ \mathsf E : \widehat{\mathscr X}(T,N)\cong  \widehat{\mathscr A}_\Z(T,N)$ which is the identity on objects.  On morphisms, this functor  sends an element of $\ft^*\subset \HBG{{T}}$ to the nilpotent part of the action of the same element in $\widehat{\mathscr A}_\Z(T,N)$, and 
	takes $ \mathbbm{w}(\la,\la') $ to the same-named morphism.
\end{Theorem}

From Proposition \ref{th:inclusion}, we have an inclusion $ \CBs{\varphi}(T,N)
\rightarrow \CBs{\varphi}^B(G,N)$.  This leads to a functor $ \widehat{\mathscr
	A}_\Z(T,N)\to \widehat{\mathscr A}_\Z^B(G,N)$ which is the identity
on objects.  Describing the additional generators needed to generate
$\CBs{\varphi}^B(G,N)$ is tricky when working purely with the algebra $\CBs{\varphi}$;
this process is simplified by considering the extended category
introduced in \cite[\S 3]{Webdual}.  Here we proceed a little differently.

\nc{\Hfrac}{\mathbb{L}}
Let $\Hfrac$ denote the fraction field of $\HBG{T}$ and $\widehat{\Hfrac}$ the fraction field of the completion $\widehat{\HBG{T}}$.  Recall from \cite[Prop. 4.2]{WebGT} that $\CBs{\varphi}(T,N)$ is a principal Galois order inside the skew group algebra $\Hfrac\rtimes \ft_{\Z}$ for the usual action of $\ft_{\Z}$ on $\Hfrac$ by translations on $\ft$. Similarly, $\CBs{\varphi}(G,N)$ is a principal Galois order inside the Weyl invariants of $(\Hfrac\rtimes \ft_{\Z})^W$.   The algebra $\CBs{\varphi}^B(G,N)$ is the corresponding flag order in the skew group algebra $\Hfrac\rtimes\widehat{W}$.  That is, in particular, these inclusions induce isomorphisms: 
\[\Hfrac\otimes_{\HBG{T}}\CBs{\varphi}(T,N)\cong \Hfrac\rtimes \ft_{\Z},\qquad \Hfrac^W\otimes_{\HBG{G}}\CBs{\varphi}(G,N)\cong (\Hfrac\rtimes \ft_{\Z})^W,\qquad \Hfrac\otimes_{\HBG{T}}\CBs{\varphi}^B(G,N)\cong \Hfrac\rtimes\widehat{W}.\]

By \cite[Lem. 2.11(4)]{WebGT}, the Hom space \[\Hom_{\widehat{\mathscr A}_\Z^B(G,N)}(\nu, \nu')=\CBs{\varphi}^B/(\CBs{\varphi}^B\mathfrak{m}_{\nu}^N+\mathfrak{m}_{\nu'}^N
	\CBs{\varphi}^B),\] is isomorphic to 
\begin{equation}\label{eq:Fhat}
	\widehat{\HBG{{T}}}\otimes_{\HBG{{T}}}\big \{\sum_{w\in \widehat{W}}a_ww\in \CBs{\varphi}^B(G,N)\subset \Hfrac\rtimes\widehat{W}\mid a_w=0 \text{ if }w\nu\neq \nu'\big\}
\end{equation} 

Applying the same result with $G=B=T$ and the product decomposition $\widehat{W}=\ft_{\Z}\cdot W$, we obtain an isomorphism:
\begin{equation}\label{eq:A-localization}
	\widehat{\Hfrac}\otimes_{\widehat{\HBG{T}}} \Hom_{\widehat{\mathscr A}_\Z^B(G,N)}(\nu, \nu')\cong\bigoplus_{w\in W} 	\widehat{\Hfrac}\otimes_{\widehat{\HBG{T}}} \Hom_{\widehat{\mathscr A}_\Z(T,N)}(w\nu, \nu')\cdot w
\end{equation}
On the other hand, we also have a functor $\widehat{\mathscr X}(T,N)\to \widehat{\mathscr X}^B(G,N)$, which is the identity on objects and which acts on morphisms by 
\begin{equation}\label{eq:G-push-saturation}
    \widehat{H}^{T}(  {}_\lambda\Stein^{T,N}_{{\lambda'}}) \rightarrow\widehat{H}^{B}(  {}_\lambda\Stein^{G,N}_{{\lambda'}})\rightarrow \widehat{H}^{G}(  {}_\lambda\Stein^{G,N}_{{\lambda'}}),
\end{equation} the composition of pushforward in $T$-equivariant homology, followed by $G$-saturation, the map $\widehat{H}^{B}(X)\cong \widehat{H}^{G}((G\times X)/B)\to \widehat{H}^{G}(X)$ which ``averages'' $B$-equivariant cycles under $G$.    

We can write
\begin{align*}
    \Hom_{\widehat{\mathscr X}_\Z^B(G,N)}(\nu, \nu')&\cong \widehat{H}^{B}(\{(g B, n) \mid n\in
N^{\lambda}_{\leq}\cap gN_{\leq}^{\lambda'}\} )\\
&\cong \widehat{H}^{T}(\{(g B, n) \mid n\in
N^{\lambda}_{\leq}\cap gN_{\leq}^{\lambda'}\} )
\end{align*}
The fixed points of $T$ on the space are given by $\{(wB,n)\mid w\in W, n\in N^T \}$.   
Applying localization in $T$-equivariant Borel-Moore homology, we find that
we have a natural isomorphism:
\begin{equation}\label{eq:X-localization}
	\widehat{\Hfrac}\otimes_{\widehat{\HBG{T}}} \Hom_{\widehat{\mathscr X}_\Z^B(G,N)}(\nu, \nu')\cong\bigoplus_{w\in W} 	\widehat{\Hfrac}\otimes_{\widehat{\HBG{T}}} \Hom_{\widehat{\mathscr X}_\Z(T,N)}(w\nu, \nu')\cdot w
\end{equation}

By \cite[Thm 4.3]{Webdual}, we have the following result.
\begin{Theorem} \label{th:XA}
	There is an equivalence $\mathsf{E} \colon \widehat{\mathscr X}^B(G,N)\rightarrow  \widehat{\mathscr A}_\Z^B(G,N) $ compatible with the isomorphisms \eqref{eq:A-localization} and \eqref{eq:X-localization}, making the following diagram commute
	$$				\begin{tikzcd}
		\widehat{\mathscr X}(T,N) \arrow{r}  \arrow{d} &
		\widehat{\mathscr A}_\Z(T,N) \arrow{d} \\
		\widehat{\mathscr X}^B(G,N)
		\arrow{r}  &  \widehat{\mathscr A}_\Z^B(G,N)
	\end{tikzcd}
	$$		
\end{Theorem}

Let $ \la, \la' $ be the antidominant lifts of $ \gamma, \gamma' \in \tilde{\ft}_\Z / W $. The $ P_\lambda/ B $ fibre bundle $X_\la\to Y_{\gamma}$ implies that  the convolution algebra $H^{G}(
X_\la\times_{Y_{\gamma}}X_\la)$ is a copy of the nilHecke algebra
for $W^\lambda $.   By \cite[Th. 8.6.7]{CG97}, this algebra acts on  the pushforward of the constant sheaf from $X_\la$.  Since the nilHecke algebra is a matrix algebra on the commutative algebra $H^*_G(pt)$, this shows that this pushforward is a sum of $\# W$ copies of the constant
sheaf on $ Y_{\gamma}$.  Thus, any primitive idempotent in this nilHecke algebra (in
particular, the symmetrizing idempotent $e'({\la})\in \C W^{\la}$) gives this constant sheaf.
This shows that:
\begin{equation}
	\widehat{H}^{{L}}( {}_\gamma\pStein_{{\gamma'}})
	\cong e'({\la})
	\widehat{H}^{{L}}({}_{\la}\Stein_{{\lambda'}}
	)e'({\la'})\label{eq:Stein-projection}
\end{equation}

By \cite[Lem. 2.8]{WebGT} (also discussed in the proof
of Lemma \ref{lem:change-Weyl}), we have a similar formula
\begin{equation}
	\Hom(\Wei_\gamma, \Wei_{\gamma'})\cong
	e'({\la})\Hom({}^B\Wei_\lambda, {}^B\Wei_{\lambda'})e'({\la'}).\label{eq:nat-projection}
\end{equation}
Thus, comparing equations \eqref{eq:Stein-projection} and
\eqref{eq:nat-projection}, we have that:
\begin{Corollary} \label{co:XA}
	There is an equivalence $ \mathsf{E}\colon \widehat{\mathscr X}(G,N) \rightarrow \widehat{\mathscr A}_\Z(G,N) $.
\end{Corollary}

\begin{Remark}
	In \cite[proof of Theorem 4.4]{WebGT}, following a suggestion of Nakajima, the third author gave a sketch proof of Theorem \ref{th:XA} and Corollary \ref{co:XA} using abelianization in equivariant homology.
\end{Remark}

\begin{Remark}
	The effect of requiring $Z$-semisimplicity is encapsulated very cleanly in this
	theorem.  If we consider the weight functors $ \Wei^f_\gamma$ on the category of all modules (rather than restricting to $ Z$-semisimple ones) we obtain a different category, temporarily denoted $\widehat{\mathscr A}_\Z(G,N)^f$, whose morphism spaces are given by $$
 Hom (\Wei^f_\gamma, \Wei^f_{\gamma'}) = \varprojlim
		\CB/(\CB\mathfrak{m}_\gamma^N+\mathfrak{m}_{\gamma'}^N
		\CB)$$  A simple Gelfand-Tsetlin module over $\CB$ will factor through one of the quotients $\CB_{\varphi}$, but there can be extensions of such modules where the center $Z$ acts non-semisimply, and thus don't factor through any such quotient.
  
  In order to fix Corollary \ref{co:XA} to work in this setting, we define $ \widehat{\mathscr X}(G,N)^f $ by working with $\tG$-equivariant cohomology instead of $G$-equivariant so that
$$  \Hom_{\widehat{\mathscr X}(G,N)^f}(\gamma, \gamma') = \widehat{H}^{\tG}(
	{}_\gamma\pStein_{{\gamma'}})$$
   
  With these definitions, essentially the same argument gives us an equivalence of categories
  $$
  \widehat{\mathscr X}(G,N)^f \rightarrow \widehat{\mathscr A}_\Z(G,N)^f
  $$
	 The additional equivariant parameters for the larger group $\tG$
	capture the action of the nilpotent part of $Z$, which as we mentioned above might be non-trivial.
\end{Remark}

\subsection{Geometric viewpoint on parabolic restriction}
\label{sec:geometric-viewpoint}
Now, let us consider how to understand Corollary \ref{co:mapWei} in the context of the geometric description provided by
Corollary \ref{co:XA}.

Let $ \nu, \nu' \in \tilde{\ft}_\Z/W_L $ be $ \xi$-negative, and let $ \gamma, \gamma' \in \tilde{\ft}_\Z/W $ denote their images.
From Corollary \ref{co:XA}, we have an isomorphism
$$\widehat{H}^{{G}}(  {}_\gamma\pStein_{{\gamma'}}) \xrightarrow{\sim}
\Hom(\Wei_\gamma, \Wei_{\gamma'})
$$

On the other hand, from Corollary \ref{co:mapWei}, we have a morphism
$$
\Hom(\Wei_\nu^L, \Wei_{\nu'}^L) \rightarrow \Hom(\Wei_{\gamma}, \Wei_{\gamma'})
$$
where $ \Wei_\nu^L, \Wei_{\nu'}^L $ denote the weight functors on the category of $ \mathcal A(L, N_0^\xi)$-modules.

We begin with the following observation which follows immediately from Definition \ref{def:xineg}.
\begin{Lemma}\label{lem:xi-neg-flag}
	If $ \nu $ is $\xi$-negative, then $ Y_\nu^{L,N} = Y_\nu^{L, N^\xi_0}\times N^{\xi}_{+}$.  This induces an isomorphism $ {}_\gamma\pStein^{L,N}_{{\gamma'}}\cong {}_\gamma\pStein^{L,N^0_\xi}_{{\gamma'}}\times N^{\xi}_{+}$.
\end{Lemma}

Now, the map $L\rightarrow G$ induces a map $Y_\nu^L \rightarrow Y_\gamma $ and thus we have a map
\begin{equation} \label{eq:G-sat}
	\widehat{H}^{L}(  {}_\nu\pStein_{{\nu'}}^L)\to
	\widehat{H}^{P}(  {}_\gamma\pStein_{{\gamma'}})\to
	\widehat{H}^{G}(  {}_\gamma\pStein_{{\gamma'}})
\end{equation}
via $G$-saturation of $P$-equivariant cycles
after pushforward as in \eqref{eq:G-push-saturation}; note that any parabolic $P$ with Levi $L$ will give the same result.

Equivalently, we can use the identifications
\begin{gather*}
	\widehat{H}^{G}(  {}_\gamma\pStein_{{\gamma'}})\cong \widehat{H}^{{P}_\lambda}( \{(gP_{\lambda'}, n) \mid n\in
	N^{\lambda}_{\leq}\cap gN_{\leq}^{\lambda'}\}) \\
	\widehat{H}^{L}(  {}_\nu\pStein_{{\nu'}}^L)\cong \widehat{H}^{P^L_\lambda}( \{(gP_{\lambda'}^L, n) \mid n\in
	N^{\lambda}_{\leq}\cap gN^{\lambda'}_{\leq}\})
\end{gather*}
where $ \lambda, \lambda' $ denote the antidominant lifts of $ \gamma, \gamma'$.
Since $ \nu $ is $ \xi $-negative, $W^\la \subset W_L $, and so $P_\la^L $ and $P_\la $ have the same reductive quotients.  
Thus, the map (\ref{eq:G-sat}) is simply the push-forward in homology for the map
$$
\{(gP_{\lambda'}, n) \mid n\in
N^{\lambda}_{\leq}\cap gN_{\leq}^{\lambda'}\} \rightarrow \{(gP_{\lambda'}^L, n) \mid n\in
N^{\lambda}_{\leq}\cap gN^{\lambda'}_{\leq}\}.
$$

The functor $\mathsf{E}$ is far from the only equivalence between the corresponding categories.  In order to compare the restriction functors with geometry, it is convenient to use a variation of the functor $\mathsf{E}_L\colon \widehat{\mathscr X}(L,N^{\xi}_0)\rightarrow  \widehat{\mathscr A}_\Z(L,N^{\xi}_0)$. 

For each $\la\in \tilde{\ft}_{\Z}$, we define $ \kappa_\la \in \Hfrac$ by the following product over the weights $\mu$ of $N$, taken with multiplicity:
\begin{equation}
	\label{eq: def of kappa}
	\kappa_\la = \frac{\displaystyle \prod_{\substack{\langle \mu, \xi\rangle<0\\ \langle \mu, \la\rangle>0}}\prod_{j = 1}^{\langle \mu, \la\rangle-1}(  \mu - j )}{\displaystyle \prod_{\substack{\langle \mu, \xi\rangle<0\\\langle \mu, \la\rangle\leq 0}}\prod_{j = 0}^{-\langle \mu, \la \rangle-1}(  \mu + j )}
\end{equation}

Note that $\kappa_\la$ acts invertibly on $\Wei_\la^B$; since this assignment is $W_L$-equivariant, this induces a natural transformation on the functor $\Wei_{\nu}$ where $\nu$ is the image of $\la$ in $\tilde{\ft}/W_L$.  
Let \[\mathsf{E}'_L=\kappa_{\nu'}^{-1}\mathsf{E}_L\kappa_{\nu}\colon \widehat{H}^{\tilde{L}}( {}_\nu\pStein_{{\nu'}}^L)\to \Hom(\Wei^L_\nu, \Wei^L_{\nu'})\] be the twist of the functor $\mathsf{E}_L$ from Corollary \ref{co:XA} by the maps $\kappa_\nu$.

\begin{Lemma}\label{lem:EL-kappa}
	Assume that $L$ is abelian.  The functor $\mathsf{E}'_L$ sends 
	\begin{equation}
		\mathbbm{w}(\nu,\nu')\mapsto \frac{1}{\Phi_0'(\nu,\nu')}r_{\nu-\nu'},\qquad \Phi_0'(\nu,\nu'):=\Phi_0^{L,N^{\xi}_0}(\nu,\nu')\prod_{\langle \mu, \xi\rangle<0}\frac{\displaystyle \prod_{\substack{j=1,\dots,-\langle \mu,\nu-\nu' \rangle\\j\neq \langle \mu,\nu' \rangle}}   (\mu - j)}{\displaystyle \prod_{\substack{j=0,\dots, \langle \mu,\nu-\nu' \rangle-1\\j\neq -\langle \mu,\nu' \rangle}} (\mu + j)  }\label{eq:ELprime}  	\end{equation} 
\end{Lemma}

\begin{proof}
	By definition, this functor sends 
	\begin{align}
		\mathbbm{w}(\nu,\nu')&\mapsto \kappa_{\nu'}^{-1}\frac{1}{\Phi_0^{L,N^{\xi}_0}(\nu,\nu')}r_{\nu-\nu'}\kappa_{\nu}\notag\\
		&=\frac{1}{\Phi_0^{L,N^{\xi}_0}(\nu,\nu')}\frac{\displaystyle \prod_{\substack{\langle \mu, \xi\rangle<0\\\langle \mu, \nu'\rangle\leq 0}}\prod_{j = 0}^{-\langle \mu, \nu' \rangle-1}(  \mu + j )}
		{\displaystyle \prod_{\substack{\langle \mu, \xi\rangle<0\\ \langle \mu, \nu'\rangle>0}}\prod_{j = 1}^{\langle \mu, \nu'\rangle-1}(  \mu - j )}\frac{\displaystyle \prod_{\substack{\langle \mu, \xi\rangle<0\\ \langle \mu, \nu\rangle>0}}\prod_{j = 1+\langle \mu, \nu'-\nu\rangle}^{\langle \mu, \nu'\rangle-1}(  \mu - j )}{\displaystyle \prod_{\substack{\langle \mu, \xi\rangle<0\\\langle \mu, \nu\rangle\leq 0}}\prod_{j = \langle \mu, \nu-\nu' \rangle}^{-\langle \mu, \nu' \rangle-1}(  \mu + j )}r_{\nu-\nu'}\notag\end{align}   
	For any weight $ \mu$, its contribution to this product depends on the signs of $\langle \mu, \nu\rangle$ and $ \langle \mu, \nu'\rangle$.  
	
	If $\langle \mu, \nu\rangle>0, \langle \mu, \nu'\rangle \le 0$, then the contribution of $\mu$ to the denominator is trivial, and  its contribution to the numerator becomes 
	\[\prod_{j = 0}^{-\langle \mu, \nu' \rangle-1}(  \mu + j ) \prod_{j = -\langle \mu, \nu'\rangle+1}^{-\langle \mu, \nu'-\nu\rangle-1}(  \mu + j )= \prod_{\substack{j=0,\dots, \langle \mu,\nu-\nu' \rangle-1\\j\neq -\langle \mu,\nu' \rangle}} (\mu + j)\]
	
	Dually, if  $\langle \mu, \nu\rangle \le 0, \langle \mu, \nu'\rangle >0$, the numerator is trivial, and the denominator is 
	\[\prod_{j = 1}^{\langle \mu, \nu'\rangle-1}(  \mu - j )\prod_{j =\langle \mu, \nu' \rangle+1}^{\langle \mu, \nu'-\nu \rangle}(  \mu - j )=\prod_{\substack{j=1,\dots,-\langle \mu,\nu-\nu' \rangle\\j\neq \langle \mu,\nu' \rangle}}   (\mu - j)\]
	On the other hand, if $\langle \mu, \nu\rangle > 0, \langle \mu, \nu'\rangle > 0$, respectively $\langle \mu, \nu\rangle\leq 0, \langle \mu, \nu'\rangle \leq 0$, then these contributions are:
	\[ \frac{\displaystyle \prod_{j = 1+\langle \mu, \nu'-\nu\rangle}^{\langle \mu, \nu'\rangle-1}(  \mu - j )}
	{\displaystyle \prod_{j = 1}^{\langle \mu, \nu'\rangle-1}(  \mu - j )},\qquad \text{respectively} \qquad\frac{\displaystyle \prod_{j = 0}^{-\langle \mu, \nu' \rangle-1}(  \mu + j )}{\displaystyle \prod_{j = \langle \mu, \nu-\nu' \rangle}^{-\langle \mu, \nu' \rangle-1}(  \mu + j )}.\]
	Applying the obvious cancelation gives the desired result.  
\end{proof}

Recall that in Theorem \ref{th:res-weight}, we fixed a natural
isomorphism  of functors
$\rb_{\nu}\colon \Wei_{\nu}^{L}\circ\res \to \Wei_{\gamma}$.  Also recall the map $  \rb : \Hom(\Wei_{\nu}^{L}, \Wei_{\nu'}^{L}) \rightarrow \Hom(\Wei_\gamma, \Wei_{\gamma'}) $ from  
Corollary \ref{co:mapWei}.  By definition, for $ x \in \Hom(\Wei_{\nu}^{L}, \Wei_{\nu'}^{L}) $, we have
$
\rb(x) = \rb_{\nu'} x\rb_{\nu}^{-1} $.

We will use these factors $\kappa_\nu $ to twist these maps, and define
\begin{equation}\label{eq:rb-def}
	\rb'_{\nu}=\rb_{\nu}\kappa_{\nu}^{-1}\qquad \text{and}\qquad \rb'(x)=\rb_{\nu'}' x(\rb_{\nu}')^{-1}.
\end{equation}		
\begin{Theorem}\label{thm:saturation} 	Assume that $ \nu, \nu' \in \tilde{\ft}_\Z/W_L$ are both $\xi$-negative.
	Let $ \gamma, \gamma' $ be their images in $\tilde{\ft}_\Z/W$.  Then we have a commutative diagram
	\begin{equation*}
		\begin{tikzcd}[sep=huge]
			\widehat{H}^{{L}}( {}_\nu\pStein_{{\nu'}}^L) \arrow{r} {\text{(\ref{eq:G-sat})}} \arrow{d}[left]{\mathsf{E}_L} &
			\widehat{H}^{G}(  {}_\gamma\pStein_{{\gamma'}}) \arrow{d}{\mathsf{E}_G}\\
			\Hom(\Wei^L_\nu, \Wei^L_{\nu'})
			\arrow{r}[below]{\rb'}   &  \Hom(\Wei_\gamma, \Wei_{\gamma'})
		\end{tikzcd}
	\end{equation*}
\end{Theorem}

Here $ \Wei^L_\nu $ denotes the weight functor on the category of $ \CB(L, N^\xi_0) $ modules and $ \Wei_\gamma $ denotes the weight functor on the category of $ \CB(G, N) $ modules.

We'll complete this proof in a couple of steps.  First, we consider the abelian case:
\begin{Lemma}\label{lem:abelian-saturation}
	Theorem \ref{thm:saturation} holds in the case where $G=T$ is abelian.
\end{Lemma}
\begin{proof}
	Note that in this case $L=T$ as well, and $ \nu = \gamma, \nu' = \gamma$.  
	We can equivalently show the commutativity of the diagram:
	\begin{equation}\label{eq:GT-square}
		\begin{tikzcd}[sep=huge]
			\widehat{H}^{{T}}( {}_\nu\pStein_{{\nu'}}) \arrow{r} {\text{(\ref{eq:G-sat})}} \arrow{d}[left]{\mathsf{E}'_T} &
			\widehat{H}^{T}(  {}_\nu\pStein_{{\nu'}}) \arrow{d}{\mathsf{E}_T}\\
			\Hom(\Wei_\nu, \Wei_{\nu'}) \arrow{r}[below]{\rb}   &  \Hom(\Wei_\nu, \Wei_{\nu'})
		\end{tikzcd}
	\end{equation}
	The morphism $ \rb $ from Corollary \ref{co:mapWei} comes from the functor $\res\colon \CB(T,N)\mmod\to \CB(T, N^\xi_0)\mmod$ and sends the morphism $r_{\nu-\nu'}$ to $r_{\xi}^{-k}r_{\nu-\nu'+k\xi}$ for $k\gg 0$, which is well-defined by the $\xi$-negativity, and the map (\ref{eq:G-sat})  sends $\mathbbm{w}(\la,\la')$ to $\mathbbm{w}(\la,\la')$.
	In the left-hand square of \eqref{eq:twosquares}, going right and then down sends
	\[\mathbbm{w}(\nu,\nu')\mapsto \frac{1}{\Phi_0^{L,N}(\nu,\nu')}r_{\nu-\nu'},\]
	while by Lemma \ref{lem:EL-kappa}, going down and then
	to the right sends
	\[\mathbbm{w}(\nu,\nu')\mapsto \frac{1}{\Phi_0^{L,N^{\xi}_0}(\nu,\nu')} \prod_{\langle \mu, \xi\rangle<0}\frac{\displaystyle \prod_{\substack{j=0,\dots, \langle \mu,\nu-\nu' \rangle-1\\j\neq \langle \mu,\nu' \rangle}} (\mu + j)  }{\displaystyle \prod_{\substack{j=1,\dots,-\langle \mu,\nu-\nu' \rangle\\j\neq \langle \mu,\nu' \rangle}}  (\mu - j)} r_{\xi}^{-k}r_{\nu-\nu'+k\xi}
	.\]
	Now, we must apply \eqref{eq: inverse monopole relations}:
	\begin{equation*}r_\xi ^{-k}r_{\nu-\nu'+k\xi}=\prod_{\langle\mu,k\xi\rangle > 0 > \langle \mu,\nu-\nu'\rangle} \prod_{j=1}^{d(\langle \mu, k\xi\rangle, \langle \mu,\nu-\nu'\rangle)} \frac{1}{\mu - j } \prod_{\langle\mu,k\xi\rangle < 0 < \langle \mu,\nu-\nu'\rangle} \prod_{j=0}^{d(\langle \mu,k \xi\rangle, \langle \mu,\nu-\nu' \rangle)-1} \frac{1}{\mu + j  }  r_{\nu-\nu'} ,	
	\end{equation*}
	Note that for $k\gg 0$, this is the same as 
	\begin{equation}\label{eq: rewrite 1}r_\xi ^{-k}r_{\nu-\nu'+k\xi}=\prod_{\langle\mu,\xi\rangle > 0 > \langle \mu,\nu-\nu'\rangle} \prod_{j=1}^{- \langle \mu,\nu-\nu'\rangle} \frac{1}{\mu - j } \prod_{\langle\mu,\xi\rangle < 0 < \langle \mu,\nu-\nu'\rangle} \prod_{j=0}^{\langle \mu,\nu-\nu' \rangle-1} \frac{1}{\mu + j  }  r_{\nu-\nu'} ,	
	\end{equation} which is, as promised, independent of $k$.  
	On the other hand, we have that \begin{equation}\label{eq: rewrite 2}\frac{\Phi_0^{L,N^{\xi}_0}(\nu,\nu')}{\Phi_0^{L,N}(\nu,\nu')}=\prod_{\langle\mu,\xi \rangle\neq 0} \prod_{\substack{j = 1,\dots, -\langle \mu,  \la-\la'\rangle\\j\neq\langle\mu,\la' \rangle }}(  \mu - j )\end{equation}
	Combining  equations \eqref{eq: rewrite 1} and \eqref{eq: rewrite 2},  
	we have the commutativity of this square.
\end{proof}

\begin{Lemma} \label{le:CommuteAbelian} Under the hypotheses of \ref{thm:saturation}, we have a commutative diagram
	\begin{equation*}
		\begin{tikzcd}[sep=huge]
			\widehat{H}^{{L}}( {}_\la\Stein_{{\la'}}^L) \arrow{r} {\text{(\ref{eq:G-sat})}} \arrow{d}[left]{\mathsf{E}_L} &
			\widehat{H}^{G}(  {}_\gamma\Stein_{{\gamma'}}) \arrow{d}{\mathsf{E}_G}\\
			\Hom({}^B\Wei^L_\la, {}^B\Wei^L_{\la'})
			\arrow{r}[below]{\rb'}   &  \Hom({}^B\Wei_\gamma, {}^B\Wei_{\gamma'})
		\end{tikzcd}
	\end{equation*}
\end{Lemma}

\begin{proof}
	As in the previous proof, we can equivalently show the commutativity of the diagram:
	\begin{equation*}
		\begin{tikzcd}[sep=huge]
			\widehat{H}^{L}( {}_\la\Stein_{{\la'}}^L) \arrow{r} {\text{(\ref{eq:G-sat})}} \arrow{d}[left]{\mathsf{E}'_L} &
			\widehat{H}^{G}(  {}_\gamma\Stein_{{\gamma'}}) \arrow{d}{\mathsf{E}_G}\\
			\Hom({}^B\Wei^L_\la, {}^B\Wei^L_{\la'}) \arrow{r}[below]{\rb}   &  \Hom({}^B\Wei_\gamma, {}^B\Wei_{\gamma'})
		\end{tikzcd}
	\end{equation*}
	The morphism $ \rb $ from Corollary \ref{co:mapWei} comes from the functor res, which involves passing between modules for the  algebras $ \CBs{\varphi}^B(G,N), \CBs{\varphi}^{B_L}(L,N)$ and $\CBs{\varphi}^{B_L}(L, N^\xi_0)$.

	As we did earlier, let us write $B_L=B_L$ and $ {}^{B_L}\Wei_\la^{L, N} $ for the weight functor on the category of $ \CBs{\varphi}^{B_L}(L, N) $-modules.
	
	In order to prove the commutativity of the diagram, we break it into two squares
	\begin{equation} \label{eq:twosquares}
		\begin{tikzcd}[column sep=36pt]
			\widehat{H}^{L}( {}_\la\Stein_{{\la'}}^{L, N^\xi_0}) \arrow{r}{\beta} \arrow{d}[left] {\mathsf{E}'_L}  & \widehat{H}^{L}( {}_\la\Stein_{{\la'}}^{L, N}) \arrow{r} \arrow{d}{\mathsf{E}_L} &
			\widehat{H}^{G}(
			{}_\gamma\Stein_{{\gamma'}})  \arrow{d}{\mathsf{E}_G}\\
			\Hom( {}^{B_L}\Wei^L_\la,  {}^{B_L}\Wei^L_{\la'}) \arrow{r}[below]{\alpha}  & \Hom( {}^{B_L}\Wei^{L,N}_\la,  {}^{B_L}\Wei^{L,N}_{\la'}) \arrow{r}  &  \Hom( {}^{B}\Wei_\gamma,  {}^{B}\Wei_{\gamma'})
		\end{tikzcd}
	\end{equation}
	The map $\alpha$ at the bottom of the left square is induced by the
	fact that ${}^{B_L}\Wei_{\la}^{L,N}(M)={}^{B_L}\Wei_{\la}^L(M[r_{\xi}^{-1}])$, so
	a natural transformation $f\colon {}^{B_L}\Wei_{\la}\to {}^{B_L}\Wei_{\la'}$ applied to
	$M[r_{\xi}^{-1}]$ induces a map \[\alpha(f)\colon
	{}^{B_L}\Wei_{\la}^{L,N}(M)\to {}^{B_L}\Wei_{\la'}^{L,N}(M).\]  The top isomorphism
	$\beta$ is induced by the isomorphism $\widehat{H}^{L}(
	{}_\la\Stein_{{\la'}}^{L, N^\xi_0}) \cong \widehat{H}^{L}(
	{}_\la\Stein_{{\la'}}^{L, N})$, which comes from Lemma
	\ref{lem:xi-neg-flag}. Note that all the maps in this left-hand square are isomorphisms.  
	
	We will reduce proving the commutativity of the left square to the abelian case.  The ring $\HBG{T}$ acts freely on the morphism spaces in both categories, so we can check the commutativity after tensor product with $\Hfrac$ without loss of generality.  
	Since $\xi$ is $W_L$-invariant by definition, all the maps in the diagram \eqref{eq:GT-square} will be $W_L$-invariant.  
	Thus, by Lemma \ref{le:CommuteAbelian}, the back square of the cube below commutes:
	\tikzcdset{scale cd/.style={every label/.append style={scale=#1},
			cells={nodes={scale=#1}}}}
	\[  \begin{tikzcd}[column sep=-25pt,scale cd=.935]
		& {\displaystyle \bigoplus_{w\in W_L} \widehat{\Hfrac}\otimes_{\widehat{\HBG{T}}} \Hom_{\widehat{\mathscr X}_\Z(T,N^\xi_0)}(w\nu, \nu')\cdot w} \arrow[rr]   \arrow[dd]         &                          & {\displaystyle\bigoplus_{w\in W_L} 	\widehat{\Hfrac}\otimes _{\widehat{\HBG{T}}}\Hom_{\widehat{\mathscr X}_\Z(T,N)}(w\nu, \nu')\cdot w}          \arrow[dd]  \\
		{\Hom_{\widehat{\mathscr X}_\Z(L,N^\xi_0)}(\nu, \nu')} \arrow[ru] \arrow[dd]\arrow[rr,crossing over]            &                          & {\Hom_{\widehat{\mathscr X}_\Z(L,N)}(\nu, \nu')} \arrow[ru]            &               \\
		& {\displaystyle \bigoplus_{w\in W_L} 	\widehat{\Hfrac}\otimes_{\widehat{\HBG{T}}} \Hom_{\widehat{\mathscr A}_\Z(T,N^\xi_0)}(w\nu, \nu')\cdot w} \arrow[rr]  &                          & {\displaystyle\bigoplus_{w\in W_L} 	\widehat{\Hfrac}\otimes_{\widehat{\HBG{T}}} \Hom_{\widehat{\mathscr A}_\Z(T,N)}(w\nu, \nu')\cdot w}  \\
		{\Hom_{\widehat{\mathscr A}_\Z(L,N^\xi_0)}(\nu, \nu')} \arrow[ru]  \arrow[rr] &                          & {\Hom_{\widehat{\mathscr A}_\Z(L,N)}(\nu, \nu')} \arrow[ru] \arrow[from=uu, crossing over] &              
	\end{tikzcd}  \]
	The horizontal maps along the sides are injective by 
	\eqref{eq:A-localization} and \eqref{eq:X-localization}, and the sides commute by the compatibility of these isomorphisms.
	Thus, the front commutes as well, establishing the left-hand square of \eqref{eq:twosquares}.
	
	Now let's concentrate on the right hand
	square of (\ref{eq:twosquares}).  We will establish this commutativity for all $\lambda, \lambda'$, not necessarily antidominant.
	To this end, we define a functor $ \widehat{\mathscr X}^B(L,N) \rightarrow \widehat{\mathscr X}^B(G,N) $ which is the identity on objects and on morphisms is given by the natural map
	$$
	\widehat{H}^{L}(  {}_\lambda\Stein^L_{{\lambda'}}) \rightarrow \widehat{H}^{G}(  {}_\lambda\Stein_{{\lambda'}}).
	$$
	When we compare this with the isomorphism \eqref{eq:X-localization}, we find that it is induced by the usual map $W_L\to W$, since fixed point classes push forward to fixed point classes.  
	
	Next, recall the inclusion of algebras $ \CBs{\varphi}^{B_L}(L,N) \rightarrow \CBs{\varphi}^B(G,N) $.  This map is compatible with the natural map $\widehat{\Hfrac}\rtimes \widehat{W}_L\to \widehat{\Hfrac}\rtimes \widehat{W}$
	and leads to a morphism
	\begin{equation} \label{eq:WeiLWeiG}
		\Hom({}^{B_L}\Wei^{L,N}_\lambda, {}^{B_L}\Wei^{L,N}_{\lambda'}) \rightarrow \Hom({}^B\Wei^{G,N}_\lambda, {}^B\Wei^{G,N}_{\lambda'})
	\end{equation}
	as in Corollary \ref{co:mapWei}, based on the description \eqref{eq:Fhat}.  Thus we obtain a functor $
	\widehat{ \mathscr A}_\Z^{B_L}(L,N) \rightarrow \widehat{\mathscr
		A}_\Z^B(G,N) $ which is the identity of objects and given on
	morphisms by (\ref{eq:WeiLWeiG}). When $G=L=T$, this map is the identity, and after applying the isomorphism \eqref{eq:A-localization}, is induced the inclusion $W_L\hookrightarrow W$.
	and the identity map on $\Hom_{\widehat{\mathscr A}_\Z(T,N)}(w\nu, \nu')$ for each $w\in W_L$.  
	That is, the top and bottom squares of the cube below are commutative:
	\[ \begin{tikzcd}[column sep=-25pt,scale cd=.935]
		& {\displaystyle \bigoplus_{w\in W_L} \widehat{\Hfrac}\otimes_{\widehat{\HBG{T}}} \Hom_{\widehat{\mathscr X}_\Z(T,N)}(w\nu, \nu')\cdot w} \arrow[rr]   \arrow[dd]         &                          & {\displaystyle\bigoplus_{w\in W} 	\widehat{\Hfrac}\otimes _{\widehat{\HBG{T}}}\Hom_{\widehat{\mathscr X}_\Z(T,N)}(w\nu, \nu')\cdot w}          \arrow[dd]  \\
		{\Hom_{\widehat{\mathscr X}_\Z(L,N)}(\nu, \nu')} \arrow[ru] \arrow[dd]\arrow[rr,crossing over]            &                          & {\Hom_{\widehat{\mathscr X}_\Z(G,N)}(\nu, \nu')} \arrow[ru]            &               \\
		& {\displaystyle \bigoplus_{w\in W_L} 	\widehat{\Hfrac}\otimes_{\widehat{\HBG{T}}} \Hom_{\widehat{\mathscr A}_\Z(T,N)}(w\nu, \nu')\cdot w} \arrow[rr]  &                          & {\displaystyle\bigoplus_{w\in W} 	\widehat{\Hfrac}\otimes_{\widehat{\HBG{T}}} \Hom_{\widehat{\mathscr A}_\Z(T,N)}(w\nu, \nu')\cdot w}  \\
		{\Hom_{\widehat{\mathscr A}_\Z(L,N)}(\nu, \nu')} \arrow[ru]  \arrow[rr] &                          & {\Hom_{\widehat{\mathscr A}_\Z(G,N)}(\nu, \nu')} \arrow[ru] \arrow[from=uu, crossing over] &              
	\end{tikzcd}  \]
	The commutation of the back square is immediate.  This implies that the front square is commutative as well.  This is exactly the right-hand square of \eqref{eq:twosquares}, completing the proof.  
\end{proof}
\begin{proof}[Proof of Theorem \ref{thm:saturation}]
	Now, we wish to reduce from the diagram \eqref{eq:twosquares} to the
	corresponding diagram for the spherical algebras:
	\begin{equation} \label{eq:Y-twosquares}
		\begin{tikzcd}[column sep=36pt]
			\widehat{H}^{L}( {}_\nu\pStein_{{\nu'}}^{L, N^\xi_0}) \arrow{r}\arrow{d}[left] {\mathsf{E}'_L}  & \widehat{H}^{L}( {}_\nu\pStein_{{\nu'}}^{L, N}) \arrow{r} \arrow{d}{\mathsf{E}_L} &
			\widehat{H}^{G}(
			{}_\gamma\pStein_{{\gamma'}})  \arrow{d}{\mathsf{E}_G}\\
			\Hom( \Wei^L_\nu, \Wei^L_{\nu'}) \arrow{r}  & \Hom( \Wei^{L,N}_\nu,  \Wei^{L,N}_{\nu'}) \arrow{r}  &  \Hom( \Wei_\gamma,  \Wei_{\gamma'})
		\end{tikzcd}
	\end{equation}
	First, note that the top line of \eqref{eq:twosquares} sends $e'({\nu})$ to $e'({\nu})$, so the
	top line of the righthand square in \eqref{eq:Y-twosquares} is just
	obtained from the top line of \eqref{eq:twosquares} by multiplying
	on the left and right by idempotents as in
	\eqref{eq:Stein-projection}.   On the other hand, the same is true of
	the bottom row, using \eqref{eq:nat-projection}.
\end{proof}

\section{Flavoured KLRW algebras}
\label{sec:flavoured-KLRW}

In this section, we study two variants of KLR algebras which we will later connect to Coulomb branch algebras associated to quiver gauge theories.  
We  briefly discuss weighted KLRW algebras (there is
already a significant literature on them, for example \cite{WebwKLR,WebRou,
	Webalt, BCS,Bowmanmany}).  It will be more convenient
for us to work with a slight variation on these algebras, which we introduce here, and which may
prove to be of independent interest. We call these {\bf flavoured KLRW
	algebras}, since they allow us to more easily incorporate the flavour
parameters of the Coulomb branch, and make a statement more uniform in
these parameters.  These are close analogues of the metric KLRW
algebras which we introduced in \cite{KTWWYO}, but apply to more general
quivers. In Section \ref{sec:induct-restr}, we'll introduce a more general notion of $\longi$-flavoured KLRW algebras which include both the weighted and flavoured algebras defined in this section as special cases.

\subsection{Reminder on weighted KLRW algebras}
\label{sec:remind-weight-klr}

In this
section, we remind the reader of the definition of (reduced) weighted
KLRW algebras.

Let $ \Gamma = (\vertex, \edge)$ be a quiver and let $ \Bw \in
\Z_{\geq 0}^{\vertex} $ be a
dimension vector.  Recall that for an edge $e=i \to j$, we set $t(e)=i$ and $h(e)=j$.  Recall the  Crawley-Boevey
quiver $\Gamma^{\Bw}$, defined in Section \ref{se:quiverdata}.  Its vertex set is $\vertex\sqcup \{\infty\}$ and its edge set is the union of the ``old'' edges $\edge(\Gamma)$ and $w_i$ ``new'' edges oriented from $i$ to $\infty$.  Assume we
have chosen a {\bf weighting} of this graph, i.e. a map $\vartheta\colon \edge(\Gamma^{\Bw})
\to \R,  e \mapsto \vartheta_e$.

\begin{Definition}
	A {\bf loading} is a function  $\boldsymbol{\ell}:\R \to \vertex\cup\{0\}$ which is
	non-zero at finitely many points.  Equivalently, a loading is a finite subset of the real line and a labeling of its elements with vertices of $ \Gamma$.
	
	We call a point $a\in \R$ {\bf corporeal} for the loading $\boldsymbol{\ell}$ if $\boldsymbol{\ell}(a)\neq 0$, and call $\boldsymbol{\ell}(a)$ its label.  We call a point $a\in \R$ {\bf ghostly} for this loading if $\boldsymbol{\ell}(a-\vartheta_e)=h(e)$ for some old edge $e$, and call $e$ its label.  We call the points $\vartheta_e$ for the new edges $e$ {\bf red}; these do not depend on the loading.   
\end{Definition} 

For our purposes, it will be more convenient to organize this information differently.  Let $\lx_1<\lx_2<\cdots <\lx_n$ be the full list of corporeal values $\lx_i\in \R$ (i.e. points where $\boldsymbol{\ell}(\lx_i)\neq 0$).  Let $i_k=\boldsymbol{\ell}(\lx_k)$, giving a sequence $\bi = (i_1,\ldots,i_n)$. This identification of the corporeal points with the integers $\Sc=[1,n]$identifies each ghostly and red point as image of an element of the sets $\Sg$ and $\Sr$ where
\begin{equation}
	\Sg=\big\{(k,e) \in \Sc\times \edge\mid   i_k=h(e)\big\} \qquad \Sr= \big\{(\star,e) \mid e \in \edge(\Gamma^\Bw), \infty=h(e) \big\}\label{eq:Gplus}	
\end{equation}
under the map \[\mathfrak{a}_{\boldsymbol{\ell}}(x)=\begin{cases} \lx_k & x=k\in \Sc\\ \lx_k+\vartheta_e& x=(k,e)\in \Sg\\ \vartheta_e & x=(\star,e)\in \Sr.\end{cases}\]  We can think of this map $\mathfrak{a}_{\boldsymbol{\ell}}\colon \Sall\to \R$ as defined on the union $\Sall = \Sc\cup \Sg \cup \Sr$. We call a loading or loaded sequence {\bf generic} if all the values
of $\mathfrak{a}_{\boldsymbol{\ell}}$ are distinct.  Throughout the rest
of the paper, we will only consider generic loadings. For a generic loading, we can think of $\Sg$ and $\Sr$ as identified with the set of ghostly and red points of the loading.  
Since we think of red points as ghosts attached to the Crawley-Boevey vertex, we will sometimes want to consider the red and ghostly points together in the set $\Sgr=\Sg\cup \Sr$.

For a generic
loading, 
the map $\mathfrak{a}_{\boldsymbol{\ell}}$ induces an order on $\Sall$, which carries all the important
information of the loading.  In particular, for different values of
$\vartheta_e$, different orders are possible.   We'll call the
resulting triple $((i_1,\dots, i_n),(\lx_1,\dots, \lx_n),<)$ a {\bf loaded sequence}.

\begin{Definition}\label{diagram-def}
	A {\bf weighted KLRW} {\bf diagram} is a collection  of
	finitely many oriented smooth curves in
	$\R\times [0,1]$; we call these curves {\bf strands}.  Each strand must have
	one endpoint on $y=0$ and one on $y=1$, at distinct points from the other strands.  These diagrams satisfy the following:
	\begin{itemize}
		\item There is a red strand for each new edge $e$, with endpoints at
		$x=\vartheta_e$, and labeled by $t(e)\in \vertex$.
		\item There are black strands which are not constrained to be vertical, but
		whose projection to the $y$-axis must be a
		diffeomorphism onto $[0,1]$.  These are also labeled with vertices in   $\vertex$, and are allowed to carry a finite number of dots.
		\item For
		every edge $e \in \edge(\Gamma)$ with $h(e)=i$ we add a ``ghost'' of each strand labeled $i$
		shifted $\vartheta_e$ units to the right (or left if $\vartheta_e$ is
		negative).  The ghost is labelled by $e$ and is depicted with a black dotted line.  When we need to contrast the original strands with ghosts,
		we refer to the original strands as ``corporeal.''
	\end{itemize}

	We require that there are no tangencies or triple
	intersection points between any combination of strands (corporeal, red or ghostly), and
	no dots on intersection points.  We consider these diagrams up to isotopy (relative to the top and bottom) which preserves all these conditions.
\end{Definition}

For example, if we have an edge $i\to j$, then the diagram $a$ is a
weighted KLRW diagram, whereas $b$ is not since it has a tangency between a
strand and a ghost, and two triple points: \[a=
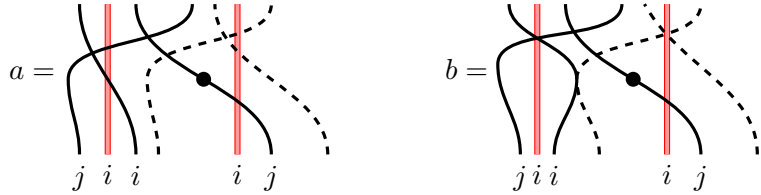
\begin{tikzpicture}[baseline,very thick,xscale=1.5]
	\draw[wei] (-.25,-1) -- node[below,at start]{$i$} (-.25,1);
	\draw[wei] (.9,-1) -- node[below,at start]{$i$} (.9,1);
	\draw (-.5,-1) to[out=90,in=-90] node[below,at start]{$j$} (-.6,0) to[out=90,in=-90](.5,1);
	\draw  (1.2,-1) to[out=90,in=-90] node[below,at start]{$j$}
	node[midway,circle,fill=black,inner sep=2pt]{} (0,1);
	\draw[dashed] (.2,-1) to[out=90,in=-90] (.1,0) to[out=90,in=-90](1.2,1);
	\draw[dashed]  (1.7,-1) to[out=90,in=-90]   (.7,1);
	\draw(0,-1) to[out=90,in=-90]node[below,at start]{$i$} (-.5,1);
\end{tikzpicture}\qquad \qquad
b=
\begin{tikzpicture}[baseline,very thick,xscale=1.5]
	\draw[wei] (-.25,-1) -- node[below,at start]{$i$} (-.25,1);
	\draw[wei] (.9,-1) -- node[below,at start]{$i$} (.9,1);
	\draw (-.4,-1) to[out=90,in=-90] node[below,at start]{$j$} (-.6,.2) to[out=90,in=-90](.5,1);
	\draw  (1.2,-1) to[out=90,in=-90] node[below,at start]{$j$}
	node[midway,circle,fill=black,inner sep=2pt]{} (0,1);
	\draw[dashed] (.3,-1) to[out=90,in=-90] (.1,0) to[out=90,in=-90](1.2,1);
	\draw[dashed]  (1.7,-1) to[out=90,in=-70]   (.7,1);
	\draw(-.1,-1) to[out=90,in=-90] node[below,at start]{$i$}(.1,0) to[out=90,in=-90] (-.5,1);
\end{tikzpicture}
\]

\begin{Definition}\label{def:degree}
	The degree of a weighted KLRW diagram is the sum of:
	\begin{enumerate}
		\item $2$ times the number of dots,
		\item $-2$ times the number of crossings of corporeal strands with the same
		label, and
		\item $1$ for each crossing of a
		corporeal strand with label $i$ and a red or ghostly strand with label $e$
		such that $t(e)=i$.
	\end{enumerate}
\end{Definition}

Reading the positions of the corporeal strands along the lines $y=0$ and $y=1$, we obtain loadings, which we call the
{\bf bottom} and {\bf top}  of the diagram.   There is a notion of {\bf
	composition}  $ab$ of weighted KLRW diagrams $a$ and $b$: this is
given by stacking $a$ on top of $b$ and attempting to join the
bottom of $a$ and top of $b$. If the loadings
from the bottom of $a$ and top of $b$ don't match, then the
composition is not defined and by convention is 0, which is not a
weighted KLRW diagram, just a formal symbol.
This composition rule makes the formal
span of all weighted KLRW diagrams over $\C$ into a graded algebra
{\it $\doubletilde{T}^\vartheta$.} For each loading
$\boldsymbol{\ell}$, we have a straight line diagram $e_{\boldsymbol{\ell}} \in \doubletilde{T}^\vartheta$ where every
horizontal slice is $\boldsymbol{\ell}$, and there are no dots.

We will need the notion of equivalent loadings,
defined in \cite[Def. 2.9]{WebwKLR}. Informally, two loadings are
equivalent when they are isotopic without passing any strand through a
relevant ghostly or red point.
\begin{Definition}\label{def:equivalent}
	Let $\boldsymbol{\ell},\boldsymbol{\ell}'$ be two loadings with associated sequences of nodes $\Bi,\Bi'$.
	We say  $\boldsymbol{\ell},\boldsymbol{\ell}'$ are {\bf
		equivalent} if there is a bijection $\sigma \colon \Sc\to \Sc$
	such that for all $r\in \Sc$ we have that:
	\begin{enumerate}
		\item $i'_{\sigma(r)}=i_r$
		\item  for every $(x,e)\in \Sgr$ such that $i_r=t(e)$, we
		have that
		$r<(x,e)$ if and only if $\sigma(r)<'(\sigma(x),e)$ (by
		convention $\sigma(\star)=\star$).  
	\end{enumerate}
\end{Definition}

\begin{Definition}\label{def-wKLR}
	The {\bf weighted KLRW algebra} $\walg^\vartheta$ is the quotient
	of {\it
		$\doubletilde{T}^\vartheta$} by relations similar to the original
	KLR relations, but with interactions between differently labelled strands
	turned into relations between strands and ghosts.

	We give the list of local relations below.  (Note that $ \vartheta_e < 0 $ in all these pictures.)
	Some care must be used when understanding what it means to apply these
	relations locally.  In each case, the LHS and RHS have a dominant term
	which are related to each other via an isotopy through a disallowed
	diagram with a tangency, triple point or a dot on a crossing.  You can only apply the
	relations if this isotopy avoids tangencies, triple points and dots on crossings
	everywhere else in the diagram; one can always choose isotopy
	representatives sufficiently generic for this to hold.
	
	Another important subtlety is that the relations below are only correct for $\vartheta$ generic.  If, for example, $\vartheta_e=0$ for some edge, causing a corporeal strand and its ghost to coincide, then these relations change (but the non-generic relations can be found by applying a small perturbation to $\vartheta$, and applying the relations below).  For example, the usual bigon relation in a KLR algebra (i.e. \cite[(3.1c)]{KTWWYO}) is a limit of (\ref{strand-bigon}--\ref{ghost-bigon2a}) below, as $\vartheta_e\to 0$ for all edges.
	\newseq
	\begin{equation*}\subeqn\label{dots-1}
		\begin{tikzpicture}[scale=.45,baseline]
			\draw[very thick](-4,0) +(-1,-1) -- +(1,1) node[below,at start]
			{$i$}; \draw[very thick](-4,0) +(1,-1) -- +(-1,1) node[below,at
			start] {$j$}; \fill (-4.5,.5) circle (5pt);
			\node at (-2,0){=}; \draw[very thick](0,0) +(-1,-1) -- +(1,1)
			node[below,at start] {$i$}; \draw[very thick](0,0) +(1,-1) --
			+(-1,1) node[below,at start] {$j$}; \fill (.5,-.5) circle (5pt);
			\node at (4,0){for $i\neq j$};
	\end{tikzpicture}\end{equation*}
	\begin{equation*}\label{dots-2}\subeqn
		\begin{tikzpicture}[scale=.45,baseline]
			\draw[very thick](-4,0) +(-1,-1) -- +(1,1) node[below,at start]
			{$i$}; \draw[very thick](-4,0) +(1,-1) -- +(-1,1) node[below,at
			start] {$i$}; \fill (-4.5,.5) circle (5pt);
		\node at (-2,0){=}; \draw[very thick](0,0) +(-1,-1) -- +(1,1)
			node[below,at start] {$i$}; \draw[very thick](0,0) +(1,-1) --
			+(-1,1) node[below,at start] {$i$}; \fill (.5,-.5) circle (5pt);
			\node at (2,0){+}; \draw[very thick](4,0) +(-1,-1) -- +(-1,1)
			node[below,at start] {$i$}; \draw[very thick](4,0) +(0,-1) --
			+(0,1) node[below,at start] {$i$};
		\end{tikzpicture}\qquad
		\begin{tikzpicture}[scale=.45,baseline]
			\draw[very thick](-4,0) +(-1,-1) -- +(1,1) node[below,at start]
			{$i$}; \draw[very thick](-4,0) +(1,-1) -- +(-1,1) node[below,at
			start] {$i$}; \fill (-4.5,-.5) circle (5pt);
			\node at (-2,0){=}; \draw[very thick](0,0) +(-1,-1) -- +(1,1)
			node[below,at start] {$i$}; \draw[very thick](0,0) +(1,-1) --
			+(-1,1) node[below,at start] {$i$}; \fill (.5,.5) circle (5pt);
			\node at (2,0){+}; \draw[very thick](4,0) +(-1,-1) -- +(-1,1)
			node[below,at start] {$i$}; \draw[very thick](4,0) +(0,-1) --
			+(0,1) node[below,at start] {$i$};
		\end{tikzpicture}
	\end{equation*}
	\begin{equation*}\label{strand-bigon}\subeqn
		\begin{tikzpicture}[very thick,scale=.8,baseline]
			\draw (-2.8,0) +(0,-1) .. controls (-1.2,0) ..  +(0,1)
			node[below,at start]{$i$}; \draw (-1.2,0) +(0,-1) .. controls
			(-2.8,0) ..  +(0,1) node[below,at start]{$i$}; \node at (-.5,0)
			{=}; \node at (0.4,0) {$0$};
			\node at (1.5,.05) {and};
		\end{tikzpicture}
		\hspace{.4cm}
		\begin{tikzpicture}[very thick,scale=.8 ,baseline]
			
			\draw (-2.8,0) +(0,-1) .. controls (-1.2,0) ..  +(0,1)
			node[below,at start]{$i$}; \draw (-1.2,0) +(0,-1) .. controls
			(-2.8,0) ..  +(0,1) node[below,at start]{$j$}; \node at (-.5,0)
			{=};

			\draw (1.8,0) +(0,-1) -- +(0,1) node[below,at start]{$j$};
			\draw (1,0) +(0,-1) -- +(0,1) node[below,at start]{$i$};
		\end{tikzpicture}
	\end{equation*}
	In all the diagrams below, we assume that $i$ and $j$ are vertices with an edge $e\colon i\to j$, and the ghost shown is that attached to $e$ for the strand with label $j$ shown:
	\begin{equation*}\label{ghost-bigon1}\subeqn
		\begin{tikzpicture}[very thick,xscale=1.6 ,yscale=.8,baseline]
			\draw (1,-1) to[in=-90,out=90]  node[below, at start]{$k$} (1.5,0) to[in=-90,out=90] (1,1)
			;
			\draw[dashed] (1.5,-1) to[in=-90,out=90] (1,0) to[in=-90,out=90] (1.5,1);
			\draw (2.5,-1) to[in=-90,out=90]  node[below, at start]{$j$} (2,0) to[in=-90,out=90] (2.5,1);
			\node at (3,0) {=};
			\draw (3.7,-1) -- (3.7,1) node[below, at start]{$k$}
			;
			\draw[dashed] (4.2,-1) to (4.2,1);
			\draw (5.2,-1) -- (5.2,1) node[below, at start]{$j$};  \node at
			(6.2,0){for $i\neq  k$};
		\end{tikzpicture}
	\end{equation*}
	\begin{equation*}\label{ghost-bigon1a}\subeqn
		\begin{tikzpicture}[very thick,xscale=1.6 ,yscale=.8,baseline]
			\draw (1.5,-1) to[in=-90,out=90]  node[below, at start]{$k$} (1,0) to[in=-90,out=90] (1.5,1)
			;
			\draw[dashed] (1,-1) to[in=-90,out=90] (1.5,0) to[in=-90,out=90] (1,1);
			\draw (2,-1) to[in=-90,out=90]  node[below, at start]{$j$} (2.5,0) to[in=-90,out=90] (2,1);
			\node at (3,0) {=};
			\draw (4.2,-1) -- (4.2,1) node[below, at start]{$k$}
			;
			\draw[dashed] (3.7,-1) to (3.7,1);
			\draw (4.7,-1) -- (4.7,1) node[below, at start]{$j$};  \node at (6.2,0){for $i\neq  k$};
		\end{tikzpicture}
	\end{equation*} \begin{equation*}\label{ghost-bigon2}\subeqn
		\begin{tikzpicture}[very thick,xscale=1.4,baseline=25pt]
			\draw (1,0) to[in=-90,out=90]  node[below, at start]{$i$} (1.5,1) to[in=-90,out=90] (1,2)
			;
			\draw[dashed] (1.5,0) to[in=-90,out=90] (1,1) to[in=-90,out=90] (1.5,2);
			\draw (2.5,0) to[in=-90,out=90]  node[below, at start]{$j$} (2,1) to[in=-90,out=90] (2.5,2);
			\node at (3,1) {=};
			\draw (3.7,0) -- (3.7,2) node[below, at start]{$i$}
			;
			\draw[dashed] (4.2,0) to (4.2,2);
			\draw (5.2,0) -- (5.2,2) node[below, at start]{$j$} node[midway,fill,inner sep=2.5pt,circle]{};
			\node at (5.75,1) {$-$};
			
			\draw (6.2,0) -- (6.2,2) node[below, at start]{$i$} node[midway,fill,inner sep=2.5pt,circle]{};
			\draw[dashed] (6.7,0)-- (6.7,2);
			\draw (7.7,0) -- (7.7,2) node[below, at start]{$j$};\end{tikzpicture}
	\end{equation*}
	\begin{equation*}\label{ghost-bigon2a}\subeqn
		\begin{tikzpicture}[very thick,xscale=1.4,baseline=25pt]
			\draw (1.5,0) to[in=-90,out=90]  node[below, at start]{$i$} (1,1) to[in=-90,out=90] (1.5,2)
			;
			\draw[dashed] (1,0) to[in=-90,out=90] (1.5,1) to[in=-90,out=90] (1,2);
			\draw (2,0) to[in=-90,out=90]  node[below, at start]{$j$} (2.5,1) to[in=-90,out=90] (2,2);
			\node at (3,1) {=};
			\draw (4.5,0) -- (4.5,2) node[below, at start]{$i$};
			\draw[dashed] (4,0) to (4,2);
			\draw (5,0) -- (5,2) node[below, at start]{$j$} node[midway,fill,inner sep=2.5pt,circle]{};
			\node at (5.75,1) {$-$};
			
			\draw (6.9,0) -- (6.9,2) node[below, at start]{$i$} node[midway,fill,inner sep=2.5pt,circle]{};
			\draw[dashed] (6.4,0)-- (6.4,2);
			\draw (7.4,0) -- (7.4,2) node[below, at start]{$j$}; \end{tikzpicture}
	\end{equation*}
	\excise{ \begin{equation*}\subeqn\label{triple-boring}
			\begin{tikzpicture}[very thick,scale=1 ,baseline]
				\draw (-3,0) +(1,-1) -- +(-1,1) node[below,at start]{$m$}; \draw
				(-3,0) +(-1,-1) -- +(1,1) node[below,at start]{$i$}; \draw
				(-3,0) +(0,-1) .. controls (-4,0) ..  +(0,1) node[below,at
				start]{$j$}; \node at (-1,0) {=}; \draw (1,0) +(1,-1) -- +(-1,1)
				node[below,at start]{$m$}; \draw (1,0) +(-1,-1) -- +(1,1)
				node[below,at start]{$i$}; \draw (1,0) +(0,-1) .. controls
				(2,0) ..  +(0,1) node[below,at start]{$j$};
			\end{tikzpicture}
	\end{equation*}}
	\begin{equation*}\subeqn \label{eq:triple-point1}
		\begin{tikzpicture}[very thick,xscale=1.6,yscale=.8,baseline]
			\draw[dashed] (-3,0) +(.4,-1) -- +(-.4,1);
			\draw[dashed]      (-3,0) +(-.4,-1) -- +(.4,1);
			\draw (-1.5,0) +(.4,-1) -- +(-.4,1) node[below,at start]{$j$}; \draw
			(-1.5,0) +(-.4,-1) -- +(.4,1) node[below,at start]{$j$};
			\draw (-3,0) +(0,-1) .. controls (-3.5,0) ..  +(0,1) node[below,at
			start]{$i$};\node at (-.75,0) {=};  \draw[dashed] (0,0) +(.4,-1) -- +(-.4,1);
			\draw[dashed]      (0,0) +(-.4,-1) -- +(.4,1);
			\draw (1.5,0) +(.4,-1) -- +(-.4,1) node[below,at start]{$j$}; \draw
			(1.5,0) +(-.4,-1) -- +(.4,1) node[below,at start]{$j$};
			\draw (0,0) +(0,-1) .. controls (.5,0) ..  +(0,1) node[below,at
			start]{$i$};
			\node at (2.25,0)
			{$-$};
			\draw (4.5,0)
			+(.4,-1) -- +(.4,1) node[below,at start]{$j$}; \draw (4.5,0)
			+(-.4,-1) -- +(-.4,1) node[below,at start]{$j$};
			\draw[dashed] (3,0)
			+(.4,-1) -- +(.4,1); \draw[dashed] (3,0)
			+(-.4,-1) -- +(-.4,1);
			\draw (3,0)
			+(0,-1) -- +(0,1) node[below,at start]{$i$};
		\end{tikzpicture}
	\end{equation*}
	\begin{equation*}\subeqn\label{eq:triple-point2}
		\begin{tikzpicture}[very thick,xscale=1.6,yscale=.8,baseline]
			\draw[dashed] (-3,0) +(0,-1) .. controls (-3.5,0) ..  +(0,1) ;
			\draw (-3,0) +(.4,-1) -- +(-.4,1) node[below,at start]{$i$}; \draw
			(-3,0) +(-.4,-1) -- +(.4,1) node[below,at start]{$i$};
			\draw (-1.5,0) +(0,-1) .. controls (-2,0) ..  +(0,1) node[below,at
			start]{$j$};\node at (-.75,0) {=};
			\draw (0,0) +(.4,-1) -- +(-.4,1) node[below,at start]{$i$}; \draw
			(0,0) +(-.4,-1) -- +(.4,1) node[below,at start]{$i$};
			\draw[dashed] (0,0) +(0,-1) .. controls (.5,0) ..  +(0,1);
			\draw (1.5,0) +(0,-1) .. controls (2,0) ..  +(0,1) node[below,at
			start]{$j$};
			\node at (2.25,0)
			{$+$};
			\draw (3,0)
			+(.4,-1) -- +(.4,1) node[below,at start]{$i$}; \draw (3,0)
			+(-.4,-1) -- +(-.4,1) node[below,at start]{$i$};
			\draw[dashed] (3,0)
			+(0,-1) -- +(0,1);\draw (4.5,0)
			+(0,-1) -- +(0,1) node[below,at start]{$j$};
		\end{tikzpicture}.
	\end{equation*}
	\begin{equation*}\subeqn\label{cost}
		\begin{tikzpicture}[very thick,baseline,scale=.8,xscale=.7]
			\draw (-2.8,0)  +(0,-1) .. controls (-1.2,0) ..  +(0,1) node[below,at start]{$i$};
			\draw[wei] (-1.2,0)  +(0,-1) .. controls (-2.8,0) ..  +(0,1) node[below,at start]{$i$};
			\node at (-.3,0) {=};
			\draw[wei] (2.8,0)  +(0,-1) -- +(0,1) node[below,at start]{$i$};
			\draw (1.2,0)  +(0,-1) -- +(0,1) node[below,at start]{$i$}
			node[midway,circle,fill=black, inner sep=2pt]{};
		\end{tikzpicture}\qquad\qquad
		\begin{tikzpicture}[very thick,baseline,scale=.8,xscale=.7]
			\draw[wei] (6.8,0)  +(0,-1) .. controls (5.2,0) ..  +(0,1) node[below,at start]{$j$};
			\draw (5.2,0)  +(0,-1) .. controls (6.8,0) ..  +(0,1) node[below,at start]{$i$};
			\node at (7.7,0) {=};
			\draw (9.2,0)  +(0,-1) -- +(0,1) node[below,at start]{$i$};
			\draw[wei] (10.8,0)  +(0,-1) -- +(0,1) node[below,at start]{$j$};
		\end{tikzpicture}
	\end{equation*}
	\begin{equation*}\subeqn
		\begin{tikzpicture}[very thick,scale=.8,baseline]
			\draw (-3,0)  +(1,-1) -- +(-1,1) node[at start,below]{$i$};
			\draw (-3,0) +(-1,-1) -- +(1,1)node [at start,below]{$j$};
			\draw[wei] (-3,0)  +(0,-1) .. controls (-4,0) ..  node[at start,below]{$m$}+(0,1);
			\node at (-1,0) {=};
			\draw (1,0)  +(1,-1) -- +(-1,1) node[at start,below]{$i$};
			\draw (1,0) +(-1,-1) -- +(1,1) node [at start,below]{$j$};
			\draw[wei] (1,0) +(0,-1) .. controls (2,0) .. node[at start,below]{$m$} +(0,1);
			\node at (2.8,0) {$+ $};
			\draw (6.5,0)  +(1,-1) -- +(1,1)  node[at start,below]{$i$};
			\draw (6.5,0) +(-1,-1) -- +(-1,1) node [at start,below]{$j$};
			\draw[wei] (6.5,0) +(0,-1) -- node[at start,below]{$m$} +(0,1);
			\node at (3.8,-.2){$\delta_{i,j,m} $}  ;
		\end{tikzpicture}
	\end{equation*}
	\begin{equation*}\subeqn\label{dumb}
		\begin{tikzpicture}[very thick,scale=.8,baseline]
			\draw[wei] (-3,0)  +(1,-1) -- +(-1,1);
			\draw (-3,0)  +(0,-1) .. controls (-4,0) ..  +(0,1);
			\draw (-3,0) +(-1,-1) -- +(1,1);
			\node at (-1,0) {=};
			\draw[wei] (1,0)  +(1,-1) -- +(-1,1);
			\draw (1,0)  +(0,-1) .. controls (2,0) ..  +(0,1);
			\draw (1,0) +(-1,-1) -- +(1,1);    \end{tikzpicture}\qquad \qquad
		\begin{tikzpicture}[very thick,scale=.8,baseline]
			\draw(-3,0) +(-1,-1) -- +(1,1);
			\draw[wei](-3,0) +(1,-1) -- +(-1,1);
			\fill (-3.5,-.5) circle (3pt);
			\node at (-1,0) {=};
			\draw(1,0) +(-1,-1) -- +(1,1);
			\draw[wei](1,0) +(1,-1) -- +(-1,1);
			\fill (1.5,.5) circle (3pt);
		\end{tikzpicture}
	\end{equation*}
	
	For the relations (\ref{cost}) and (\ref{dumb}), we also include their mirror
	images. We also include isotopy through all triple points not shown as relations.

\end{Definition}
Given $ \Bv \in \Z_{\geq 0}^I $, we let
$\walg^\vartheta_{\Bv}$ be the subalgebra containing $ v_i $ black strands labelled $ i $, for $ i \in I $.
Since the relations (\ref{dots-1}--\ref{dumb}) are homogeneous with respect to the grading induced by the degree of KLRW diagrams given in Definition \ref{def:degree},
this makes $\walg^\vartheta_{\Bv}$ into a graded algebra.

We'll also be interested in the so-called steadied quotients of these algebras.
A loading is  {\bf unsteady} if there is a group of black strands we can move right
arbitrarily far without changing the equivalence class of the
loading.

The {\bf steadied quotient} $T^{\vartheta}_{\Bv}$ is the quotient of
$\tilde{T}^{\vartheta}_{\Bv}$ by the 2-sided ideal generated by $e_{\boldsymbol{\ell}}$, as $\boldsymbol{\ell}$ ranges over all the unsteady
loadings (\cite[Def. 2.22]{WebwKLR}).

\begin{Remark}
	A reader comparing with the definition of unsteady in \cite{WebwKLR}
	might have trouble seeing why this is equivalent.  In this paper, we
	are only interested in the special case discussed in  \cite[\S
	3.1]{WebwKLR}; that is, we are using the charge
	$c$ which assigns $c(i)=1+i$ for all old vertices and $i-\sum v_i$ to
	$\infty$.  The equivalence in this case exactly follows the proof of  \cite[Th. 3.6]{WebwKLR}.
\end{Remark}

\subsection{Flavoured KLRW algebras}
\label{sec:metric-weighted-klr}

We now turn to introducing flavoured KLRW algebras, which can be
thought of as variations on the weighted KLRW algebras, that can more
aptly describe the representation theory of Coulomb branch algebras $\CB(\Bv,\Bw)$.  In this and the following sections, we will also explain the parallels and differences between flavoured KLRW algebras and weighted KLRW algebras. In particular, this will allow us to appeal to results for weighted KLRW algebras which have been developed in \cite{WebwKLR}.

\begin{Definition} \label{def:Flavor}
	A {\bf flavour} on the quiver $\Gamma^{\Bw}$ is a map $\varphi\colon \edge(\Gamma^{\Bw})
	\to \C, e \mapsto \varphi_e$.  A
	$\varphi$--\textbf{flavoured sequence} is a triple $(\Bi,\Ba, <)$ consisting of:
	\begin{itemize}
		\item An ordered $n$-tuple $\Bi=(i_1,\dots, i_n) \in I^n$.
		\item An ordered $n$-tuple $\Ba=(a_1,\dots,a_n)\in \C^n$ of complex
		numbers; we call these the {\bf longitudes} of the corresponding
		vertices in $\Bi$.
		\item A total order $<$ on $\Sall$ (as defined just below \eqref{eq:Gplus}) extending
		the usual comparison order on $[1,n]$.  As before, we refer
		to elements of $\Sc$ as {\bf corporeal}, elements of $\Sr$ as {\bf red} and elements of $\Sg$ as {\bf ghostly}.  	\end{itemize}
	A red/ghostly element $g=(x,e)\in \Sgr$ is endowed with the longitude $a_g=a_x+\varphi_e$ if $x \in \Sc$, and a red element is endowed with the longitude $a_g=\varphi_e$ if $x=\star$.
	We require the following properties to be satisfied:
	\renewcommand{\theenumi}{\roman{enumi}}
	\begin{enumerate}
		\item The {\bf real longitudes} $\Re(a_g)$, for $g\in \Sall$, are weakly increasing with
		respect to $<$.
		\item If $g\in \Sgr, m\in \Sc$ and $\Re(a_g)=\Re(a_m)$, then
		$g<m$.
	\end{enumerate}
\end{Definition}
\begin{Remark}
	Note that we had previously used ``flavour'' to refer to an element of the Lie algebra $\mathfrak{f}=\operatorname{Lie}(F)$, which gives a choice of quantization parameters for $\CB(\Bv,\Bw)$; as discussed in Section \ref{se:quiverdata}, in the case of a quiver gauge theory corresponding to the data $\Gamma,\Bv,\Bw$, we have  $F =  (\Cx)^{\edge(\Gamma^\Bw)}$, so a flavour in the sense defined above is indeed an element of $\mathfrak{f}$.
\end{Remark}

For the remainder of the section, fix a flavour $\varphi$.  A pair of flavoured sequences $(\Bi,\Ba,<)$, $(\Bi',\Ba',<')$ corresponding to $\varphi$ are {\bf
	equivalent} if there is a permutation $\sigma\in S_n$ such that:
\begin{enumerate}
	\item For all $m\in \Sc$, we have $i_{m}=i'_{\sigma(m)}$.
	\item For any corporeal $m\in \Sc$, and
	$(k,e)\in \Sgr$ such that $t(e)=i_m$, we have that $m < (k,e)$ if and
	only if $\sigma(m) <' (\sigma(k),e)$.
	\item For any corporeals $k,m\in \Sc$ with $i_k=i_m$, we have
	that $\Re(a_k)<\Re(a_m)$ if and only if $\Re(a'_{\sigma(k)})<\Re(a'_{\sigma(m)})$.
\end{enumerate}

We'll sometimes want to think about the elements of a flavoured
sequence one vertex at a time.  Fixing a flavoured sequence $(\Bi,\Ba,<)$ we'll refer to the set $\{k\in \Sc \hspace{1mm}|\hspace{1mm}
i_k=i\}$ as the {\bf corporeals with label $i$}.  Suppose there are
$v_i$ corporeals with label $i$; this gives a dimension vector $\Bv
\in \Z_{\geq}^{\vertex}$.  The longitudes $\Ba $ can then be organized into a point of $ \ft := \prod_{i \in \vertex} \C^{v_i} $.

On $ \ft $, we have an action of the Weyl group $ W = \prod S_{v_i} $.  A point $ \gamma \in \ft / W $ will be regarded as a tuple of multisets $ \gamma_i $ of complex numbers, with $ \gamma_i $ of size $ v_i $.  Conversely, we can produce flavoured sequences from points of $ \ft/W $.
\begin{Lemma}\label{lem:sequence-from-flavours}
	Let $ \gamma \in \ft / W $.  There is a $\varphi$--flavoured sequence, unique up to equivalence, for which $ \gamma_i $ is the multiset of
	longitudes $a_k$ such that $i_k=i$.
\end{Lemma}
\begin{proof}
	We have fixed the multiset of corporeal longitudes $(a_1,\dots, a_n)$ for each vertex $i$, and
	the equation $a_g=a_k+\varphi_e$ fixes the longitudes on ghosts
	associated to the edge $ e$. Choose an order on the union of the
	multisets $\gamma_i$ so that the real parts are weakly
	increasing, and denote the resulting list $\Ba$.  Define $\Bi$ so that the $k$th element in order comes
	from the set $\gamma_{i_k}$.

	We now have the associated set $\Sall$, and we let $<$ be any order on
	this set so that the real longitudes are weakly increasing compatible
	with property (ii);
	that is, amongst elements of a fixed
	real longitude, we put the ghostly/red elements first and then corporeals, using any order within each group.  Any two
	such orders are related by a permutation that only permutes pairs of
	corporeals with the same real longitude, or elements of $\Sgr$ with the same real
	longitude.  This is manifestly an equivalence.
\end{proof}
\newcommand{\va}{\ensuremath{\alpha}}
\newcommand{\vb}{\ensuremath{\beta}}
\begin{Example}\label{ex:Kronecker}
	Let $\Gamma$ be the Kronecker quiver with cyclic
	orientation and $ \Bw = 0 $.  We label vertices by $0$ and $1$, and edges by $e$ and $f$:
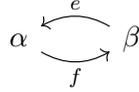
	$$
	\begin{tikzcd}
		\va  \arrow[bend right = 30,swap]{r}{f} & \vb \arrow[bend right = 30,swap]{l}{e}
	\end{tikzcd}
	$$
	Suppose we choose constant flavour $\varphi_e = \varphi_f = 1$.   Let's consider possible flavoured sequences corresponding to $\Bi = (\va,{\vb})$ or $\Bi = ({\vb},\va)$.  
	Note that for either choice of $\Bi$ there is a unique ghostly element corresponding to each edge, so we may label the ghostlies by $e$ and $f$. Let $a$ be the longitude associated to vertex $\va$ (resp.~$b$ the longitude for ${\vb}$).  Then the longitude of the ghost labelled ${e}$ is $a+1$ (resp.~of ${f}$ is $b+1)$.  Thus, we can have the following
	patterns of sequence and longitude:
	\[
	\begin{array}{|c|c|c|}\hline
		\text{Sequence}&\text{Longitudes}&\text{Inequalities}\\\hline
		(\va,e,{\vb},f) & (a,a+1,b,b+1)&    \Re(a+1)\leq    \Re(b)\\
		(\va,{\vb},e,f) & (a,b,a+1,b+1)&\Re(a)\leq  \Re(b) < \Re(a+1)\\
		(\va,{\vb},f,e) & (a,b,a+1,b+1)&\Re(a)=  \Re(b) \\
		({\vb},\va,e,f) & (a,b,a+1,b+1)&\Re(a)=  \Re(b) \\
		({\vb},\va,f,e) & (b,a,b+1,a+1)&\Re(b)\leq  \Re(a) < \Re(b+1)\\
		({\vb},f,\va,e) & (b,b+1,a,a+1)& \Re(b+1)\leq  \Re(a)\\\hline
	\end{array}
	\]
	Any other sequence cannot be flavoured.
\end{Example}

\begin{Definition}\label{def:flavoured-KLR-diagram}
	A {\bf flavoured KLRW diagram} is a collection of finitely many
	oriented smooth curves (which, as before, we call ``strands'') whose projection to the $y$-axis must be a
	diffeomorphism to $[0,1]$, labeled with vertices in
	$ \vertex $.  Each strand must have one endpoint on $y=0$ and
	one on $y=1$, at distinct points from the other strands.
	
	The
	strands are divided into three sets: the {\bf corporeal} (which are drawn
	as solid black lines), the {\bf red} (which are drawn
	as solid red lines) and the {\bf ghostly} (which are drawn as dashed black lines). 
	These must satisfy the usual genericity property of avoiding
	tangencies and triple points between any set of strands.
	
	In addition, a flavoured KLRW diagram carries the data of $\varphi$-flavoured sequences $(\Bi,\Ba,<)$ and
	$(\Bi',\Ba',<')$ corresponding to the lines $y=0$ and $y=1$.  These
	induce bijections between the set of strands and the sets $\Sall$ and
	$\Sall'$ respectively, by matching the order $<$ with the left to right order
	of strands at $y=0$, and matching $<'$ with the left-to-right order at
	$y=1$. We require that:
	\begin{enumerate}
		\item These bijections are compatible with the division of
		$\Sall,\Sall'$ into $\Sc,\Sg,\Sr$ and the set of strands into corporeal,
		ghostly and red subsets.  In particular, the bijection $\Sall
		\to \Sall'$ induces a permutation $\sigma\in S_n$ of corporeal
		strand which respects their labels: $i_{m}=i'_{\sigma(m)}$.
		\item The bijection $\Sgr\to \Sgr'$ between ghostly/red elements in the flavoured
		sequence is given by $(k,e)\mapsto (\sigma(k),e)$, that is, it is
		induced by the bijection on corporeals.
		\item The longitudes satisfy $a_{m}-a'_{\sigma(m)}\in \Z$, that is,
		the difference between the
		longitudes at the top and bottom of each strand lies in $\Z$; if
		this holds for corporeal strands, then it automatically follows
		for ghostly/red strands.
	\end{enumerate}
	We consider these diagrams up to isotopy which preserves all these
	conditions.  Note that unlike in the weighted case, we allow the
	points at $y=0$ and $y=1$ to move in these isotopies, as long as their
	order is preserved.  
	Corporeal strands are allowed to carry a finite number of dots.  
\end{Definition}

Let us now describe our conventions for drawing flavoured KLRW diagrams.  Strands corresponding to corporeals are drawn as solid black lines,  strands corresponding to ghostly elements are drawn as dashed black lines, and strands corresponding to red elements as solid red lines.  At the top and bottom of each diagram, we include two rows of information:  
\begin{enumerate}
	\item  In the first row, we write the corresponding vertex $i_k$ for $k\in \Sc$, the edge $e$ for $(k,e)\in \Sg$ and the tail vertex $t(e)$ for $(\star,e)\in \Sr$.  
	\item In the second row we write the longitude.  
\end{enumerate}
See equation (\ref{eq: example flavoured KLRW diagram}) for an example of these conventions.

Note that any corporeal (or ghostly) strand does not have a single well-defined longitude, but rather two: the longitudes attached to its top and bottom. Red strands, on the other hand do have a single well-defined longitude. 
However, the longitudes at the top and bottom of a corporeal or
ghostly strand  can only differ by an integer, so they give a well-defined element of
$\C/\Z$.  In particular, for any pair of strands the difference
between the longitudes at the top of the two strands is
integral if and only if the same is true of the longitudes at the
bottom of the strands; we say a set of strands where all pairs have this property has
{\bf integral difference}.

\begin{Example}\label{ex:Kronecker2}
	Let us return to the example of the Kronecker quiver, with the same conventions as above, but the dimension vector \[v_{\va}=2\qquad v_{\vb}=1\qquad w_{\va}=2\qquad w_\vb=1.\]
	$$
	
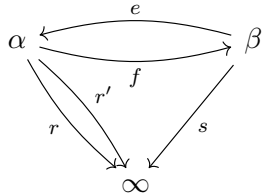
\begin{tikzcd}
		\va  \arrow[bend right = 15,swap]{rr}{f} \arrow[bend right = 10,swap]{rdd}{r} \arrow[bend left = 10]{rdd}{r'} && \vb \arrow[bend right = 15,swap]{ll}{e} \arrow{ldd}{s}\\ \\
		& \infty &
	\end{tikzcd}
	$$  
	We fix a flavour $\varphi$ given by $e,f \mapsto 1$ and $r\mapsto -4, r' \mapsto 0$ and $s \mapsto 2$.  We define two $\varphi$-flavoured sequences $(\Bi,\Ba,<)$ and $(\Bi',\Ba',<')$ as follows.  First we set:
	\begin{align*}
		\Bi=(\va,\va,\vb), \; \Ba=(-6,-1,0) \\
		\Bi'=(\va,\vb,\va), \; \Ba'=(-3,-2,3)
	\end{align*}
	From this we obtain the corresponding sets of ghostly/red elements:
	\begin{align*}
		\Sg&=\{ (1,e),(2,e),(3,f) \} &\Sr&= \{ (\star,r), (\star,r'),(\star,s)\} \\
		\Sg'&=\{ (1,e),(2,f),(3,e) \} &\Sr'&= \{ (\star,r), (\star,r'),(\star,s)\}
	\end{align*}
	To ease notation, for $(x,e) \in \Sgr$ we write $e_x$ if $x \in [1,3]$ and simply $e$ if $x=\star$ (and similarly for elements of $G'$).  We now define $<$ and $<'$ as follows:
	\begin{align*}
		1<e_1<r<2<3<r'<e_2<f_3<s \\
		r <' 1 <' 2<' e_1<'f_2<'r'<'s<'3<'e_3
	\end{align*} Here is an example of a flavoured KLRW diagram where the flavour on the top is $(\Bi,\Ba,<)$ and the flavour on the bottom is $(\Bi',\Ba',<')$:
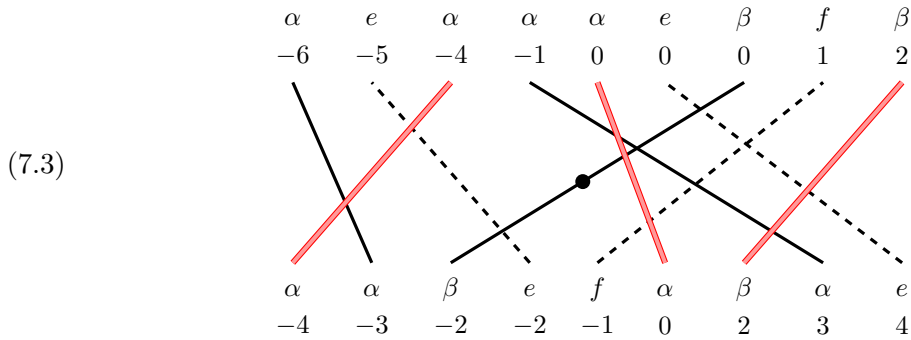
	\begin{equation}
		\label{eq: example flavoured KLRW diagram}
		\tikz[baseline, yscale=1.2,xscale=1.5]{
			\draw[very thick] (-2,-1) to(-2.7,1);
			\draw[dashed,very thick] (-.6,-1) to(-2,1);        
			\draw[very thick] (-1.3,-1)to node [pos=.45,circle,fill=black,inner sep=2pt] {} (1.3,1);
			\draw[very thick,dashed] (0,-1)to (2,1);
			\draw[very thick] (2,-1)--(-.6,1);
			\draw[very thick,dashed] (2.7,-1)to (.6,1);
			\draw[wei] (-2.7,-1) to (-1.3,1);
			\draw[wei] (1.3,-1)to (2.7,1); 
			\draw[wei] (.6,-1)to (0,1); 
			\draw (-2.7,0) +(0.0,-1.3) node {\small$\va$};
			\draw (-2.7,0) +(0.0,-1.7) node {\small$-4$};
			\draw (-2.7,0) +(0.0,1.7) node {\small$\va$};
			\draw (-2.7,0) +(0.0,1.3) node {\small$-6$};
			\draw (-2,0) +(0.0,-1.3) node {\small$\va$};
			\draw (-2,0) +(0.0,-1.7) node {\small$-3$};
			\draw (-2,0) +(0.0,1.7) node {\small$e$};
			\draw (-2,0) +(0.0,1.3) node {\small$-5$};
			\draw (-1.3,0) +(0.0,-1.3) node {\small$\vb$};
			\draw (-1.3,0) +(0.0,-1.7) node {\small$-2$};
			\draw (-1.3,0) +(0.0,1.7) node {\small$\va$};
			\draw (-1.3,0) +(0.0,1.3) node {\small$-4$};
			\draw (-.6,0) +(0.0,-1.3) node {\small$e$};
			\draw (-.6,0) +(0.0,-1.7) node {\small$-2$};
			\draw (-.6,0) +(0.0,1.7) node {\small$\va$};
			\draw (-.6,0) +(0.0,1.3) node {\small$-1$};
			\draw (0,0) +(0.0,-1.3) node {\small$f$};
			\draw (0,0) +(0.0,-1.7) node {\small$-1$};
			\draw (1.3,0) +(0.0,1.7) node {\small$\vb$};
			\draw (1.3,0) +(0.0,1.3) node {\small$0$};
			\draw (.6,0) +(0.0,-1.3) node {\small$\va$};
			\draw (.6,0) +(0.0,-1.7) node {\small$0$};
			\draw (0,0) +(0.0,1.7) node {\small$\va$};
			\draw (0,0) +(0.0,1.3) node {\small$0$};
			\draw (1.3,0) +(0.0,-1.3) node {\small$\vb$};
			\draw (1.3,0) +(0.0,-1.7) node {\small$2$};
			\draw (.6,0) +(0.0,1.7) node {\small$e$};
			\draw (.6,0) +(0.0,1.3) node {\small$0$};
			\draw (2,0) +(0.0,-1.3) node {\small$\va$};
			\draw (2,0) +(0.0,-1.7) node {\small$3$};
			\draw (2,0) +(0.0,1.7) node {\small$f$};
			\draw (2,0) +(0.0,1.3) node {\small$1$};
			\draw (2.7,0) +(0.0,-1.3) node {\small$e$};
			\draw (2.7,0) +(0.0,-1.7) node {\small$4$};
			\draw (2.7,0) +(0.0,1.7) node {\small$\vb$};
			\draw (2.7,0) +(0.0,1.3) node {\small$2$};
		}
	\end{equation}
	Focusing for instance on the top of the diagram, from the first row we can read off $\Bi$ and the total order $<$ on $\Sall$, and from the second row we can read off the longitudes $\Ba$.  Note that the longitudes of ghostly/red elements are determined by $\Ba$ and $\varphi$.
\end{Example}

\begin{Definition}\label{def:flavoured-KLR}
	The {\bf flavoured KLRW algebra} $\fKLR^{\varphi} = \fKLR^{\varphi}(\Gamma^{\Bw})$ is the algebra given by the $\C$-span
	of the $\varphi$-flavoured KLRW diagrams, modulo isotopy preserving genericity
	and  the local relations
	(\ref{dots-1}--\ref{dumb}) if  the set of strands involved in the
	relation have integral difference.  If any pair of strands involved
	does not have integral difference, then
	we can simply isotope through a triple point or tangency.  As usual, when we
	multiply diagrams, we must be able to match the flavoured sequences
	at the top of one diagram and the bottom of the other, or the
	product is 0 by convention.  
\end{Definition}

Fix a dimension vector $ \Bv \in \Z_{\geq}^{\vertex} $, which gives the number of corporeal strands with each label, and let $ \ft = \prod_i \C^{v_i} $ as above.  Let $ \mathsf S \subset \ft $ be any set.  We let
$\fKLR^{\varphi}_{\mathsf{S}}$ be the subalgebra where we only
allow flavoured sequences corresponding to elements of $\mathsf{S}$
at the top and bottom of diagrams. We will be particularly interested in the case where $ \mathsf S = \mathscr S $ is an orbit of the extended affine Weyl group
$\widehat{W}=\prod_i S_{v_i}\ltimes \Z^{v_i}$.

We define a grading on the flavoured KLRW algebra analogous to that
on the weighted KLRW algebra: we give a crossing or dot the same
grading it would have in the weighted KLRW algebra, except that crossings of
strands which don't have integral difference are given degree
0.

\begin{Remark}
	In \cite{KTWWYO}, we specialized to a bipartite quiver (with the two
	sets of vertices called {\bf even} and {\bf odd}), and chose edge
	orientations to point from even to odd.  In that paper, we defined the metric KLRW algebra which is a special case of the flavoured KLRW algebra defined here.  
	
	The most straightforward way to make this connection precise is
	flavour every old edge in
	$\Gamma$ with $1/2$, and each new edge with 0 and assume that every odd/even corporeal has
	longitude at top or bottom given by $k/2$, where $k$ is an integer of the correct
	parity.  In the conventions of \cite{KTWWYO}, $k$ would have been
	the corresponding longitude.  Similarly, assume that the longitudes of red strands are of the form $r/2$, where $r$ is an integer of the same parity as the corresponding label.  Up to this factor of 2, we obtain a
	metric longitude in the sense of \cite[Def. 3.21]{KTWWYO} from a
	flavoured sequence in the special case described here, and sets $R_i$ defined by the longitudes on red strands with label $i$.  This defines 
	an
	isomorphism $ \fKLR^\varphi_{\mathcal S} \cong \mathscr T^{\mathbf
		R}$, where $\mathcal S  $ consists of flavoured sequences with the
	parity convention above, between the flavoured KLRW algebra and the
	metric KLRW algebra.

	This use of half-integers here is in contrast with our usual
	conventions in this paper.   We can apply Lemma
	\ref{lem:real-reduction}, using the cocycle $ \eta $ with value $1/2$ on even vertices and $0$ on odd vertices, to show that the metric
	KLRW is also equivalent to the flavoured KLRW algebra where all edges
	are given weight 0, and all longitudes are integral.  
\end{Remark}

\begin{Definition} \label{def:idem}
	Let
	$ e(\Bi,\Ba,<) $ denote the idempotent given by the straight-line
	diagram with the flavoured sequence $(\Bi,\Ba,<) $.  Given $\gamma\in
	\ft/W$, let $e(\gamma)= e(\Bi,\Ba,<) $ for the flavoured
	sequence associated to $\gamma$ by Lemma
	\ref{lem:sequence-from-flavours}.
\end{Definition}

As in \cite[Prop.~2.15]{WebwKLR}, there is a natural symmetry in the
definition of flavoured KLRW algebras.  We may view a flavour $\varphi$
as a 1-cocycle on the graph $\Gamma^\Bw$.  Let $\eta:
\vertex \sqcup \{ \infty \} \rightarrow \C$ be a 0-cocycle with $ \eta_\infty = 0 $.  Then we may define a cohomologous 1-cocycle
$\varphi- d \eta$, by $(\varphi - d\eta)_e = \varphi_e -
\eta_{h(e)}+\eta_{t(e)}$.  Given an orbit $\mathscr S \subset \prod_i
\C^{v_i} $, we may define a new orbit $\mathscr S + \eta$ by
simultaneously translating the $\C^{v_i}$ components by $\eta_i$.

\begin{Lemma}
	\label{lemma: cohomologous data}
	With notation as above, there is an isomorphism
	\[\fKLR^\varphi_{\mathscr{S}}\cong \fKLR^{\varphi-d
		\eta}_{\mathscr{S}+\eta}\] defined by shifting longitudes and
	reordering as necessary.
\end{Lemma}
\begin{proof}
	As in \cite[Prop.~2.15]{WebwKLR}\footnote{Note that the published version of this paper has a sign error, and should read ``$\vartheta-d\eta$'', not ``$\vartheta+d\eta$.''}, under this isomorphism a corporeal strand with label $i$ has its longitudes at top and bottom both shifted by $\eta_i$, while a ghostly/red strand labelled by an edge $e$ has its longitude shifted by $\eta_{t(e)}$.  Since all corporeal strands labelled by $t(e)$ also have their longitudes shifted by $\eta_{t(e)}$,  this shifting preserves all crossings between ghostly/red strands labelled $e$ and corporeal strands labelled $t(e)$.
\end{proof}

\begin{Definition}\label{def:sfKLR}
	By analogy with loadings, we call a flavoured sequence $(\Bi,\Ba, <)$ {\bf unsteady
	} if, for some $0<k<n $, the last $k$
	elements of $\Sall$ consists of a group of corporeal elements and all their
	ghosts.  Importantly, this group should not contain any ghost of one of the corporeal strands which doesn't lie in it, nor any red strands. 
	We define $\sfKLR^{\varphi}_{\mathscr{S}}$, the {\bf steadied quotient of the
		algebra} $\fKLR^{\varphi}_{\mathscr{S}}$, to be the quotient of $\fKLR^{\varphi}_{\mathscr{S}}$ by the
	two-sided ideal generated by all the idempotents for unsteady
	flavoured sequences.
\end{Definition}

\begin{Example}
	Consider the situation of Example \ref{ex:Kronecker2}, and assume $w_\va=1,w_\vb=0$, so there is one element of $\Sr$, which we denote $r$.  In this case, $(\vb,f,r,\va,e)$ is unsteady  since the last 2 entries are a strand and its only ghost; similarly,  $(r,\vb,\va,f,e)$ is unsteady because of its last 4 entries.  On the other hand, $(\vb,r,\va,f,e)$ is not unsteady.  
\end{Example}

\subsection{Reduction to the integral case}
\begin{Definition}\label{def:GammaS}
	For a given orbit $\mathscr{S}\subset \ft$, we let $\GammaInt$ be the
	subgraph of $\Gamma\times (\C/\Z)$ consisting of the set $\widetilde{\vertex}$ of pairs $(i,\redu{z})$,
	for $z\in\C$  which appear  as a coordinate in the factor $\C^{v_i}$
	for some element of $\mathscr{S}$.  This has an adjacency $(i,\redu{z})\to (j,\redu{w})$ for
	each edge $e=i \to j$ with $\varphi_e\equiv z-w\pmod \Z$.
\end{Definition}

We have an induced dimension vector $\vInt:\mathsf{E}(\GammaInt) \to \Z_{\geq 0}$ defined as follows.  Let $x=(x_i)\in \ft$ be any element of $\mathscr{S}$.  Then $\tilde v_{i,\redu{z}}$ is the number of entries in $x_i$ whose class in $\C/\Z$ is equal to $\redu{z}$.  Note that we have a canonical isomorphism $\C^{v_i} \cong \prod_{\redu{z}}\C^{\tilde{v}_{i,\redu{z}}}$, identifying the coordinates $(i,k)$ such that $x_{i,k}\equiv z\pmod \Z$ with the coordinates of $\C^{\tilde{v}_{i,\redu{z}}}$.  Thus, we can naturally consider $x$ as an element of $\prod_{i,\redu{z}} \C^{\tilde{v}_{i,\redu{z}}}$.  Let $\tilde{\mathscr{S}}$ be the orbit under the Weyl group of $\tilde{\Gamma}$ of $x$; this is the same as the elements of $\mathscr{S}$ whose projection to $\C^{\tilde{v}_{i,\redu{z}}}$ lies in $(\Z+z)^{\tilde{v}_{i,\redu{z}}}$.

We will also need an induced vector $\wInt$; this can be read off by adding the Crawley-Boevey vertex $(\infty,0)$ to $\GammaInt$ and applying the same rules as above to new edges.  Thus $\tilde w_{i,\redu{z}}$ is the number of new edges with flavour lying in the coset $\redu{z}$. Note that unlike $\vInt$, the sum of the entries of $\wInt$ might be less than that for $\Bw$, since there might be flavours on new edges not congruent to any coordinate of an element of $\mathscr{S}$; in fact, for a generic orbit $\mathscr{S}$, we will have $\wInt=0$.

Finally, we also have an induced flavour $\varphiInt$ on $\GammaInt$ by pulling back $\varphi$ by the projection map $ \GammaInt \to \Gamma$, that is,  the flavour of $(i,\redu{z})\to (j,\redu{w})$ is equal to $\varphi_e$, where $e=i \to j$.

\begin{Example}
   Consider the Kronecker quiver with conventions as in Example \ref{ex:Kronecker2}, with  $v_{\va}=5$, $v_{\vb}=6$ and $\mathscr{S}$ is the orbit containing the elements $(0,1/3,1/2,2/3,2/3)\in \C^{v_{\va}}$ and $(0,1/6,1/3,1/3,1/2,2/3)\in \C^{v_{\beta}}$.  If we choose the flavours \[\varphi_e=1/3\qquad \varphi_f=0\qquad \varphi_{r}=0\qquad \varphi_{r'}=\sqrt{2} \qquad \varphi_s=1/2\] then every component of $\GammaInt$ lies in a union of 6-cycles obtained as 3-fold covers of the Kronecker quiver.  In this case, we obtain one full 6-cycle, and 3 vertices from another, with the Crawley-Boevey graph drawn below.  Note that since $(\va,\sqrt{2})$ is not a vertex ($\sqrt{2}$ is not a coordinate in the correct orbit), $r'$ does not contribute to $\tilde{w}$.
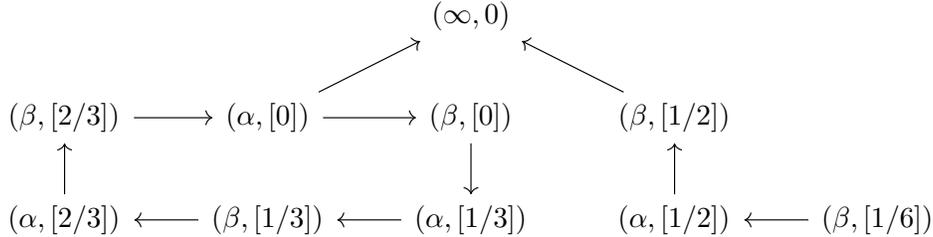
\[\begin{tikzcd}[ampersand replacement=\&,column sep=2.25em]
	\&\& {(\infty,0)} \\
	{(\vb,\redu{2/3})} \& {(\va,\redu{0})} \& {(\vb,\redu{0})} \&  {(\vb,\redu{1/2})} \&\\
	{(\va,\redu{2/3})} \& {(\vb,\redu{1/3})} \& {(\va,\redu{1/3})} \& {(\va,\redu{1/2})} \&{(\vb,\redu{1/6})}
	\arrow[from=2-2, to=2-3]
	\arrow[from=2-2, to=1-3]
	\arrow[from=3-5, to=3-4]
	\arrow[from=3-4, to=2-4]
	\arrow[from=3-2, to=3-1]
	\arrow[from=3-1, to=2-1]
	\arrow[from=2-1, to=2-2]
	\arrow[from=3-3, to=3-2]
	\arrow[from=2-3, to=3-3]
	\arrow[from=2-4, to=1-3]
\end{tikzcd}\]
The non-zero values of the dimension vectors are:
\[v_{\va,\redu{0}}=1\qquad v_{\va,\redu{1/3}}=1\qquad v_{\va,\redu{1/2}}=1\qquad v_{\va,\redu{2/3}}=2\qquad v_{\vb,\redu{0}}=1\qquad v_{\vb,\redu{1/6}}=1  \]
\[v_{\vb,\redu{1/3}}=2\qquad v_{\vb,\redu{1/2}}=1\qquad v_{\vb,\redu{2/3}}=1\qquad w_{\va,\redu{0}}=1\qquad w_{\vb,\redu{1/2}}=1\]
\end{Example}

\begin{Lemma}\label{lem:real-reduction}
	Let $ \mathscr S \subset \prod_i \C^{v_i} $ be an orbit.  Let $ \GammaInt, \varphiInt, \vInt, \wInt $ be as above. 
	\begin{enumerate}
		\item We have an isomorphism of algebras $\fKLR^\varphi_{\mathscr{S}}(\Gamma^\Bw) \cong \fKLR^{\varphiInt}_{\mathscr{\tilde{S}}}(\GammaInt^{\wInt})$.
		
		\item
		There is an isomorphism of algebras
		$$
		\fKLR^\varphi_{\mathscr{S}}(\Gamma^\Bw) \ \cong  \ \fKLR^{\varphi'}_{\mathscr{S}'}(\GammaInt^{\wInt}),
		$$
		where $\varphi'$ is an integral flavour on $\GammaInt$ such that $\varphi' = \varphiInt - d\eta$ for some $0$--cocycle $\eta$ on $\GammaInt$, and where $\mathscr{S}'=\prod_{i,\redu{z}} \Z^{\tilde v_{i,\redu{z}}}$.
	\end{enumerate}
\end{Lemma}

\begin{proof}\hfill
	\begin{enumerate}[wide]
		\item If we have a flavoured sequence with longitude in the orbit $\mathscr{S}\subset \prod_i \C^{v_i}$, we
		can canonically lift this to a flavoured sequence in
		$\GammaInt$, by giving an element of $\Sall$ with label $i$ and
		longitude $a$ the label $(i,\redu{a})\in
		\widetilde{\vertex}$, and leaving the longitude unchanged.
		This defines a homomorphism $\fKLR^\varphi_{\mathscr{S}}(\Gamma) \to
		\fKLR^{\varphiInt}_{\tilde{\mathscr{S}}}(\GammaInt)$.  On the other hand, if we have a diagram in $\fKLR^{\varphiInt}_{\tilde{\mathscr{S}}}(\GammaInt)$, then by assumption, any strand with label  $(i,\redu{a})$ has longitude in $\redu{a}$, so we can define an inverse map just turning the label to $i$, keeping the longitude the same.  
		\item A choice of flavour defines a 1-cocycle on the graph
		$\Gamma^{\Bw}$, which defines a class $\beta\in
		H^1(\GammaInt;\C/\Z)$. 
		The equation  $\varphi_e \equiv z-w
		\pmod \Z$ exactly guarantees that $\beta$ is the coboundary of the
		$\C/\Z$-valued
		0-cocyle sending $(i,\redu{z})\mapsto -\redu{z}$.  In other words $\beta$ is trivial; one can understand $\GammaInt$ as the unique minimal cover of $\Gamma$ with this property.
		Choose a $\C$-valued $0$-cycle $\eta$ on $\GammaInt$ with the property that for
		each vertex $ (i, \xi) $, we have that $ [\eta(i, \xi)] = \xi $.  In this case, we have that 
		$\varphi' := \varphiInt-d\eta$ is integer valued, and we also achieve $\mathscr{S}' = \tilde{\mathscr{S}} + \eta$.
		
		By part (1) together with Lemma \ref{lemma: cohomologous data}, we have isomorphisms:
		\[\fKLR^\varphi_{\mathscr{S}}(\Gamma) \cong
		\fKLR^{\varphiInt}_{\tilde{\mathscr{S}}}(\GammaInt)\cong
		\fKLR^{\varphiInt-d\eta}_{\mathscr{S}+\eta}(\GammaInt) \cong
		\fKLR^{\varphi'}_{\mathscr{S}'}(\GammaInt).\qedhere\]
	\end{enumerate}
\end{proof}

\subsection{Connecting flavoured KLRW and weighted KLRW algebras}
\label{sec:connecting}

As promised, we will lay out here the parallels between the weighted and flavoured approaches.  We have a close analogy based on equating:
\bigskip

\centerline{\begin{tabular}{|c|c|}\hline
		Weighted KLRW & Flavoured KLRW\\ \hline
		weighting & flavour\\
		loading/loaded sequence & flavoured sequence\\
		weighted KLRW diagram & flavoured KLRW diagram \\
		positions of strands & longitudes\\\hline
\end{tabular}}
\bigskip

The key differences here are:
\begin{enumerate}
	\item Loaded sequences are necessarily valued in the real numbers, and exactly  match the $x$-values of the relevant KLRW diagram.  In particular, any generic horizontal slice of a weighted KLRW diagram gives a loading, and so we must be able to deform these continuously.  Small deformations of a weighted KLRW diagram genuinely change the underlying loadings.
	\item Flavoured sequences can be valued in the complex numbers or even in more general sets (see Definition \ref{def:A-flavour}).  A flavoured KLRW diagram only has well-defined longitudes at the top and bottom, and a slice in the middle has no fixed flavoured sequence, and might correspond to an order that is not compatible with any flavoured sequence.  Integrality plays an important role in flavoured KLRW algebras, since the relations depend on whether strands have integral difference; there is no corresponding notion for weighted KLRW algebras.   
\end{enumerate}
Philosophically, weighted KLRW algebras capture the behaviour of flavoured KLRW algebras in the case where all flavours on edges and all longitudes are integers.  Thus, we assume this integrality for the remainder of this section, and let $\fKLR^{\varphi}_{\Bv}=\fKLR^{\varphi}_{\mathscr{S}}$ where $\mathscr{S}=\prod_i\Z^{v_i}\subset
\prod_i\C^{v_i}$.  We'll compare with the weighted KLRW algebra $\tilde{T}_{\Bv}^{\vartheta}$ for the same Dynkin
diagram, with weights $\vartheta_e=\varphi_e-1/2$.

In this case, we make a precise connection between weighted and flavoured KLRW algebras.  This begins with a precise correspondence between flavoured and loaded sequences.  The underlying idea is simply think of the flavoured sequence as a loaded sequence, but we need to perturb this definition a small amount.
Given
flavoured sequence $(\Bi,\Ba,<)$ with integral longitudes, we consider the loaded sequence $\boldsymbol{\ell}(\Bi,\Ba,<)=(\Bi,\boldsymbol{\ell},<')$ where \[\lx_k=a_k+k\epsilon\qquad  \text{for some } 0<\epsilon
\ll\frac{1}{2n}\] 
and $<'$ is the order induced by the function $\mathfrak{a}_{\boldsymbol{\ell}}\colon \Sall\to \R$.  Note that this order is independent of $\epsilon$ given our upper bound on it.
The $x$ values of ghostly/red points is given by  $\mathfrak{a}_{\boldsymbol{\ell}}(k,e)=\lx_k+\vartheta_e=\lx_k+\varphi_e-1/2$,  which is the longitude of $(k,e)$ minus $\frac{1}{2}-k\epsilon$.  
\begin{Lemma}\label{lem:same-order}
	The flavoured sequence $(\Bi,\Ba,<)$ is equivalent to the sequence $(\Bi,\Ba,<')$ using the order from the loaded sequence $\boldsymbol{\ell}(\Bi,\Ba,<)$.
\end{Lemma}
\begin{proof}
	Consider two elements $x,y\in \Sall$, then we need to check that the relative order is the same for the flavoured and loaded sequences whenever:
	\begin{enumerate}
		\item $x=k,y=m\in \Sc$ with $k<m$: in this case, $a_k\leq a_{m}$, so 
		\[\lx_k=a_k+k\epsilon \leq a_m+k\epsilon < a_m+m\epsilon=\lx_m.\] 
		\item $x=k\in \Sc,y=(m,e)\in \Sg$: in this case, we have $x>y$ if $ a_k\geq a_m+\varphi_e$ and $x<y$ if $a_k<a_m+\varphi_e$.  On the other hand, we have \[\mathfrak{a}_{\boldsymbol{\ell}}(k) -\mathfrak{a}_{\boldsymbol{\ell}}(m,e)=\lx_k-\lx_m-\vartheta_e=a_k-a_m-\varphi_e+(k-m)\epsilon+\frac{1}{2}.\] Since $a_k-a_m-\varphi_e$ is an integer, for  $\epsilon$ sufficiently small, this is positive if $a_k\geq a_m+\varphi_e$ and negative otherwise.  
	\end{enumerate}
\end{proof}

This shows that every flavoured sequence has a loaded sequence which gives the same order.  However, the opposite is not true:  there can be loaded sequences not equivalent to those coming from any flavoured sequence, due to the integrality requirements.  The existence of non-parity idempotents in \cite{KTWWYO} is an example of this phenomenon.  

Consider diagrams which interpolate between flavoured and weighted
diagrams: they should obey all the requirements of a flavoured diagram,
but only have a choice of flavoured sequence at $y=1$, whereas at
$y=0$, they satisfy the weighted condition that the distance between
strands and ghosts is exactly given by the weights.  The result is a
$\fKLR^{\varphi}_{\Bv}$-$\tilde{T}_{\Bv}^{\vartheta}$ bimodule $\tFW$  over
the weighted and flavoured KLRW algebras.  This is the special case of the bimodule relating KLRW algebras flavoured by different sets given in Definition \ref{def:bimodule}; in particular, Lemma \ref{lem:bimodule} carefully verifies that this bimodule structure is well-defined.

In this bimodule, we can form a straight-line diagram in $\tFW$ joining the loading
$\boldsymbol{\ell}(\Bi,\Ba,<)$, and the flavoured sequence $(\Bi,\Ba,<)$ at
the bottom.
\begin{Theorem}\label{thm:flavour-weight}
	Let $(\Bi,\Ba,<)$ be a flavoured sequence, and set $\boldsymbol{\ell}=\boldsymbol{\ell} (\Bi,\Ba,<)$.  We
	have an isomorphism $e(\Bi,\Ba,<) \tFW\cong e(\boldsymbol{\ell})\tilde{T}_{\Bv}^{\vartheta}$.  The
	functor $\tFW\otimes_{\tilde{T}_{\Bv}^{\vartheta}}-$ realizes the category of modules over the flavoured KLRW
	algebra $\fKLR^{\varphi}_{\Bv}$ as a quotient of the category of modules
	over the weighted KLRW algebra $\tilde{T}_{\Bv}^{\vartheta}$, by the subcategory
	of modules killed by all loadings that correspond to an integral flavoured sequence.    
\end{Theorem}
\begin{proof}
	Any diagram in $e(\Bi,\Ba,<) \tFW$ can be factored into the straight line
	diagram joining $(\Bi,\Ba,<)$ to $\boldsymbol{\ell}$ at
	the bottom and an element of the weighted KLRW algebra at the top.
	This is gives the isomorphism $e(\Bi,\Ba,<) \tFW\cong e(\boldsymbol{\ell})\tilde{T}_{\Bv}^{\vartheta}$.
	
	This shows that as a right module over the weighted KLRW algebra, $\tFW$
	is a projective module with endomorphisms given by the flavoured KLRW
	algebra $\fKLR^{\varphi}_{\Bv}$.  The result follows.
\end{proof}

The failure of this functor to be a Morita equivalence is a ``degenerate'' property, which  happens for a relatively small set of flavours $ \varphi $.  These are analogous to aspherical parameters for Cherednik algebras or singular central characters of $U(\mathfrak{gl}_n)$ (this analogy can be made precise by realizing these algebras as Coulomb branches).

\begin{Remark}
	This result can be extended to the non-integral case, using the
	isomorphism of Lemma \ref{lem:real-reduction} so that
	$\fKLR^{\varphi}_{\mathscr{S}}$ is isomorphic to a flavoured KLRW
	algebra (possibly for a different graph) with integral flavours and longitudes.
\end{Remark}

We can construct a steadied quotient
$\FW=\sfKLR^{\varphi}_{\Bv}\otimes_{\fKLR^{\varphi}_{\Bv}} \tFW\otimes_{\tilde{T}_{\Bv}^{\vartheta}}  {T}_{\Bv}^{\vartheta}$ of $\tFW$ as well.

\begin{Proposition} \label{prop:steady-weighted} For any flavoured sequence $(\Bi,\Ba,<)$ with $\boldsymbol{\ell}=\boldsymbol{\ell} (\Bi,\Ba,<)$, we
	have an isomorphism $e(\Bi,\Ba,<) \FW\cong e(\boldsymbol{\ell}) {T}_{\Bv}^{\vartheta}$.  The
	functor $\FW\otimes-$ realizes the modules over the steadied flavoured KLRW
	algebra $\sfKLR^{\varphi}_{\Bv}$ as a quotient of modules
	over the steadied weighted KLRW algebra ${T}_{\Bv}^{\vartheta}$, by the subcategory
	of modules killed by all loadings that carry an integral flavouring.
\end{Proposition}

\begin{proof}
	Let $I_1\subset \fKLR^{\varphi}_{\Bv}$ and $I_2\subset
	\tilde{T}_{\Bv}^{\vartheta}$ be the kernels of the maps to the steadied quotients.
	Consider the module $e(\Bi,\Ba,<) \FW$.  This is by
	definition, the quotient of $e(\Bi,\Ba,<)
	\tFW=e(\boldsymbol{\ell})\tilde{T}_{\Bv}^{\vartheta}$ by the submodule $e(\Bi,\Ba,<)I_1 \tFW+e(\Bi,\Ba,<)
	\tFW I_2$.  Of course, $e(\boldsymbol{\ell}) {T}_{\Bv}^{\vartheta}$ is the
	quotient by $e(\Bi,\Ba,<)   \tFW I_2\cong e(\boldsymbol{\ell})\tilde{T}_{\Bv}^{\vartheta}
	I_2$.  Thus, we only need to prove that $e(\Bi,\Ba,<)I_1 \tFW\subset e(\Bi,\Ba,<)
	\tFW I_2$.  The submodule $e(\Bi,\Ba,<)I_1 \tFW$ is spanned by
	diagrams of the form $aeb$ where $e\in
	\fKLR^{\varphi}_{\Bv}$ is an unsteady idempotent.   If
	$\boldsymbol{\ell}'$ is the corresponding loading, then by our proof above, we can
	write $eb=me(\boldsymbol{\ell}')b'$ where $m$ is the straight line diagram joining
	$e$ to $e(\boldsymbol{\ell})$ and $b'$ is the image of $b$ under the isomorphism
	$e\tFW\cong e(\boldsymbol{\ell}) {T}_{\Bv}^{\vartheta}$.  Since $\boldsymbol{\ell}'$ is also
	unsteady, $e(\boldsymbol{\ell}')b'\in I_2$ and
	$aeb=ame(\boldsymbol{\ell}')b'\in  e(\Bi,\Ba,<)
	\tFW I_2.$
	
	This shows that $e(\Bi,\Ba,<) \FW\cong e(\boldsymbol{\ell})
	{T}_{\Bv}^{\vartheta}$.  The rest of the proof is identical to
	that of Theorem \ref{thm:flavour-weight}.
\end{proof}

\nc{\iH}{H}

\subsection{Connecting flavoured KLRW algebras and cyclotomic KLR algebras}

Choose a large integer $\iH\gg 0$.  For any sequence $\Bi\in \vertex^n$ with $ v_i $ occurrences of $ i $, we have a corresponding flavoured sequence with $(\Bi,\mathbf{h},<)$ where $\mathbf{h}=(\iH, 2\iH, \dots, n \iH)$.  We let $e(\Bi,\iH)$ be the idempotent corresponding to this flavoured sequence (see Definition \ref{def:idem}).  Similarly $ e(\Bi, -\iH) $ denotes the idempotent defined as above, except using $ -H $.

We'll need to consider the usual KLR algebra \cite{KLII,Rou2KM}; the precise presentation we want is given by \cite[Def. 2.5]{Webmerged}, with $Q_{ij}(u,v)=(u-v)^{\# j\to i}(v-u)^{\# i\to j}$.
This
is another diagrammatic algebra, which we can describe as a
special case of the KLRW algebra we defined above (as suggested by the name, this is the
opposite of the historical order these were introduced).  In this
special case:
\begin{itemize}
	\item we take all weights to be $0$, that is, all ghosts coincide with the corresponding
	corporeal strand.
	\item we have no red strands, that is, no edges connecting to the Crawley-Boevey vertex.  
\end{itemize}
We let $R_{\Bv}$ denote this algebra in the case where there are $v_i$ strands of label $i$.
The cyclotomic quotient $R^\Bw_{\Bv}$ of $R_{\Bv}$ is the quotient of this algebra by the two-sided ideal generated by $w_i$ dots on the right-most strand.   By \cite[Th. 4.18]{Webmerged}, this is the same as $T^0_\Bv$, where we include $w_i$ red strands labelled $ i$ all of whom have weight $0 $.

For any fixed integral flavour $\varphi$, let $e_{\iH}=\sum_{\Bi} e(\Bi,-\iH)\in \fKLR^\varphi_{\Bv}$ be the sum
of these idempotents.  By Proposition \ref{prop:steady-weighted} and \cite[Th. 3.6]{WebwKLR}, we have that:
\begin{Proposition}
	\label{prop: cyclotomic KLR from fKLR}
	The algebra	$e_{\iH}{}\sfKLR^{\varphi}_{\Bv}e_{\iH}$ is isomorphic to the cyclotomic KLR algebra $R^\Bw_{\Bv}$ for the quiver $\Gamma$.
\end{Proposition}

Thus, as usual, $M\mapsto e_{\iH}M$ is a quotient functor,
realizing $ R^\Bw_{\Bv}\mmod$ as a quotient of
$\sfKLR^{\varphi}_{\Bv}\mmod$ by the modules killed by
$e_{\iH}$.  Note that this result is independent of $\varphi$,
so this captures a part of the category insensitive to this flavour
beyond its integrality.  

Let $ \mathscr S $ be a fixed orbit.  The algebra $\fKLR^\varphi_{\mathscr{S}}(\Gamma^\Bw)$ is non-zero, but often it will have trivial steady quotient.  We can precisely describe when this is the case by considering the perspective of categorification.  As in Definition \ref{def:GammaS}, we use $ \mathscr S $ to define a new quiver $ \tilde \Gamma $.  Let $\lambda$ be a highest weight of $\fg_{\tilde{\Gamma}}$ such that $\al_{i}^{\vee}(\lambda)=\tilde{w}_i$.  Let $\mu =\la-\sum_{j\in \tilde{\Gamma}}\tilde{v}_j\al_j$.  One consequence of the categorification theorem for cyclotomic KLR algebras (\cite[Thm. 3.21]{Webmerged}) is that the cyclotomic KLR algebra $R^{\tilde{\Bw}}_{\tilde{\Bv}}$ is non-zero if and only if the $\mu$-weight space of the simple $\fg_{\tilde{\Gamma}}$-module $V(\la)$ is non-zero.	
Thus, reducing the integral case with Lemma \ref{lem:real-reduction} and applying  Proposition \ref{prop: cyclotomic KLR from fKLR}, we find: 
\begin{Corollary}\label{cor:weight-nonzero}
	The algebra $\sfKLR^\varphi_{\mathscr{S}}(\Gamma^\Bw)$ is non-zero if the $\mu$-weight space of $V(\la)$  is non-zero.	
\end{Corollary}
\begin{Remark}
	The ``only if'' direction of this theorem is also true, but it's a bit outside the scope of this paper to prove.  The most straightforward approach is to consider the deformation of the steadied flavoured KLRW algebra analogous to the approach \cite{Webunfurl}.  This allows us to reduce to the case of a generic flavor, where the idempotent $e_{ \iH}$ will induce a Morita equivalence to the cyclotomic KLR algebra.   
\end{Remark}
While we won't prove the full ``only if'' direction, we do require one weaker version of it.  
\begin{Lemma}\label{lem:disconnected vertex}
	If there is a component of $\tilde{\Gamma}$ that does not contain a vertex $(i,[z])$ with $\tilde{w}_{i,[z]}>0$, then  $\sfKLR^\varphi_{\mathscr{S}}(\Gamma^\Bw)=0$.  
\end{Lemma}
\begin{proof}
	As usual, we reduce to the integral case, so we can assume
	there is a component $C$ of $\Gamma$ where $\Bw$ vanishes.
	For a fixed $\gamma\in (\ft+\varphi)/W$, let $\gamma_H$ be the
	resulting element where we add an integer $H$ to each coordinate
	corresponding to a vertex in $C$.  Let $\theta$ denote the
	diagram with $e(\gamma)$ at the bottom and $e(\gamma_H)$ at
	the top, with strands joining terminals that correspond to the
	same coordinate.  For $H\gg 0$, this has the effect of moving
	all strands with labels in $C$ to the right, and all other
	strands to the left while introducing a minimal number of
	crossings.  Let $\theta'$ be the reflection of this diagram
	through a horizontal line.  For $H\gg 0$, the idempotent
	$e(\gamma_H)$ is unsteady, so $\theta$ and $\theta'$ are both
	zero in the steadied quotient.
	
	The relations (\ref{strand-bigon}--\ref{ghost-bigon1}) show
	that for any $H$, we have
	$\theta'\theta=e(\gamma)$ (this is where we use that $C$ has
	no vertex with $w_i>0$).  Thus, $e(\gamma)$ is zero in
	$\sfKLR^\varphi_{\mathscr{S}}(\Gamma^\Bw)$.  Since $ \gamma $ was arbitrary, all idempotents vanish in   $\sfKLR^\varphi_{\mathscr{S}}(\Gamma^\Bw)$ and thus the algebra must be 0.
\end{proof}

\section{Induction and restriction for flavoured KLRW algebras}

\subsection{Induction and restriction bimodules}
\label{sec:induct-restr}
We will now define induction and restriction functors for flavoured KLRW algebras, paralleling those for
other versions of KLRW algebras.  For this construction we will need a
more general version of flavoured KLRW algebras, where the longitudes
are allowed to take values in a more general set rather than $\C$. 
Let $\longi$ be a set equipped with:
\begin{enumerate}
	\item An action of a partially ordered abelian group
	$(A,+,\preceq)$. We denote the action of $a\in A$ on $x \in \longi$ by $x+a$.
	\item A partial pre-order $\preceq$ on $\longi$ compatible with the order on $A$,
	i.e. satisfying $x+a\preceq x+a'$ if $a\preceq a'$.  
	\item An equivalence relation $\sim$ on $\longi$ such that $\preceq$ defines a total order on each equivalence class.
	\item A basepoint $*\in \longi$.
\end{enumerate} 
\begin{Definition}\label{def:A-flavour}
	An $\longi$-flavouring of an oriented graph $\Gamma$ is an assignment of
	an element $\varphi_e\in A$ for each edge $e$ of $\Gamma^{\Bw}$.
	A $\longi$-flavoured sequence for a flavoured graph is a triple $(\Bi,\Ba,<)$
	consisting of an $n$-tuple
	$\Bi\in \vertex^n$, a choice of longitude $\Ba\in  \longi^n$ and
	an order $<$ on the set $\Sall$ (defined as before, see (\ref{eq:Gplus})).  We
	define the longitude of $g=(x,e)\in \Sgr$ to be $a_g=a_x+\varphi_e$ if $x \in \Sc$, and $a_g=*+\varphi_e$ if $x=\star$.  We require the properties:
	\renewcommand{\theenumi}{\roman{enumi}}
	\begin{enumerate}
		\item If $g,g'\in \Sall$ and $g>g'$, then $a_g\not\prec a_{g'}$.
		\item If $g\in \Sgr, m\in \Sc$ and $a_g\approx a_m$ (that is,
		$a_g\succeq a_m$ and $a_m\succeq a_g$), then
		$g<m$.
	\end{enumerate}
\end{Definition}

The cases where we want to apply this are:
\medskip

\centerline{\begin{tabular}{|c|c|c|c|c|}\hline
		$\longi$ & $A$ & $\preceq$ & $\sim$ & $*$ \\ \hline
		$\C$ & $\C$ & $z\preceq z'$ iff $\Re(z)\leq \Re(z')$& $z\sim z'$ iff $z-z'\in \Z$ & $0$\\
		$\R$ & $\R$ & $\leq$ & $a\sim b \, \,\forall a,b$ & 0 \\
		$\Z\times \C$& $\C$ &  lexicographic & $(m,z)\sim (m,z')$ iff $m=m'$ and $z-z'\in \Z$ & $(0,0)$\\\hline
\end{tabular}}
\medskip

\noindent By ``lexicographic,'' we mean that $(m,x) \succeq (n,y)$ if $m>n$ or if $m=n$ and
$\Re(x)\geq \Re(y)$.

The first of these cases, $\longi=\C$, recovers the definition of flavoured sequences from Definition \ref{def:Flavor}.  The second case $\longi = \R$ recovers the definition of loaded sequences.  The third case where $\longi = \Z\times \C$ will be the main case of interest in this section.

We define the $\longi$-flavoured KLRW diagrams and algebra exactly as in
Definitions \ref{def:flavoured-KLR-diagram} and \ref{def:flavoured-KLR}, with the top and bottom of the
diagram now having $\longi$-flavoured sequences.  That is:

\begin{Definition}\label{def:ell-flavoured-KLR-diagram}
	Fix a flavour $\varphi$.
	A {\bf  $\longi$-flavoured KLRW diagram (corresponding to $\varphi$)} is a collection of finitely many
	oriented smooth curves satisfying the conditions of Definition \ref{def:flavoured-KLR-diagram}.  The only change is that we must now label the top and bottom with the data of $\longi$-flavoured sequences $(\Bi,\Ba,<)$ and
	$(\Bi',\Ba',<')$, respectively. There is one important change to the compatibility conditions in  Definition \ref{def:flavoured-KLR-diagram}: instead of requiring integral difference between the labels at the top and bottom of a strand, we require them to be equivalent under $\sim$.
\end{Definition}

\begin{Definition}\label{def:ell-flavoured-KLR}
	The {\bf $\longi$-flavoured KLRW algebra} $\fKLR_{\longi}$
	is the algebra given by the $\C$-span
	of the $\longi$-flavoured KLRW diagrams, modulo isotopy preserving genericity
	and  the local relations
	(\ref{dots-1}--\ref{dumb}) if the strands involved in the relation
	all have longitudes in the same equivalence class under $\sim$; if any pair of strands involved
	have longitudes not equivalent under $\sim$, then
	we can simply isotope through a triple point or tangency.  As usual, when we
	multiply diagrams, we must be able to match the flavoured sequences
	at the top of one diagram and the bottom of the other, or the
	product is 0 by convention. 
\end{Definition}

\begin{Remark}\label{rem:general-k}
	One reason that this definition is useful, though we will not
	develop it in this paper, is to study the
	representation theory of Coulomb branches over fields $\mathbbm{k}$ of
	characteristic 0 other than $\C$.   The appropriate object
	describing this on the ``KLR side'' is a
	$\mathbbm{k}$-flavoured KLRW algebra for a preorder on $\mathbbm{k}$ refining the partial order where $a\leq b$ if and only if $b-a\in \Z_{\geq 0}$.  
\end{Remark}

As one would expect, the cases $\longi=\C$ and $\longi=\R$ discussed in the table above recover the flavoured and weighted KLRW algebras from Section \ref{sec:flavoured-KLRW}.  Since we have already covered these algebras in detail, we will concentrate on the case $\longi=\Z\times \C$.  
We denote the resulting flavoured KLRW algebra by ${}^{\Z}\fKLR^{\varphi}$.  Note that for any strand in a diagram in this algebra, the $ \Z$-components of longitude must be the same at the top and bottom of the diagram.  Thus, these strands are naturally labelled with elements of the product $\Z\times \vertex$.

Given a $ \Z \times \C $-flavoured sequence $(\Bi, \Ba, <)$, we let $v_i^{(p)}$ be the number of $k \in \Sc $ such that $ i_k = i $ and $a_k \in \{p\} \times \C $.  These quantities will be unchanged from the top to the bottom of a diagram, as they correspond to the number of strands with label $ (p,i) $.  Moreover, the top and bottom of a diagram will differ by the action of $\prod_{i,p} S_{v_i^{(p)}}\ltimes \Z^{v_i^{(p)}}$, which acts on the $\C$-component of the longitudes.  Thus it is natural to consider
${}^{\Z}\fKLR^{\varphi}_{\mathscr{S}^{(*)}}$, the subalgebra where
the longitudes of the form $\{p \} \times \C $, give points in a fixed orbit
$\mathscr{S}^{(p)}\subset \prod_{i}\C^{v_i^{(p)}}$.

We will also be interested in the
usual (longitude values in $ \C$) flavoured KLRW algebra $\fKLR^{\varphi}_{\mathscr{S}^{(p)}}$
corresponding to one fixed $ \mathscr{S}^{(p)}$.  The cases $p=0$ and $p\neq 0$ play slightly
different roles here, since only the former case keeps the
Crawley-Boevey vertex; thus $\Bw^{(0)}=\Bw$ and $\Bw^{(p)}=0$
otherwise.

We have a map
\begin{equation}
	\bigotimes _{p\in \Z}\fKLR^{\varphi}_{\mathscr{S}^{(p)}}\to
	{}^{\Z}\fKLR^{\varphi}_{\mathscr{S}^{(*)}}.\label{eq:horiz-comp}
\end{equation}
given by replacing a longitude $a$ appearing in a diagram in
$\fKLR^{\varphi}_{\mathscr{S}^{(p)}}$ with $(p,a)$, and then horizontally
composing the corresponding diagrams with $p$ increasing from left to right.

\begin{Lemma} \label{le:fTpfTZ}
	The map \eqref{eq:horiz-comp} is an isomorphism of algebras.
\end{Lemma}
\begin{proof}
	One can easily check that this map is injective by comparing
	polynomial representations;  this is parallel to the
	argument of \cite[Cor. 4.15]{Webmerged}.
	
	Thus we need only show that it is surjective.  To see this, note
	that at both the top and bottom of the diagram, strands are weakly
	ordered left to right by the first factor in their longitude.  Thus,
	two strands where these factors are different must be in the same
	order at the top and bottom of the diagram.  This implies that the
	resulting KLRW diagram can be isotoped to be the horizontal
	composition of diagrams only crossing strands with the same first
	factor.  That is, the map is surjective as well.
\end{proof}

Consider two sets $\longi, \longi'$ as above, with actions of $A,A'$ and a subset $R\subset \longi \times \longi'$;  we assume that this satisfies two conditions on the set ${}_xR=\{y\in \longi' \mid (x,y)\in R\}$ and $R_y=\{x\in \longi \mid (x,y)\in R\}$:
\begin{enumerate}
	\item The sets ${}_xR$ and $R_y$ are closed under the equivalence relations $\sim, \sim'$.
	\item The pre-order $\preceq$ induces a total order on ${}_xR$ and $R_y$ and on the set of equivalence classes in these sets: i.e. if $x_1\preceq x_2$, then $x'_1\preceq x_2'$ whenever $x_1\sim x_1'$ and $x_2\sim x_2'$.  
\end{enumerate}
\begin{Definition}\label{def:bimodule}
	Let $\rif(R)$ be the set of flavoured KLR diagrams equipped with a $\longi$-flavoured sequence at the top and a $\longi'$-flavoured sequence at the bottom, modulo the flavoured KLR relations (\ref{dots-1}--\ref{dumb}) in the case where the labels $T$ at the top of the strands involved  and labels $B$ at the bottom of all strands involved satisfy $T\times B\subset R$, and the isotopy relations otherwise.  
\end{Definition}
\begin{Lemma}\label{lem:bimodule}
	The vector space $\rif(R)$ is a $\fKLR_{\longi}\operatorname{-}\fKLR_{\longi'}$ bimodule.
\end{Lemma}
\begin{proof}
	Obviously, the composition of a diagram in $\rif(R)$ with a flavoured KLRW diagram on the left or right gives a new flavoured KLRW diagram.  Furthermore, this clearly preserves the relations  in $\rif(R)$, by locality.  Thus, the only subtle point is why attaching a relation in $\fKLR_{\longi}$ gives a relation in $\rif(R)$.  
	
	First, assume we have an ``interesting'' relation in $\fKLR_{\longi}$, i.e. one involving strands whose labels are equivalent under $\sim$.  This attaches to two or three terminals at the top of the diagram, whose labels $x_1,x_2,x_3$ are all equivalent.  The labels $y_1,y_2,y_3$ on the other end of these strands may not be equivalent under $\sim'$, but since $R_{y_i}$ is closed under $\sim$, we have that $(x_i,y_j)\in R$, and so this is also a relation in $\rif(R)$.  
	
	Now, assume we have a ``boring'' relation in $\fKLR_{\longi}$, i.e. at least one of the labels $x_i$ is not equivalent to the others.  As above, let $y_j$ be the labels on the other end of these strands.  If one of $(x_i,y_j)\notin R$, then the same boring relation holds in $\rif(R)$.  Thus, assume that all $(x_i,y_j)\in R$.  Thus, $x_i\in R_{y_j}$ for all $i$ and $j$.  By assumption, this means that the elements $x_1\prec x_2\prec x_3$ are totally ordered,  and we must have that $x_1$ or $x_3$ is not equivalent to $x_2$, i.e. we can't have $x_1\sim x_3\not\sim x_2$.  We thus can't have a crossing between the corresponding strands, and thus the relation cannot appear in this case.
\end{proof}
\begin{Remark}
	The bimodule $\tFW$ from Section \ref{sec:connecting} is an example of such a  bimodule $\rif(R)$, where $\longi=\C$ and $\longi' = \R$ are as in the table above, and where $R=\{(x,y)| x\in \Z, y\in \R\}$.    
\end{Remark}

We'll principally be interested in this bimodule in the case where $\longi=\C$, and $\longi'=\Z\times \C$, and the relation $R=\{(y,(m,x)) | y-x\in\Z\}$.  
This defines a $\fKLR^{\varphi}$-${}^{\Z}\fKLR^{\varphi}$
bimodule $\rif$ given by the set of flavoured KLRW diagrams with longitudes
at top given by elements of $\C$ and at the bottom by
elements of $\Z\times \C$.   We will abuse notation,
and use $\rif$ to denote the same bimodule, with the right action
transferred by the homomorphism of \eqref{eq:horiz-comp} to one of
$\bigotimes _{p\in \Z}\fKLR^{\varphi}_{\mathscr{S}^{(p)}}$.  We can also
define a ${}^{\Z}\fKLR^{\varphi}$-$\fKLR^{\varphi}$ bimodule
$\corif$ by swapping the role of top and bottom.

\begin{Definition}\label{def:flavour-res-ind}
	The functor of {\bf restriction} associated to $\Bv^{(*)}$ is the
	functor \[\Res=\Hom(\rif,-)=\corif\otimes -\colon \fKLR^{\varphi}_{\mathscr{S}}\mmod \to
	\bigotimes _{p\in \Z}\fKLR^{\varphi}_{\mathscr{S}^{(p)}}\mmod.\]
	The functor of {\bf induction} associated to $\Bv^{(*)}$ is the left
	adjoint functor \[\Ind=\rif\otimes -\colon \bigotimes _{p\in \Z}\fKLR^{\varphi}_{\mathscr{S}^{(p)}}\mmod \to \fKLR^{\varphi}_{\mathscr{S}}\mmod,\] and that
	of {\bf coinduction} is the right adjoint $\Coind=\Hom(\corif,-)$.
\end{Definition}
While this bimodule is canonical, we can describe it in terms of an
algebra homomorphism, at the price of making some non-canonical
choices.  Fix a finite subset
$\mathsf{S}^{(p)}\subset \mathscr{S}^{(p)}$ and fix an integer $ H \gg 0 $.  Consider the map $\sqcup_{p} \{p\} \times \mathsf{S}^{(p)} \to \C$ given by $(p,x)\mapsto Hp+x$.  For $H$ sufficiently large, this map is order preserving;
note that this would never be the case on all of $\Z\times \C$, hence
the need to choose finite subsets.  With the above choices fixed, we have a
homomorphism $\wp$ to $\fKLR^{\varphi}_{\mathscr{S}}$ from the subalgebra of ${}^{\Z}\fKLR^{\vartheta}$ where
we fix all longitudes to live in $\sqcup_p \{ p \} \times
\mathsf{S}^{(p)}$, and $\rif$ matches the induced bimodule of this
homomorphism.  We can choose $\mathsf{S}^{(p)}$ so that this
subalgebra is Morita equivalent to ${}^{\Z}\fKLR^{\varphi}$.  Thus,
using this Morita equivalence, we can define the functors of
Definition \ref{def:flavour-res-ind} as the usual
induction/restriction/coinduction functors of a ring homomorphism.

\subsection{Comparison with other constructions}

Let us restrict to the integral case throughout this section.  Note that, as usual, we can reduce the general case to the integral case using Lemma \ref{lem:real-reduction}.

\begin{Lemma}\label{lem:weighted-res}
	Under the quotient functor of Theorem
	\ref{thm:flavour-weight}, the induction/restriction functors for
	flavoured KLRW algebras match the induction/restriction
	functors for weighted KLRW algebras defined in  \cite[Def. 2.17]{WebwKLR}:
	\[\tikz[->,thick]{
		\matrix[row sep=12mm,column sep=25mm,ampersand replacement=\&]{
			\node (d) {$\fKLR^{\varphi}_{\Bv}\mmod$}; \& \node (e)
			{$\displaystyle \bigotimes_{k\in
					\Z}\fKLR^{\varphi}_{\Bv^{(k)}}\mmod$}; \\
			\node (a) {$\tilde{T}^{\vartheta}_{\Bv}\mmod$}; \& \node (b)
			{$\displaystyle \bigotimes_{k\in
					\Z}\tilde{T}^{\vartheta}_{\Bv^{(k)}}\mmod$}; \\
		};
		\draw[<->] (a) -- (b) node[above,midway]{$\operatorname{Res}$} node[below,midway]{$\operatorname{Ind}$};
		\draw (a) -- (d) node[left,midway]{$\FW\otimes-$} ;
		\draw (b) -- (e) node[right,midway]{${}^{\Z}\FW\otimes -$};
		\draw[<->] (d) -- (e) node[above,midway]{$\operatorname{Res}$} node[below,midway]{$\operatorname{Ind}$};
	}\]
\end{Lemma}
In fact, we could define the restriction functors for weighted KLRW algebras using a bimodule $\rif(R)$ from Definition \ref{def:bimodule}: take $\longi=\R$ which gives weighted KLRW algebras, take $\longi'=\Z\times \R$ equipped with lexicographic order and the equivalence relation $(m,x)\sim (m',x')$ iff $m=m'$, and take the full relation $R=\R\times (\Z\times \R)$.

Let us consider the interaction between these functors and steadied
quotients.  Note that an idempotent is unsteady  if and only if it is in the image of the ring
homomorphism discussed above with $\Bv^{(p)}\neq 0$ for some
positive integer $p$.  In particular, any module $M$ over
$\fKLR^{\varphi}_{\Bv}$ that factors through the steadied
quotient $\sfKLR^{\varphi}_{\Bv}$ is killed by $\Res$
if $\Bv^{(p)}\neq 0$ for some
positive integer $p$.

We can think about this a bit more systematically by considering the
tensor product $\sfKLR^{\varphi}_{\Bv}\otimes \rif$.
The right action on this tensor product obviously factors through
the steadied quotient of ${}^{\Z}\fKLR^{\varphi}$ (repeating
Definition \ref{def:sfKLR} verbatim).  Since all strands with labels
$(p,i)$ with $p>0$ form a group that unsteadies the sequence, we must
have no such strands in the steadied quotient.  Similarly, if the
strands with label $(0,i)$ considered on their own give an unsteady
sequence, the same is true of the sequence as a whole.  From these
observations, it's easy to check that:
\begin{Lemma} \label{le:ResfTO}
	Given  $\Bv^{(0)}, \Bv^{(-1)},
	\cdots, \Bv^{(-m)}$ with $\Bv=\Bv^{(-m)}+\cdots+ \Bv^{(0)}$, we have an
	isomorphism of algebras between the  steadied quotients of the
	algebras 
	${}^{\Z}\fKLR^{\varphi}$ and
	$\fKLR^{\varphi}_{\Bv^{(-m)}}\otimes
	\fKLR^{\varphi}_{\Bv^{(-m+1)}} \otimes \cdots \otimes \fKLR^{\varphi}_{\Bv^{(0)}} $.  Thus, in this case, we obtain a well-defined functor
	\[\Res=\Hom(\rif,-)=\corif\otimes -\colon \sfKLR^{\varphi}_{\Bv}\mmod \to \fKLR^{\varphi}_{\Bv^{(-m)}}\otimes  \fKLR^{\varphi}_{\Bv^{(-m+1)}}\otimes \cdots \otimes \,{}_{-}\hspace{-.2mm}{\sfKLR}^{\varphi}_{\Bv^{(0)}}\mmod.\]
\end{Lemma}

Note that this restriction functor commutes with inflation from
steadied quotients to the full algebra.  The same is not true of its
left and right adjoints.

\subsection{Categorical actions}
\label{sec:categorical-actions}

In this subsection, we assume that $\Gamma$ has no edge loops, and
continue to only consider the integral case.

We now focus on the functor from Lemma \ref{le:ResfTO} in the case where  $m=1$, and $\Bv^{(-1)}$ is a
multiple of a unit vector, i.e. supported on a single vertex $i$, with
some multiplicity $k$.  Let $\operatorname{NH}_k $ be the nilHecke algebra on $k$ strands; this is the algebra given by KLR diagrams (without red strands, ghosts or longitudes) with $k$ strands that have the same label, satisfying the relations (\ref{dots-1}--\ref{strand-bigon}; see \cite[\S 2.2(3)]{KLI}. 
\begin{Lemma}
    For any idempotent $e((i,\dots, i),\Ba,<)$, we have an isomorphism \[e((i,\dots, i),\Ba,<)\fKLR^{\varphi}_{\Bv^{(-1)}}e((i,\dots, i),\Ba,<)\cong \operatorname{NH}_k, \] and this induces a Morita equivalence between $\fKLR^{\varphi}_{\Bv^{(-1)}}$ and $\operatorname{NH}_k$. 
\end{Lemma}
 This result depends on the lack of edge loops;  if $\Gamma$ were to have an edge loop, $\fKLR^{\varphi}_{\Bv^{(-1)}}$ would be much more complicated.  
\begin{proof}
In this case, all strands have the same label and there are no ghosts or red strands,  so only relations (\ref{dots-1}--\ref{strand-bigon} are relevant. Thus, simply forgetting labels and longitudes gives the desired map to $\operatorname{NH}_k $ and adding them back its inverse.  For any pair $((i,\dots, i),\Ba,<)$ and $(i,\dots, i),\Ba',<')$, we can simply label the identity diagram in $\operatorname{NH}_k $ of $k$ vertical lines with $\Ba$ in the order $<$ at the top and $\Ba'$ in the order $<'$ at the bottom, and obtain an isomorphism between these idempotents, showing that $e((i,\dots, i),\Ba,<)$ generates $\fKLR^{\varphi}_{\Bv^{(-1)}}$ as a 2-sided ideal.  Thus we have the desired Morita equivalence.  
\end{proof}
The core result that makes all of higher representation
theory work is that $\operatorname{NH}_k$ is isomorphic to the rank
$n!$ matrix algebra over the symmetric polynomials
$\operatorname{Sym}_k$ in $k$-variables.

Thus, the ring $\operatorname{NH}_k\otimes {}_{-}\hspace{-.2mm}{\sfKLR}^{\varphi}_{\Bv^{(0)}}$ is Morita equivalent to
$\operatorname{Sym}_k\otimes  {\sfKLR}^{\varphi}_{\Bv^{(0)}}$.
\begin{Definition}\label{def:divided-power}
	The divided power functor
	$\EuScript{E}_i^{(k)}\colon {\sfKLR}^{\varphi}_{\Bv}\mmod \to
	{\sfKLR}^{\varphi}_{\Bv^{(0)}}\mmod$ is the composition of the restriction functor
	$ {\sfKLR}^{\varphi}_{\Bv}\mmod\to \operatorname{NH}_k\otimes \sfKLR^{\varphi}_{\Bv^{(0)}}\mmod$, followed by this Morita equivalence
	and forgetting the action of $\operatorname{Sym}_k$.  Let $\EuScript{F}_i^{(k)}$ be the left adjoint of $\EuScript{E}_i^{(k)}$.  
\end{Definition}

Exactly as in Lemma \ref{lem:weighted-res}, these functors match under the  quotient functor of Theorem
\ref{thm:flavour-weight} with the categorical Lie algebra action of \cite[Th 3.1]{WebwKLR}.  Thus we have that:
\begin{Proposition}\label{prop:flavoured-cat-action}
	The functors $\EuScript{E}_i$ and $\EuScript{F}_i$ for $i\in \vertex$ give a categorical $ \fg_\Gamma$-action sending the weight $\mu=\sum w_i\varpi_i-v_i\alpha_i$ to the category ${
		\sfKLR}^{\varphi}_{\Bv}\mmod$.
\end{Proposition}

It is natural to ask which representation of $\fg_\Gamma$ is categorified by Proposition \ref{prop:flavoured-cat-action}.  In the papers \cite{KTWWY,KTWWYO}, we showed that for bipartite simply-laced Kac-Moody types this representation is described by the product monomial crystal.  In \cite{Gibson}, it is shown that in finite type A this representation can be identified with a generalized Schur module, and can be described via a generalized Demazure module in all finite types.  In general, we will now explain that this representation always surjects onto the irreducible representation with highest weight $ \lambda = \sum_i w_i \varpi_i $, and can sometime be identified with a tensor product of fundamental representations.  

We always have an equivariant map $\bigoplus {
	\sfKLR}^{\varphi}_{\Bv}\mmod\to \bigoplus R^{\Bw}_{\Bv}\mmod$
induced by the quotient functor $M\mapsto e_{\iH} M$ (see Proposition \ref{prop: cyclotomic KLR from fKLR}), and thus an equivariant map $K_{\C}(\bigoplus_{\Bv} {
	\sfKLR}^{\varphi}_{\Bv}\mmod)\to
K_{\C}(\bigoplus_{\Bv}R^{\Bw}_{\Bw}\mmod)$.  This latter Grothendieck
group is an irreducible representation of $\mathfrak{g}_{\Gamma}$
with highest weight $\la$ by a special case of
\cite[Th. B]{Webmerged}.  This map is typically not an isomorphism,
but can be in extremely degenerate cases, such as when all
$\varphi_e=0$.  

For the next result, we will need to assume that $\varphi_e=0$ for all old edges.  For $\Gamma$ without a cycle, all
cases can be reduced to this one by adding a 1-coboundary, using
Lemma \ref{lem:real-reduction}(2).   

\begin{Lemma} Let $ \lambda_1, \dots, \lambda_n $ be fundamental weights with $ \sum \lambda_k = \lambda $.  Suppose that $\varphi_e=0$ for all old edges and $|\varphi_e - \varphi_{e'}|>\sum v_i$ for all new edges $ e, e'$. Then,
	$K_{\C}(\bigoplus {
		\sfKLR}^{\varphi}_{\Bv}\mmod)\cong V(\lambda_1) \otimes \cdots \otimes V(\lambda_n) $ as $ \fg_\Gamma$-modules.
\end{Lemma}
\begin{proof}
	By \cite[Th. 3.6]{WebwKLR}, in this case, the weighted KLRW algebra is  the KLRW algebra $T^{\boldsymbol{\la}}$, so ${
		\sfKLR}^{\varphi}_{\Bv}\mmod$ is a quotient of
	$T^{\boldsymbol{\la}}\mmod$.  The kernel of this quotient is the
	modules killed by all idempotents which cannot be realized as a
	flavoured sequence.  
	
	Consider an idempotent in $T^{\boldsymbol{\la}}$; the potential difficulty of realizing this with a flavoured sequence is that if we want to make sure that two strands are in the correct order, they may need to have different longitudes, and there might not be enough different longitudes between the values $\varphi_e$ and $\varphi_{e'}$ for two new edges.
	
	However,   the condition $|\varphi_e - \varphi_{e'}|>\sum v_i$ guarantees that there are enough different possible longitudes between any two of these values to accommodate all corporeal strands, all with different longitudes, and thus having any order we desire.  This shows we have a Morita equivalence and completes the proof, since $K_{\operatorname{CC}}(T^{\boldsymbol{\la}}) \cong V(\la_1)\otimes \cdots \otimes V(\la_n)$ by \cite[Theorem 4.38]{Webmerged}.
\end{proof}

\begin{Remark}\label{rem:non-int-action} As we've done several times now, we can consider the
	non-integral case using the usual trick of Lemma
	\ref{lem:real-reduction}.   In this case, we'll apply the theorems
	above for the quiver $\GammaInt$.  In particular, it might be that
	$\Gamma$ has edge loops and $\GammaInt$ does not. For example, when
	$\Gamma$ is the Jordan quiver and the flavour on its edge is not
	integral, $\GammaInt$ will have no edge loops; it will be a finite
	subset of a union of cycles or infinite linear quivers, depending on
	whether the label is rational or irrational.  This case has been extensively considered in
	\cite{WebRou}; in particular, its Grothendieck group is calculated
	in  \cite[Th. B]{WebRou} as a higher level Fock space.  
\end{Remark}

\section{Parabolic restriction and flavoured KLRW algebras}

In this section, we'll apply the
theory developed in Section \ref{sec:functors} to the
case of a quiver gauge theory.

We begin by explaining how flavoured KLRW algebras appear in relation to the
Coulomb branches of quiver gauge theories.   In Theorem \ref{thm:wKLR} we show that $\CB(\Bv,\Bw)\dash\GTc_\mathscr{S}$ is equivalent to a category of modules over a flavoured KLRW algebra, and we show in Theorem \ref{th:KLR-restriction} that this equivalence is compatible with the induction and restriction functors.  In Section \ref{sec:catO} we define (divided powers of) restriction and induction functors on category $\cO$ over Coulomb branch algebras, and show in  Theorem \ref{thm:divpowers} that these induce categorical Lie algebra action, in the quiver case.

\subsection{Relation between Coulomb branches and flavoured KLRW algebras}
As before, we fix a quiver $ \Gamma $, dimension vectors $ \Bw, \Bv \in \Z_{\geq}^{\vertex} $ and write $ \Gamma^\Bw $ for the Crawley-Boevey quiver. Recall that we are considering a quiver gauge theory with \[G=\prod
GL(v_i)\qquad
F = (\Cx)^{\edge(\Gamma^\Bw)} \quad N =\bigoplus_{e \in \edge(\Gamma^\Bw)}\Hom(\C^{v_{t(e)}},\C^{v_{h(e)}}) \]

\begin{Remark}
	Note that there is some
	redundancy here.  The map from $\tG$ into $GL(N)$ is usually not injective,
	Thus, any flavours
	that agree up to a coweight into the kernel give isomorphic
	specializations of $\CB(G,N)$; these are precisely the weightings
	which are cohomologous when thought of as 1-cocycles.
	
	Similarly any coweights conjugate in the normalizer of $G$ in
	$GL(N)$ (even if they are not conjugate in $\tG$) give isomorphic specializations
	of the Coulomb branch.  For example, we can permute the
	flavours on edges joining the same pair of vertices.
\end{Remark}

A point $ \varphi \in \mathfrak f $ is given by a map $ \edge(\Gamma^\Bw) \rightarrow \C $.  This is the same data as a flavour (Definition \ref{def:Flavor}).
For the remainder of this section, we will fix such a flavour $ \varphi\colon Z\to \C$ and study only the Gelfand-Tsetlin modules for $ \CB(G, N) $ where the center $Z$ acts by this character.  Of course, this is equivalent to
studying modules over the quotient algebra $ \CB_\varphi:=\CB_\varphi(G,N)$.

The image of the Gelfand-Tsetlin algebra in $\CB_{\varphi}$ is $ \C[\ft + \varphi]^W $.  Since $ \varphi $ is fixed and $ \tilde \ft = \ft \oplus \mathfrak f $, we will identify $ \ft + \varphi = \ft $.  Thus a Gelfand-Tsetlin weight for $ \CB_\varphi $ will be given by a point $ \gamma \in \ft/W = \prod_i
\C^{v_i}/ S_{v_i} $.  As in Section \ref{sec:metric-weighted-klr}, we will think of $ \gamma $ as a collection $(\gamma_i)$ of size $v_i$ multisets.  By Definition \ref{def:idem}, there is a corresponding idempotent $ e(\gamma)$ where
$\gamma_i$ gives the longitudes $a_k$ with label $i_k=i$.

Given $ \gamma \in \ft/W$, a lift $ \lambda \in \ft$ corresponds to choosing an ordering on each multiset.  Let $ W^\lambda \subset W $ be the stabilizer of $ \lambda $; this is simply the product of symmetric groups corresponding to repeated elements in each $ \gamma_i $.  Thus, $ W^\lambda \ne \{1\} $ if and only if there is repetition in one of the multisets $ \gamma_i $.  Equivalently, we
have consecutive corporeal elements of $\Sall$ in the flavoured sequence coming from $ \gamma $
with the same vertex and longitude. As in \cite[\S 5.2]{KTWWYO},
crossings of consecutive strands with the same longitude and label
generate a copy of the nilHecke algebra of $ W^\lambda $ in $e(\gamma) \fKLR e(\gamma)$.  We let $e'(\gamma)$ be a primitive idempotent in this
nilHecke algebra; for concreteness, we can take this to be the
projection to $W^\lambda$.

Consider the flavoured KLRW algebra $\fKLR^{\varphi}(\Gamma^\Bw)$
for the flavour $\varphi$.
Let $\fKLR^{\varphi}\wgmod$ be the
category of locally finite dimensional weakly gradable modules over $\fKLR^{\varphi}$, that is, the
modules with a filtration whose subquotients are gradable and where $e(\Bi,\Ba,<)M$ is finite dimensional for all such idempotents.  
We can also characterize these modules topologically:  $\widehat{\fKLR}^{\varphi}$ carries its {\bf grading topology}, the coarsest topology where the elements of
degree $\geq k$ for each $k$ define a basis of neighborhoods of the identity.
\begin{Lemma}
    A locally finite dimensional $\fKLR^{\varphi}$ module $M$ is weakly gradable if and only if as a discrete topological module, it carries a continuous action of the completion
$\widehat{\fKLR}^{\varphi}$ in the grading topology.  
\end{Lemma}
\begin{proof}
    We use several times here that $\fKLR^{\varphi}$ only has finitely many idempotents up to isomorphism.  In particular, this means that a locally finite dimensional weakly gradable module is killed by all elements of sufficiently high degree, so indeed it carries an action of this completion.  

    This also means that  the category of locally finite modules over this ring is Artinian and the quotient by the ideal generated by elements of degree $\geq k$ is Morita equivalent to a finite dimensional graded algebra.  Thus 
if $M$ is a module with such a continuous action, it contains a simple submodule, which is gradable by the Morita equivalence mentioned above.  Induction shows that $M$ is weakly gradable.  
\end{proof}

\begin{Definition}
	Let $ \widehat{\mathscr A}_{\varphi}$ be the category with objects $\ft/W$ and morphisms given by
	$$
	\Hom_{ \widehat{\mathscr A}_{\varphi}}(\gamma, \gamma') = \Hom(\Wei_\gamma, \Wei_{\gamma'}),
	$$
	where $\Wei_\gamma$ denotes the weight functor on the category of $ \CB_\varphi$-modules.
	
	Let $ \widehat{\mathscr T}_{\varphi}$ be the category with objects $\ft/W$ and morphisms given by
	$$
	\Hom_{\widehat{\mathscr T}_\varphi}(\gamma, \gamma') =  e'(\gamma') \cdot \widehat{\fKLR}^{\varphi} \cdot
	e'(\gamma).
	$$
\end{Definition}
Both of these categories have the property that $\Hom$ between $\gamma$ and $\gamma'$ is trivial unless $\gamma$ and $\gamma'$ are in the image of a single $\widehat{W}$ orbit in $\ft$.  

Thus, we can write $ \widehat{\mathscr A}_{\varphi}$ as a direct sum of subcategories $\widehat{\mathscr A}_{\mathscr{S}}$ where $\mathscr{S}$ ranges over $\widehat{W}$-orbits in $\ft$, and similarly with $\widehat{\mathscr T}_{\mathscr{S}}\subset \widehat{\mathscr T}_{\varphi}$.

Recall that in Definition \ref{def:Acat}, we defined a category $
\widehat{\mathscr A}_\Z(G,N) $ with objects $ \tilde{\ft}_\Z / W $.  We define
$ \widehat{\mathscr A}_{\Z, \varphi}(G,N) $ to be the full subcategory of
this category whose objects are the subset  $ (\ft_\Z + \varphi)/ W $.  Note that if $ \varphi $ is integral and we choose $ \mathscr S = \ft_\Z + \varphi $, then our two definitions coincide: $ \widehat{\mathscr A}_{\mathscr S} =  \widehat{\mathscr A}_{\Z, \varphi}(G,N) $.

In fact, understanding this integral case gives us all the tools we need understand all blocks.  Recall the construction of the quiver $\GammaInt$ in Definition \ref{def:GammaS}, together with its dimension vectors $\tilde{\Bv},\tilde{\Bw}$ and flavour $\varphi'$ used in Lemma \ref{lem:real-reduction}.  Moreover, we let $ G', N'$ be  the gauge group and matter representation associated to the quiver $\GammaInt$, with dimension vectors $\tilde{\Bv},\tilde{\Bw}$.

By  \cite[Cor. 4.7]{WebGT}, we have the following.
\begin{Lemma}\label{lem:A-integral}
	The subcategory $\widehat{\mathscr A}_{\mathscr{S}}$ is equivalent to $ \widehat{\mathscr A}_{\Z, \varphi'}(G',N')$.
\end{Lemma}  
See \cite[\S 4.2.3]{SWGT} for a discussion of how this result can be applied to understand all blocks of Gelfand-Tsetlin modules for $U(\mathfrak{gl}_3)$.

The following result is a direct generalization of \cite[Thm. 5.2]{KTWWYO}. The proof of this theorem will be given below.
\begin{Theorem}\label{thm:wKLR}
	We have an equivalence $ \widehat{\mathscr T}_{\varphi} \cong
	\widehat{\mathscr A}_{\varphi}$.  In particular, for any set $ \mathsf S \subset \ft $, there is an isomorphism
	\begin{equation} \label{eq:FSfKLR}
		\wtAlg(\mathsf S) \cong \fKLR_{\mathsf S}^\varphi
	\end{equation}
\end{Theorem}

From the theorem and Proposition \ref{thm:GTc-F}, we immediately conclude the following:
\begin{Corollary} \label{co:wKLR}
	For any orbit $\mathscr S \subset \ft $, there is an equivalence $$ \Wei : \CB(G,N)\dash\GTc_\mathscr{S}  \rightarrow {\fKLR}^{\varphi}_{\mathscr{S}}\wgmod$$
	such that for any $M \in \CB(G,N)\dash\GTc_\mathscr{S} $ and any $ \gamma \in \bar{ \mathscr S} $, we have
	$$
	\Wei_\gamma(M) = e'(\gamma) \Wei(M)
	$$
\end{Corollary}

Before beginning the proof of this theorem, we will examine the spaces $Y_\gamma, X_\lambda $ (from section \ref{se:recollect}) in the quiver case.  Assume for now that $\varphi$ is integral; Theorem \ref{thm:wKLR} holds for all $\varphi$, but we will reduce to the integral case in the proof, and for simplicity, we only state Lemmata \ref{le:gamma-flag} and \ref{le:XeqX} in the integral case. For any $ \gamma \in \prod_i \Z^{v_i}/S_{v_i} $, let $ |\gamma_i^{\le k}| = | \gamma_i \cap (-\infty, k]|$, i.e. the number of elements in the multiset $ \gamma_i $ which are at most $ k $.

\begin{Definition}
	A partial flag of type $ \gamma $ is a $ \Z$-indexed sequence of nested subspaces
	$$ \cdots \subseteq F_0 \subseteq F_1 \subseteq \cdots \subseteq \bigoplus_{i \in \vertex \sqcup \{ \infty \}} \C^{v_i} $$
	such that each $ F_k $ is compatible with the decomposition $ \bigoplus \C^{v_i} $ and for all $ k \in \Z, i \in I$ we have
	$$
	\dim F_k \cap \C^{v_i} = |\gamma_i^{\le k}|
	$$
\end{Definition}

Here we include $ \infty $ as the one-dimensional Crawley-Boevey vertex, with $ \C^{v_\infty} = \C_\infty $.  We assume that $\gamma_\infty = \{ 0\}$ which forces \[ \dim F_k \cap \C_\infty = \begin{cases}
	1 & \text{ if }k \ge 0\\	
	0 & \text{ if }k < 0\\	
\end{cases}
\]
Each partial flag of type $ \gamma $ gives rise to an ordinary partial flag (indexed by dimension) in each $ \C^{v_i} $.

A careful examination of the definitions of $ X_\lambda, Y_\gamma$ leads to the following result.
\begin{Lemma}\label{le:gamma-flag}
	Consider $ \gamma \in \prod_i \Z^{v_i}/S_{v_i} $ and let $
	\lambda \in \prod_i \Z^{v_i} $ a lift of $\gamma$.
	\begin{enumerate}
		\item		The partial flag variety $ G/ P_\lambda $ is isomorphic
		to the space of partial flags of type $ \gamma $.
		\item  Applying the second description of \eqref{Ygamma-def}:
		$$	Y_\gamma = \{ (F_\bullet,n) \in G/P_\lambda \times N \mid n_e(F_k) \subset F_{k + \varphi_e} \text{ for all $ e \in \edge(E^\Bw), k \in \Z $}\}$$
		\item The fibre of $ X_\lambda \rightarrow Y_\gamma $ over $ (F_\bullet, n) $ is given by choosing full flags in each $\C^{v_i} $, refining the partial flags coming from $ F_\bullet $.
	\end{enumerate}
\end{Lemma}

From Lemma \ref{lem:sequence-from-flavours},  we can associate a $\varphi$-flavoured sequence to each $\gamma $ where the entries of each multiset $ \gamma_i $ give the longitudes $a_k$ with $i_k=i$.  Then given the flavoured sequence, we can construct a loading $ \boldsymbol{\ell} $ as explained in Section \ref{sec:connecting}.

To each loading $ \boldsymbol{\ell} $, in \cite[Def. 4.2]{WebwKLR}, the second author defined a space $ X_{\boldsymbol{\ell}}$ of $ \boldsymbol{\ell}$-loaded flags and compatible representations; these assign a subspace $F'_a$ to each real number $a$ whose dimension vector is given by summing the labels on real numbers $\leq a$ under the loading.  Assume that $\boldsymbol{\ell} $ is the loading  $\boldsymbol{\ell} (\gamma)$ defined in Section \ref{sec:connecting}, attached to the idempotent $e(\gamma)$.    Given an element of $X_{\boldsymbol{\ell}}$, we have subspaces $F_k=F'_{k+\nicefrac 12}$ which define a partial flag of type $\gamma$ since there are precisely $|\gamma_i^{\le k}|$ appearances of $i$ in this loading on real numbers $\leq a+\nicefrac 12$, which is an element of $Y_{\gamma}$  by Lemma \ref{lem:same-order}.  Since considering the spaces $F'_{k+\epsilon m}$ for different values of $m$ gives a refinement of this to a complete flag, we have that:
\begin{Lemma} \label{le:XeqX}
	There is an isomorphism $ X_{\boldsymbol{\ell}} \cong X_\lambda $.
\end{Lemma}

\begin{proof}[Proof of Theorem \ref{thm:wKLR}]
	Let us fix a single orbit $\mathscr{S}$ and only consider the construction of the equivalence of this block.  By
	Lemma \ref{lem:real-reduction} and Lemma \ref{lem:A-integral}, we can reduce to the case where the flavour $\varphi$ is integral and $\mathscr{S}=\ft_\Z+\varphi$.  

	By Corollary \ref{co:XA}, this reduces to showing that for all $ \gamma, \gamma' \in \ft_\Z $,
	\begin{equation}
		e'(\gamma') \phantom{\cdot} \widehat{\fKLR}^{\varphi} 
		e'(\gamma)\cong \widehat{H}^{{G}}(
		{}_\gamma\pStein_{{\gamma'}}). \label{eq:f-Y}
	\end{equation}
	
	By (\ref{eq:Stein-projection}) applied with $G=L$, it suffices to establish an isomorphism
	\begin{equation}
		e(\gamma') \phantom{\cdot} \widehat{\fKLR}^{\varphi} 	e(\gamma)\cong \widehat{H}^{{G}}(
	{}_\lambda\Stein_{{\lambda'}}).\label{eq:f-X}
	\end{equation}
	where $ \lambda, \lambda' $ are antidominant lifts of $ \gamma, \gamma' $.
	
	Since by Theorem \ref{thm:flavour-weight},  the tensor product $\tFW\otimes_{\tilde{T}_{\Bv}^{\vartheta}}-$ is a quotient functor, we have that $\tFW$ is projective as a left module, and  $\Hom(\tFW,\tilde{T}^{\varphi}  e(\boldsymbol{\ell}) )\cong {\fKLR}^{\varphi}  e(\gamma)$ where $\boldsymbol{\ell}$ is the loading associated to $e(\gamma)$ in the weighted KLRW algebra.  This shows that $e(\gamma') 
	{\fKLR}^{\varphi}  e(\gamma) \cong e(\boldsymbol{\ell}')  \tilde{T}^{\varphi}  e(\boldsymbol{\ell})$ 
	since
	\begin{equation*}
		e(\boldsymbol{\ell}') \tilde{T}^{\varphi} e(\boldsymbol{\ell})\cong \Hom(\tFW\otimes_{\tilde{T}_{\Bv}^{\vartheta}}{\fKLR}^{\varphi}  e(\gamma) ,\tilde{T}^{\varphi}  e(\boldsymbol{\ell}) ) \cong \Hom({\fKLR}^{\varphi}  e(\gamma), {\fKLR}^{\varphi} e(\gamma'))\cong e(\gamma') 
		{\fKLR}^{\varphi}  e(\gamma) 
	\end{equation*} 
	On the other hand, \cite[Thm. 4.5]{WebwKLR} describes this part of the weighted KLRW algebra using equivariant homology of a fibre product.  The isomorphism for flavoured algebras can also be written directly using the same philosophy:
 \begin{enumerate}
     \item diagrams with no dots, no double crossings, and no pair of strands with a same label and longitudes that differ by $\Z$ crossing are sent to the homology class of preimage of the diagonal under the map ${}_\lambda\Stein_{{\lambda'}}\to G/B\times G/B$.
     \item dots are sent to the Chern classes of tautological line bundles on this preimage in ${}_\lambda\Stein_{{\lambda}}$.
     \item diagrams with a single crossing of a pair strands with a same label and longitudes that differ by $\Z$, and no other crossings are sent to the preimage of the diagonal in $G/P\times G/P$, where $P$ is the parabolic where we add in the simple root corresponding to the crossing.
 \end{enumerate}
 
 By Lemma \ref{le:XeqX}, this fibre product is isomorphic to $ {}_\lambda\Stein_{{\lambda'}}$.
	All of these isomorphisms are compatible with composition and convolution in homology, so this defines the desired equivalence of categories.
\end{proof}

\subsection{Functors in the quiver case} \label{se:FunctorQuiver}
Fix a coweight $ \xi : \Cx \rightarrow  T $ as in Section \ref{se:Quiverxi}.  We will now see how the functors $\res,\ind $ and $\coind$, as defined in Section \ref{sec:functors}, interact
with the equivalence of Theorem \ref{thm:wKLR}.

As discussed in
Section \ref{se:Quiverxi}, a
choice of $\xi$ gives us dimension vectors $\Bv^{(p)}$ for $p\in \Z$.  Let $ \gamma \in \ft / W = \prod_i \C^{v_i}/S_{v_i} $, as in the previous section.   Choosing $\nu \in \ft/W_L = \prod_p \prod_i \C^{v_i^{(p)}}/S_{v_i^{(p)}}$ lifting $\gamma$ means dividing each multiset $ \gamma_i $ into multisets $ \nu_i^{(p)} $ of sizes $ v_i^{(p)}$.  With this in mind, Definition \ref{def:xineg} translates into the following statement.
\begin{Lemma}
	$\nu$ is $ \xi$-negative if and only if
	\begin{enumerate}
		\item for all $  p, q  \in \Z, e \in \Gamma^{\Bw}, y \in \nu_{t(e)}^{(p)}, z \in \nu_{h(e)}^{(q)}$ we have
		\begin{gather*}
			\text{if $ p < q $, then $ \varphi_e + z - y \notin \Z_{>0} $} \\
			\text{ and if $ p > q $, then $ \varphi_e + z - y \notin \Z_{\leq 0}$}
		\end{gather*}
		\item For all $ i \in I $ and $ p \ne q \in \Z $, if $ y \in \nu^{(p)}_i, z \in \nu^{(q)}_i $, then $ y \ne z $.
	\end{enumerate}
\end{Lemma}
(Here as usual, we adopt the convention that $ \nu^p_\infty = \{0\} $, if $ p = 0 $, and is empty otherwise.)

We now fix an orbit $\mathscr{S}\subset
\ft + \varphi$ for $\widehat{W}$, and $\mathscr{S}^L\subset
\mathscr{S}$ an orbit of $\widehat{W}_L$.  Since fixing $\mathscr{S}$ requires
choosing the multisets of the fractional parts in $
\prod (\C/\Z)^{v_i} / S_{v_i} $, $\mathscr{S}^L$ corresponds to a
division of these fractional parts into submultisets of the correct size.
If all $\gamma_i \in \Z^{v_i}/S_{v_i}$, then we have $
\mathscr{S}=\mathscr{S}^L$.  

The orbit
$\mathscr{S}^L$ is the product of orbits $\mathscr{S}^{(p)}$ of the
extended affine Weyl groups for the smaller gauge groups $L_p = \prod_i GL(\C^{v_i^{(p)}})$.

Now, we choose a set $\mathsf{S}\subset \mathscr{S}$ whose
corresponding flavoured sequences give a transversal of the equivalence
classes. This must be a complete set in the sense of Definition \ref{de:CompleteSet} since if $\gamma$ and $\gamma'$
give equivalent flavoured sequences, by Theorem \ref{thm:wKLR}, the straight line diagram between
them gives an isomorphism $\Wei_\gamma\cong \Wei_{\gamma'}$
as functors.

We
can choose similar sets $\mathsf{S}^{(p)}\subset \mathscr{S}^{(p)}$;
we will assume that these are chosen so that any point in the product
$\prod_{p\in \Z}\mathsf{S}^{(p)}$, interpreted as an element of
$\ft/W_L$, is $\xi$-negative.  To see that this $\xi$-negativity property can be achieved, suppose that 
$\mathsf{S}^{(p)}$ is to an arbitrary complete set of
$\mathscr{S}^{(p)}$.  Fix an integer $ H \gg 0 $.  Then modify $ \mathsf S^{(p)} $ by adding $pH$ to all the entries in each multiset $ \gamma_i^{(p)} \in \mathsf S^{(p)}$ for all $ p $.  The resulting elements of $ \prod_{p \in \Z} \mathsf S^{(p)} $ will then be $ \xi$-negative.

We define  $\mathsf{S}^L=\prod_{p\in \Z}\mathsf{S}^{(p)}$, and as
discussed before, we can assume that $ \mathsf{S}^L\subset \mathsf{S}$.

As discussed in Section \ref{se:Quiverxi}, we have $ \CB(L,N) = \bigotimes_p \CB(L^{(p)}, N^{(p)}) $.  This leads to an isomorphism
$$
\wtAlg^L(\mathsf S^L) = \bigotimes_{p \in \Z} \wtAlg^{L^{(p)}}(\mathsf{S}^{(p)})
$$

Combining this isomorphism with those from Corollary \ref{co:wKLR} and Lemma \ref{le:fTpfTZ} we obtain,

\begin{equation} \label{eq:FSLfKLR}
	\wtAlg^L(\mathsf{S}^L) \cong \bigotimes_{p\in
		\Z}\wtAlg^{L^{(p)}}(\mathsf{S}^{(p)}) \cong \bigotimes_{p\in
		\Z}\fKLR^{\varphi}_{\mathsf{S}^{(p)}} \cong {}^\Z \fKLR^{\varphi}_{\mathsf{S}^{(*)}}.
\end{equation}

Recall that in \eqref{eq:I-def}, we defined a $ \wtAlg(\mathsf{S})\operatorname{-}\wtAlg^L(\mathsf S^L)$ bimodule
$\wtBim(\mathsf{S}^L,\mathsf{S})$.  Also recall that in Section \ref{sec:induct-restr}, we defined a $\fKLR^{\varphi}_{\mathsf{S}}\operatorname{-}   {}^\Z \fKLR^{\varphi}_{\mathsf{S}^{(*)}}$-bimodule $ \rif $.  By completing with respect to the grading, we obtain a $\widehat{\fKLR}^{\varphi}_{\mathsf{S}}\operatorname{-}   {}^\Z \widehat{\fKLR}^{\varphi}_{\mathsf{S}^{(*)}}$-bimodule $ \widehat{\rif} $.  Similarly, we have bimodules $\wtBim(\mathsf{S}, \mathsf{S}^L)$ and $ \widehat{\corif}$ for these algebras in the opposite orders.

\begin{Lemma}\label{lem:bim-iso}
	Under the isomorphisms (\ref{eq:FSfKLR}) and (\ref{eq:FSLfKLR}), the bimodule
	$\wtBim(\mathsf{S}^L,\mathsf{S})$ corresponds
	to the bimodule
	$\widehat{\rif}$, and $\wtBim(\mathsf{S},\mathsf{S}^L)$ to $\widehat{\corif}$.
\end{Lemma}

\begin{proof}
	By \eqref{eq:I-def}, the bimodule  $\wtBim(\mathsf{S}^L, \mathsf{S})$ is
	given by the sum of Homs between weight
	functors 
	\[\wtBim(\mathsf{S}^L, \mathsf{S})\cong
	\bigoplus_{\substack{\gamma'\in \bar{\mathsf{S}} \\ \nu\in \bar{
				\mathsf{S}^L}}}\Hom(\Wei_{\nu}^L\circ \res,\Wei_{\gamma'}).\]
	For each $\nu\in \mathsf{S}^L, \gamma'\in \mathsf{S}$, let $\gamma$ be
	the image of $\nu$ in $\tilde{\ft}/W$.  Using the isomorphism $\rb'$ defined in \eqref{eq:rb-def}, we can
	define an isomorphism $ \Hom(\Wei_{\nu}^L\circ
	\res,\Wei_{\gamma'})\cong \Hom(\Wei_{\gamma},\Wei_{\gamma'})$.

	Theorem \ref{thm:saturation} shows that the induced isomorphism of
	vector spaces
	\[\Hom(\Wei_{\nu}^L\circ
	\res,\Wei_{\gamma'})\cong 	\widehat{H}^{G}(
	{}_\gamma\pStein_{{\gamma'}}) \qquad
	\wtBim(\mathsf{S}^L,\mathsf{S})\cong \bigoplus_{\substack{\gamma'\in \mathsf{S}\\\nu\in
			\mathsf{S}^L}}	\widehat{H}^{G}(
	{}_\gamma\pStein_{{\gamma'}}) \]
	is an isomorphism of bimodules where the left actions are intertwined
	by the functor $\mathsf{E}_G$, and the right action by $\mathsf{E}_L$
	and the saturation map (\ref{eq:G-sat}).

	By Theorem \ref{thm:wKLR}, we also have an isomorphism
	$\Hom(\Wei_{\gamma},\Wei_{\gamma'}) \cong e'(\gamma') \cdot  \widehat{\fKLR}^{\varphi} \cdot
	e'(\gamma) $.  At the start of the proof, we chose a preimage $\nu$ for
	$\gamma$; since $\mathsf{S}^L=\prod_{p\in \Z}\mathsf{S}^{(p)}$, we can
	let $\nu^{(p)}\in \mathsf{S}^{(p)}$ be the projection of $\nu$.  By Lemma \ref{lem:sequence-from-flavours}, there exists a flavoured sequence with $ \nu^{(p)} $ as a the set of longitudes.  Concatenating these flavoured sequences together gives us a $ \Z \times \C$ flavoured sequence which matches the sequence for $ e(\gamma) $ after projecting to the second factor. For each diagram in $e'(\gamma') \cdot  \widehat{\fKLR}^{\varphi} \cdot
	e'(\gamma)$, we can consider the same diagram with this sequence placed at the bottom and obtain an element of $\widehat{\rif}$.

	By definition, the resulting diagram is an element of $e(\gamma')\cdot
	\widehat{\rif}\cdot e^\Z(\nu) $,
	where $e^{\Z}(\nu)$ is the corresponding idempotent in ${}^{\Z}
	\widehat{\fKLR}^{\varphi} $.  In fact, this operation on diagrams defines an
	isomorphism of vector spaces
	\[e(\gamma')\cdot
	\widehat{\rif}\cdot e^\Z(\nu)\cong e(\gamma') \cdot
	\widehat{\fKLR}^{\varphi} \cdot
	e(\gamma)\]
	where the map is simply projecting the longitudes in $\Z\times \C$ at
	the bottom of the diagram to their second component, and thus an isomorphism of vector spaces $\wtAlg(\mathsf{S}^L,\mathsf{S})\cong \widehat{\rif}$.

	To show that this linear isomorphism is furthermore a bimodule map, we need only confirm that the
	isomorphism \eqref{eq:f-X} used in the proof of Theorem \ref{thm:wKLR}
	is compatible with the saturation map (\ref{eq:G-sat}) i.e. the
	commutativity of the
	diagram
	\begin{equation*}
		\begin{tikzcd}[sep=huge]
			\widehat{H}^{{L}}( {}_\la\Stein_{{\la'}}^L) \arrow{r} [above]{\text{(\ref{eq:G-sat})}} \arrow{d}[left]{\eqref{eq:f-X}} &
			\widehat{H}^{G}(  {}_\la\Stein_{{\la'}}) \arrow{d}{\eqref{eq:f-X}}\\
			e^{\Z}(\nu') \cdot  {}^{\Z}\widehat{\fKLR}^{\varphi} \cdot
			e^{\Z}(\nu)
			\arrow{r}{}   &   e(\gamma') \cdot  \widehat{\fKLR}^{\varphi} \cdot
			e(\gamma)
		\end{tikzcd}
	\end{equation*}
	where the lower horizontal arrow is simply forgetting the
	$\Z$-component of the longitude, and $\la,\la'$ are lifts of $\nu,\nu'$ to $\ft_L+\varphi\cong \ft+\varphi$.
	
	In order to see this, we need to check that the images of the minimal
	diagrams, single crossings and single dots are related by the map
	(\ref{eq:G-sat}).  This is clear from the definition of
	\eqref{eq:f-X}, which ultimately depends on the map given in
	\cite[Thm. 4.5]{WebwKLR}. The inverse of \eqref{eq:f-X} sends each dot
	to the Chern class of a tautological bundle, and each single crossing to the
	saturation of a homology class in $ {}_{\zeta}\Stein_{{\zeta}}^{L'}$ for a Levi
	$L'\subset L$
	with one simple root and $\zeta \in \ft_\Z /W_L $ arbitrary.  By the transitivity of saturation, this
	crossing gives classes in ${}_{\zeta}\Stein_{{\zeta}}^L$ and
	${}_{\zeta}\Stein_{{\zeta}}$ which are compatible under saturation.
\end{proof}

From the discussion above,  we can restate Theorem
\ref{thm:Splus} using the functors of Definition
\ref{def:flavour-res-ind} as:
\begin{Theorem}\label{th:KLR-restriction}
	The equivalence of Theorem \ref{thm:wKLR} intertwines the functors
	$\ind_\xi$ and $\res_{\xi}$ with the functors of induction and
	restriction for flavoured KLRW algebras:
	\[\tikz[->,thick]{
		\matrix[row sep=12mm,column sep=25mm,ampersand replacement=\&]{
			\node (d) {$\fKLR^{\varphi}_{\mathscr{S}}\wgmod$}; \& \node (e)
			{$ \bigotimes_{p\in
					\Z}\fKLR^{\varphi}_{\mathscr{S}^{(p)}}\wgmod$}; \\
			\node (a) {$\CB(G,N)\dash\GTc_\mathscr{S}$}; \& \node (b)
			{$\CB(L,N_0^\xi)\dash\GTc_{\mathscr{S}^L}$}; \\
		};
		\draw (a) -- (b) node[below,midway]{$\res_\xi$};
		\draw (a) -- (d) node[left,midway]{$\Wei$} ;
		\draw (b) -- (e) node[right,midway]{$\Wei^L$};
		\draw (d) -- (e) node[above,midway]{$\operatorname{Res}$};
	}\]
\end{Theorem}

\subsection{Category $\cO$}\label{sec:catO}

Now, we consider how these results apply to category $\cO$, as
described in Section \ref{sec:category-co}.  

Let us temporarily return to the general context of Section \ref{sec:functors}, with our usual notation $ G, N, \xi, L $.  We also choose a character $\chi\colon G\to \Cx$, which defines a category $ \Oof{\CB(G,N)}$.  The standard choice for a quiver gauge theory, following Nakajima, is given by the
product of the determinant characters on $GL(v_i)$.

The character $\chi$ can be restricted to $L$, and thus also defines a
category $\Oof{\CB(L,N^{\xi}_0)}$ of modules over $\CB(L,N^{\xi}_0)$.
Unfortunately, the interaction of category $\cO$ and restriction
functors is quite complicated.  
\begin{Lemma}\label{lem:chi-xi}Consider a module $M \in \Oof{\CB(G,N)}$.
	\begin{enumerate}
		\item If $\chi(\xi)\geq 0$, then $\res_{\xi}(M)=0$.
		\item If $\chi(\xi)<0$, then $\res_{\xi}(M)\notin \Oof{\CB(L,N^{\xi}_0)}$ unless
		$\res_{\xi}(M)=0$.  However,
		$$\Supp \res_{\xi}(M) =    \Supp(M)+\Z\xi.$$
	\end{enumerate}
\end{Lemma}

Here we use the support of a Gelfand-Tsetlin module from Definition \ref{def:GT}.
Of course, the same result holds if we reverse all signs.
\begin{proof}
	Fix any weight $ \nu \in \tilde{\ft}/W_L $.  For all sufficiently large $ k \in \Z_{\geq 0} $ we have 
	$$ \Wei_\nu(\res_\xi(M)) \cong \Wei_{\nu - k \xi}(\res_\xi(M)) \cong  \Wei_{\nu - k \xi}(M)$$
	where the first isomorphism comes from the fact that $ r_\xi^k : \Wei_{\nu - k \xi} \cong \Wei_{\nu} $ is an isomorphism of weight functors for $ \CB(L, N^\xi_0) $ and the second isomorphism comes from the fact that $ \nu - k \xi $ is $ \xi $-negative and so Theorem \ref{th:res-weight} applies.
	
	If $ \chi(\xi) > 0 $, then for sufficiently large $ k $, $ \Wei_{\nu -k\xi}(M) = 0 $ (since $ M $ is in category $ \cO$) and thus $ \Wei_\nu(\res_\xi(M)) = 0 $.  Since $ \nu $ is arbitrary, this forces $ \res_\xi(M) = 0 $.

	Assume now that $ \chi(\xi) = 0$.  Since $ M $ is in category $ \cO $, $ M^P $ (see section \ref{sec:res-ind}) is in category $\cO $ for $ \CB(L,N) $.  Then consider the system (\ref{eq:dir-system})
	\begin{equation*} 
		\cdots \overset{r_{\xi}}\longrightarrow \Wei_{\nu+\xi}^{L,N}(M^P) \overset{r_{\xi}}\longrightarrow \Wei_{\nu}^{L,N}(M^P) \overset{r_{\xi}}\longrightarrow  \Wei_{\nu-\xi}^{L,N}(M^P) \overset{r_{\xi}}\longrightarrow\cdots
	\end{equation*}
	Since $\chi(\xi) = 0 $, each space in this system lies in the same $ \chi$-eigenspace.  Since $M^P $ is in category $ \cO$, the $ \chi$-eigenspaces are finite-dimensional.  Thus only finitely many spaces in this system are non-zero.  Hence by Lemma \ref{lem:weight-direct-limit}, we conclude that $ \Wei_\nu(\res_\xi(M)) = 0 $.

	If $ \chi(\xi) < 0 $, then the above observation implies that $ \Supp(\res_{\xi}(M)) = \Supp(M)+\Z\xi$.  In particular this means that the eigenspaces of $ \xi$ cannot be bounded below.
\end{proof}

Now, we return to the quiver case, where we can make some more precise statements.  
Recall that
$\fKLR^{\varphi}_{\mathscr{S}}$ has a steadied quotient
$\sfKLR^{\varphi}_{\mathscr{S}}$, defined in Definition
\ref{def:sfKLR}.  Applying Proposition \ref{prop:O-F} in the quiver
case, we find that:
\begin{Theorem}\label{thm:O-steady}
	Fix a flavour $ \varphi $ and an orbit $ \mathscr S $.  An module $M \in \CB_\varphi(\Bv,\Bw)\dash\GTc_\mathscr{S}$ lies in category $\cO$ if and only if the
	corresponding $\fKLR^{\varphi}_{\mathscr{S}}$ module factors through the steadied quotient
	$\sfKLR^{\varphi}_{\mathscr{S}}$.   Thus
	$\Oof{\CB_{\varphi}(\Bv,\Bw)}_{\mathscr{S}}\cong
	\sfKLR^{\varphi}_{\mathscr{S}}\mmod$.
\end{Theorem}
\begin{proof}
	The proof is the same as that of \cite[Thm 5.21]{KTWWYO}.
\end{proof}
Let $ H $ be a real number.  We say that $ \gamma \in \ft/W $ is $H$-\textbf{bounded above} if for all $ i \in I $ and $ z \in \gamma_i$, we have $ \Re (z) \le H $.    The following result can be viewed as a specialization of  Proposition \ref{prop:GTinO} to the quiver situation:
\begin{Lemma}\label{lem:big-gamma}
	Fix a flavour $ \varphi $ and an orbit $ \mathscr S $.  For $H$ sufficiently large, each $ M \in \CB_\varphi(\Bv,\Bw)\dash\GTc_\mathscr{S}$ satisfies:  $$ M \in \Oof{\CB_{\varphi}(\Bv,\Bw)}_{\mathscr{S}} \quad \text{ if and only if } \quad \text{ every $ \gamma \in \Supp M $ is $H$-bounded above} $$
\end{Lemma}
\begin{proof} 
	Assume that every $\gamma \in \Supp M $ is $H$-bounded above. Let $n = \sum v_i$.  
	Let $ C \subset \mathscr S $ be an equivalence class, such that $ \Wei_\gamma(M) \ne 0 $ for some $ \gamma \in C $.  Since all elements of $ C $ are $ H $-bounded above, the function $ \gamma \mapsto \Re(\chi(\gamma)) $ is bounded above on $ \mathscr S $ by $ H (n-1) $.  Moreover in $ C $, there are only finitely many $ \gamma $ such that $ \Re(\chi(\gamma)) $ is larger than any given real number.  Thus $ C $ is $-$bounded in the sense of Section \ref{sec:category-co} and so by Proposition \ref{prop:GTinO}, we have $M\in \Oof{\CB_{\varphi}(\Bv,\Bw)}_{\mathscr{S}}$.

	Now assume that $M$ is in $\Oof{\CB_{\varphi}(\Bv,\Bw)}_{\mathscr{S}}$.  Let $ \gamma \in \Supp M $. Let $H'=\max(|\Re(\varphi_e)|)$.  We claim that $ \gamma $ is $H=2H'n$-bounded above.
	
	Let $ (\Bi, \Ba, < ) $ be a flavoured sequence corresponding to $
	\gamma $ by Lemma \ref{lem:sequence-from-flavours}.  Suppose that $
	\Re(a_k) > 2H'n $ (if not we are done).  Then the intervals $ [a_k -H', a_k +H'] $ (for $ k = 1, \dots, n$ ) cannot cover $[0,2H'n] $, so there must be a real number $H''\in [0,2H'n]$ such that
	there is no $\Re(a_k)$ in the interval $[H''-H',H''+H']$.  This
	means that all the strands with real longitude $>H''+H'$ have all
	ghosts with real longitude $>H''$, and all strands  with real longitude $<H''-H'$ have all ghosts with real longitude $<H''$. 
	
	Then consider $ \{ g \in \Sall : \Re(a_g) > H'' \}$.  This set is
	non-empty and consists exactly of the last $ k $ elements of $ \Sall $ for
	some $ k $ and all their ghosts.  Thus we see that $ (\Bi,\Ba, <) $ is unsteady and hence $ \Wei_\gamma(M) = 0 $ by Theorem \ref{thm:O-steady}.
\end{proof}

Recall the notation $ \xi $ and associated $ \Bv^{(p)}$ as defined in Section \ref{se:Quiverxi} and used in the previous section.
From Lemma \ref{lem:chi-xi}, it is possible to show that if $ \Bv^{(p)} \ne 0 $ for some $ p> 0 $, then $ \res_\xi(M) = 0 $ for all $M $ in category $\cO$. Thus we only consider those $ \xi $ for which $ \Bv^{(p)} = 0 $ for $ p > 0 $.  In fact, we will focus on $ \xi $ such that only $ \Bv^{(0)}$ and $\Bv^{(-1)} $ are non-zero.

\subsection{Divided power functors}
Assume that $\Gamma$ has no edge loops, and let us turn to rephrasing the divided power functors of
Definition \ref{def:divided-power} in terms of
Coulomb branches.  Fix $i\in \vertex$ and $ k \le v_i $.  Consider a coweight $\xi\colon \Cx \to \prod_j
GL(v_j)$ which is the trivial extension of the coweight $t\mapsto
\operatorname{diag}(t,\dots, t, 1,\dots, 1)\in GL(v_i)$ with $k$ diagonal
entries equal to $t$ and $v_i-k$ entries 1.  For this choice of $ \xi$,
$\Bv^{(-1)}$ is $k$ times in the unit vector on $i$, and
$\Bv^{(0)}=\Bv-\Bv^{(-1)} = \Bv - k e_i$.

Restriction (Definition \ref{def:res}) defines a functor
\begin{equation} 
	\res_\xi \colon \CB(\Bv, \Bw)\mmod
	\to\CB(\Bv - k e_i, \Bw)\otimes  \CB(GL_k,0)\mmod.\label{eq:dp-restriction}
\end{equation}
By definition, $\CB(GL_k,0)$
is the equivariant homology of the affine Grassmannian of $GL_k$,
which is isomorphic to the quantum Toda algebra by \cite[Theorem
3]{BezFink}.  

\subsubsection{Non-integral decomposition} \label{se:nonint}
Let us specialize to the case $ k = 1$ momentarily.  The Hamiltonian reduction of $ \CB(\Cx,0) = D(\Cx) $ by $ r_\xi$ is $\C $ (equivalently, we can use Theorem \ref{th:Reduction}) and so we get a Hamiltonian reduction functor
$$
\CB(\Bv - e_i, \Bw)\otimes  \CB(\Cx,0)\mmod \rightarrow \CB(\Bv - e_i, \Bw)\mmod \quad M \mapsto M / (r_\xi - 1) M
$$
as discussed in Section \ref{sec:Hamiltonian reduction}.

We define 
$$
\res_i\colon \CB(\Bv, \Bw)\mmod
\to\CB(\Bv -  e_i, \Bw)\mmod$$ to be the composition of (\ref{eq:dp-restriction}) with this Hamiltonian reduction functor.  By Lemma \ref{le:preserveGT} and Proposition \ref{prop: qhr and gt}, the functor $ \res_i $ takes Gelfand-Tsetlin modules to Gelfand-Tsetlin modules.  

We apply the paragraph below Proposition \ref{prop: qhr and gt} (with the natural choice of $ a = \xi$ after identifying $ \ft = \ft^* $ using the trace form) and so we can decompose the Hamiltonian reduction over $ c \in \Cx $.  This gives us a decomposition of functors $ \res_i = \sum_{ c \in \Cx} \res_{i,c}$.  Moreover, we have $ \res_{i,c}(M) = \Wei^a_{\log c}(\res_\xi(M)) $ where $ \Wei^a_{\log c} $ denotes the weight space for $ \CB(\Cx, 0) $.

\subsubsection{Integral case}
Now, we return to general $ k $ and further examine the case of integral Gelfand-Tsetlin modules.

Let $\res^{(k)}_i\colon \CB(\Bv,\Bw)\dash\GTc_\Z
\to \CB(\Bv - k e_i,\Bw)\dash\GTc_\Z$ be the functor $\res_\xi$ of \eqref{eq:dp-restriction}, followed by passing to
the generalized weight space $\Wei_{0}$
for $ \CB(GL_k,0)$.  (If we specialize $ k = 1$, then we see that $ \res_i^{(1)} = \res_{i,1} $.)

\begin{Lemma}\label{lem:res-O}
	The functor $\res^{(k)}_i$ restricts to a functor $\Oof{\CB(\Bv,\Bw)}_{\Z} \to \Oof{\CB(\Bv-ke_i,\Bw)}_{\Z}$.
\end{Lemma}
\begin{proof} 
	Let $ M \in \Oof{\CB(\Bv,\Bw)}_{\Z}$.  By Lemma \ref{lem:big-gamma}, there exists $ N $ such that every $ \gamma $ in $ \Supp(M) $ is $ N$-bounded above.
	
	We claim that every $ \gamma$ in $ \Supp( \res^{(k)}_i(M)) $ is also $N$-bounded above.  By definition \[ \Wei_\gamma(\res^{(k)}_i(M)) = \Wei_{\gamma \cup \{0, \dots, 0\}}(\res_\xi(M)) .\]  Since $ \Supp(\res_\xi(M)) = \Supp(M) + \Z\xi $, we see that $ \gamma \cup \{m, \dots, m\} \in \Supp(M) $ for some $ m \in \Z $.  In particular this means that $ \gamma \cup \{m, \dots, m\} $ is $ N$-bounded above, so for all $j $ and $ x \in \gamma_j$, we have $ \Re(x) \le N $.  Thus $ \gamma $ is also $N$-bounded above.  (Here $ \gamma \cup \{m, \dots, m\} $ means that we add $k$ copies of $m $ to $\gamma_i$.)
\end{proof}  
The same proof shows that our functors $ \res_{i,c} $ from the previous section all preserve category $\cO$ as well, so this is not a special property of the integral case.

In order to compute the left adjoint of $\res^{(k)}_i\colon  \Oof{\CB(\Bv,\Bw)}_{\Z}\to \Oof{\CB(\Bv - k e_i,\Bw)}_{\Z}$
it is helpful to
think of it as composition of three functors:
\begin{enumerate}
	\item The inclusion from category $\cO$ into all GT modules; the
	left adjoint to this is the functor that takes the largest quotient
	of a GT module lying in category $\cO$.
	\item The functor $\res_\xi$, with its usual left adjoint $\ind_\xi$.
	\item The functor of passing to a weight space for the action of the
	Gelfand-Tsetlin subalgebra in $\CB(GL_k,0)$. This does not have a left
	adjoint in the category of Gelfand-Tsetlin modules, but it does have one in
	the category of pro-Gelfand-Tsetlin modules: outer tensoring a
	$\CB(\Bv - k e_i,\Bw)$-module with the projective $P_0$ representing the
	functor $\Wei_{0}$.
\end{enumerate}

While $\ind_\xi(M\boxtimes P_0)$ is only a pro-Gelfand-Tsetlin module,
its maximal quotient in category $\cO$ is honestly Gelfand-Tsetlin,
and thus defines a left adjoint functor to $\res_i^{(k)}$, which we
denote $\ind_i^{(k)}$.  

\subsubsection{Compatibility with flavoured KLRW algebra functors}
Now assume that $ \varphi$ is an integral flavour.

\begin{Theorem}\label{thm:divpowers}
	The equivalence of Theorem \ref{thm:wKLR} intertwines the divided
	power functor $\EuScript{E}_i^{(k)}$ with the functor
	$\res^{(k)}_i$ defined above, and thus its left adjoint
	$\EuScript{F}_i^{(k)}$ with $\res^{(k)}_i$:  \[\tikz[->,thick]{
		\matrix[row sep=12mm,column sep=25mm,ampersand replacement=\&]{
			\node (d) {${ \sfKLR}^{\varphi}_{\Bv}\mmod$}; \& \node (e)
			{$  { \sfKLR}^{\varphi}_{\Bv - k e_i}\mmod$}; \\
			\node (a) {$\Oof{\CB_\varphi(\Bv,\Bw)}_{\Z}$}; \& \node (b)
			{$\Oof{\CB_\varphi(\Bv-n\mathbf{e}_i,\Bw)}_{\Z}$}; \\
		};
		\draw (a) -- (b) node[below,midway]{$\res^{(k)}_i$};
		\draw (a) -- (d) node[left,midway]{$\Wei$} ;
		\draw (b) -- (e) node[right,midway]{$\Wei$};
		\draw (d) -- (e) node[above,midway]{$\EuScript{E}_i^{(k)}$};
	}\]
	\[\tikz[->,thick]{
		\matrix[row sep=12mm,column sep=25mm,ampersand replacement=\&]{
			\node (d) {${ \sfKLR}^{\varphi}_{\Bv}\mmod$}; \& \node (e)
			{$  { \sfKLR}^{\varphi}_{\Bv - k e_i}\mmod$}; \\
			\node (a) {$\Oof{\CB_\varphi(\Bv,\Bw)}_{\Z}$}; \& \node (b)
			{$\Oof{\CB_\varphi(\Bv-n\mathbf{e}_i,\Bw)}_{\Z}$}; \\
		};
		\draw (b) -- (a) node[below,midway]{$\ind^{(k)}_i$};
		\draw (a) -- (d) node[left,midway]{$\Wei$} ;
		\draw (b) -- (e) node[right,midway]{$\Wei$};
		\draw (e) -- (d) node[above,midway]{$\EuScript{F}_i^{(k)}$};
	}\]
	In particular, these functors define a categorical $\fg_\Gamma$-action which sends the weight $\mu =\sum w_i\varpi_i-v_i\alpha_i$ to the category $\Oof{\CB_\varphi(\Bv,\Bw)}_{\Z}$.
\end{Theorem}

\begin{proof}
	From Theorem \ref{th:KLR-restriction}, we have the commutative diagram
	\begin{equation}
		\begin{tikzcd} \label{eq:inProof}
			\fKLR^\varphi_\Bv\wgmod \ar[r]  & \fKLR^\varphi_{\Bv - k e_i} \otimes \fKLR^\varphi_{k e_i} \wgmod   \\
			\CB_\varphi(\Bv, \Bw)\dash\GTc_\Z \ar[u] \ar[r] & (\CB_\varphi(\Bv - k e_i, \Bw) \otimes \CB_\varphi(k e_i, 0))\dash\GTc_\Z \ar[u]
		\end{tikzcd}
	\end{equation}
	From Section \ref{sec:categorical-actions}, we see that $ \fKLR^\varphi_{k e_i} $ is the nilHecke algebra $ NH_k$.  Also, $ \EuScript{E}_i^{(k)} $ is defined by following the upper horizontal arrow in (\ref{eq:inProof}) and then applying an idempotent in $ NH_k$.
	
	On the other hand, $ \res_i^{(k)} $ is defined by following the lower horizontal line in (\ref{eq:inProof}) and then taking $ \Wei_0$ for $\CB_\varphi(k e_i, 0)) = \CB(GL_k, 0) $.   So then the result follows from Corollary \ref{co:wKLR}.
\end{proof}

\begin{Remark}
	Since the functor $\EuScript{F}_i^{(k)}$ is also isomorphic to the
	right adjoint of $\EuScript{E}_i^{(k)}$ (but not canonically so), we
	also have that $\ind^{(k)}_i$ is isomorphic to the right adjoint of
	$\res_i^{(k)}$, which we can construct as tensoring with the injective
	(infinitely generated)
	GT module $P_0^*$ corepresenting $\Wei_0$ over $\CB(GL_k,0)$,
	applying $\coind_{\xi}$, and then passing to the largest submodule in
	$\cO$.
\end{Remark}

Recall the idempotent $e_{\iH} \in\fKLR^\varphi_{\Bv}$ from Proposition \ref{prop: cyclotomic KLR from fKLR}.  From Lemma \ref{lem:GKdim}, we have that:
\begin{Proposition}\label{prop:maxGK}
	Let $M\in \Oof{\CB(\Bv,\Bw)}_{\Z}$.  We have that $e_{\iH}\Wei(M)=0$ if and only if the Gelfand-Kirillov dimension of $M$ is strictly less than $\frac{1}{2}\dim M_C(\Bv,\Bw)=\sum_{i\in \vertex} v_i$.
\end{Proposition}
\begin{proof}
	If $e_{\iH}\Wei(M)\neq 0$, then there exists $ \Bi $ such that $e(\Bi,
	\iH) \Wei(M) \ne 0 $.  By the definition of $ e(\Bi, \iH)$, this means
	that if we construct $ \gamma^\iH \in \ft/W$ where $ \gamma^\iH_i = \{
	k\iH : i_k = i \} $, then $ \Wei_{\gamma^\iH}(M) \ne 0 $.  Now choose
	any integers $ a_1, \dots, a_n $ such that $  a_k \le a_{k+1} + \iH$
	and then define $\gamma $ such that $ \gamma_i = \{ a_k : i_k = i \}
	$. We have an isomorphism of functors  $ \Wei_{\gamma^\iH}(M) \cong
	\Wei_\gamma(M)$ since the corresponding idempotents in the flavoured
	KLRW algebra are equivalent, and thus $ \Wei_\gamma(M) \ne 0 $.  As the set of such $ \gamma $ is Zariski dense in $ \ft/W$, we see that $M$ has Gelfand-Kirillov
	dimension $\sum_{i\in \vertex} v_i$ by Lemma \ref{lem:GKdim}. 
	
	If $e_{\iH}\Wei(M)=0$, then any  $ \gamma $ in the support of $ M$ must
	contain a pair $ y, z $ such that the difference $|y-z|$ is an
	integer with $|y-z|<\iH$.  If there is no such pair, then $e'(\gamma)$ is isomorphic to an idempotent $e'(\gamma')$ satisfying
	$e'(\gamma')=e_{\iH} e'(\gamma')$.  The locus where $|y-z|\in
	\{-\iH+1,-\iH+2,\dots, \iH-1\}$ holds is the union of a finite
	number of hyperplanes.  Thus, the support of $M$  is not Zariski dense. Thus
	$\operatorname{GKdim}(M)< \sum_{i\in \vertex} v_i$ by Lemma
	\ref{lem:GKdim} again.
\end{proof}

Recall that  $\cO_{\operatorname{top}}(\Bv,\Bw)$ is the quotient of $\Oof{\CB(\Bv,\Bw)}_{\Z}$ by the subcategory of objects with GK dimension  $<\frac{1}{2}\dim M_C(\Bv,\Bw)$.  Combining Proposition \ref{prop: cyclotomic KLR from fKLR}, Corollary \ref{co:wKLR} and Proposition \ref{prop:maxGK} we find that:
\begin{Corollary}\label{co:irrep}
	The functor $e_{\iH}\Wei$ induces an equivalence $\cO_{\operatorname{top}}(\Bv,\Bw)\cong R^{\Bw}_{\Bv}\mmod$, and thus an isomorphism of $\mathfrak{g}_{\Gamma}$-modules $\bigoplus_{\Bv}K_{\C}(\cO_{\operatorname{top}}(\Bv,\Bw))\cong V(\lambda).$
\end{Corollary}

\subsubsection{The non-integral case} \label{rem:non-int}
Let us return to the setting non-integral GT modules, but specialize $k = 1$, as in Section \ref{se:nonint}. 
There, we defined a decomposition $ \res_i = \sum_{c \in \Cx} \res_{i,c} $ which we will now examine from the flavoured KLRW perspective. 

We have a functor $\sfKLR^{\varphi}_n\mmod \to 
\sfKLR^{\varphi}_{n-1}\otimes \fKLR_{1}\mmod$, where the subscript
simply denotes the number of corporeal strands. The single strand in the second
factor has a longitude $z$.  Each non-zero summand that arises in the decomposition of this functor on a given block
$\sfKLR^{\varphi}_{\mathscr{S}}\mmod$ is given by a vertex $ (i,[z])$  of $\GammaInt$, where $\GammaInt $ is defined using $ \mathscr S $ in Definition \ref{def:GammaS}.  Thus we obtain a functors $ \EuScript{E}_{i,[z]} : \sfKLR^{\varphi}_n\mmod \to \sfKLR^{\varphi}_{n-1}\mmod $, which match the above functors $ \res_{i,c} $ (where $ c = \exp(2\pi i z) $), as in Theorem \ref{thm:divpowers}.  
This shows that if $\operatorname{supp} M$ lies in an orbit $\mathscr{S}$ and $\res_{i,c}(M)$ is non-zero, then $(i,[\log c])\in \tilde{\Gamma}$. 

Consider the graph $\Gamma\times \C/\Z $ with the adjacency of
Definition \ref{def:GammaS}. For simplicity, we assume that this
graph has no edge loops, even if $\Gamma$ has an edge loop. Consider the graph
$\GammaO\subset \Gamma\times \C/\Z $ is given by the union of $\GammaInt$ for
all orbits $\mathscr{S}$ that appear in the support of modules in
$\cO$ for our fixed $\varphi$ and all different values of
$\Bv$. 

Using Remark \ref{rem:non-int-action} and repeating the proof of
Corollary \ref{co:irrep} shows:
\begin{Proposition}\label{prop:non-integral irrep}
	The functors $ \res_{i,c} $ give a categorical action of
	$\mathfrak{g}_{\GammaO}$ on $ \bigoplus_\Bv \Oof{\CB(\Bv,\Bw)}$.
	The corresponding Grothendieck groups $\bigoplus_\Bv K_{\C}(\cO_{\operatorname{top}}(\Bv, \Bw))$ form an irreducible representation of $\mathfrak{g}_{\GammaO}$. 
\end{Proposition}
Let $U$ be the full subgraph of $\Gamma\times \C/\Z $ containing
all those vertices connected to $ (i,[\varphi]) $ for $\varphi$ the weight of an edge connecting  the vertex $i$ to the Crawley-Boevey vertex.  Note that if $\varphi$ is integral, then $U=\Gamma\times \{[0]\}$. 

\begin{Lemma}
	We have an equality $U=\GammaO$ and the functor $\res_{i,c}$ is not zero if and only if  $(i,[\log c])\in \GammaO$. 
\end{Lemma} 
This implies that if we take $ M \in \CB(\Bv, \Bw)\dash\GTc_\Z $ an integral Gelfand-Tsetlin module, then $ \res_i(M) = \res_{i,1}(M) $, as $ \res_{i,c}(M) = 0 $ for $ c \ne 1 $. 
\begin{proof}
	To show the inclusion $U\subset \GammaO$, it's enough to show that any $(i,[\log c])\in U$ occurs as a longitude in a non-zero idempotent in $\sfKLR^{\varphi}_n$ for some $n$.  As before, we consider a highest weight $\la$ such that $\al_{j,[z]}^{\vee}(\la)=\tilde{w}_{j,[z]}$.  By Corollary \ref{cor:weight-nonzero}, it's enough to show that the irreducible representation $V(\la)$  has a nontrivial weight space for a weight $\mu$ with $\al_{i,[\log c]}^{\vee}(\mu)< 0.$  Since $\al_{i,[\log c]}$ is conjugate to all other simple roots in its component under the action of the Weyl group, we can choose $w$ in the Weyl group of $\GammaO$ such that  $-w\al_{i,[\log c]}=\al_{k,[y]}$ is a simple root with $\al_{k,[y]}^{\vee}(\la)> 0$.   This implies $\mu=w\la$ is a non-zero weight as desired, so the equality $\la=\mu+\sum_{(j,[z])\in \GammaO}\tilde{v}_{j,[z]}\al_{j,[z]}$ gives the desired dimension vector. 
	
	Now we will show that if $(i,[\log c])\in U$, then $\res_{i,c}\neq 0$. Note that our argument also shows that any vector of weight $\mu$ in $V(\la)$ has non-zero image under the action of the Chevalley generator $E_{i,[\log c]}$ since it has negative weight for the corresponding root $\mathfrak{sl}_2$. Thus, any module $M$ with non-zero image in $\cO_{\operatorname{top}}$ supported on the corresponding orbit $\mathscr{S}$ satisfies $\res_{i,c}(M)\neq 0$.

	On the other hand, by Lemma \ref{lem:disconnected vertex}, if $\sfKLR^{\varphi}(\mathscr{S})\neq 0$ for some orbit $ \mathscr{S}$, then every component of the corresponding graph $\tilde{\Gamma}$ has vertex $(j,[z])$ with non-zero $\tilde{w}_{j,[z]}$.  Thus, $\tilde{\Gamma}\subset U$.  Ranging over all $\mathscr{S}$, this shows that $\GammaO\subset U. $  We noted above that if $\res_{i,c}(M)$ is non-zero, then $(i,[\log c])\in \tilde{\Gamma}$.  This completes the proof.
\end{proof}

From a more topological perspective, the Crawley-Boevey graph $\GammaO^{\tilde{\Bw}}$
can be built as follows:
\begin{enumerate}
	\item Take the
	minimal cover of the Crawley-Boevey graph $\Gamma^{\Bw}$ that
	trivializes the flavour $\varphi$ as an element of
	$H^1(\Gamma^{\Bw},\C/\Z)$.
	\item Delete all but one pre-image of the Crawley-Boevey vertex.  This
	will be the Crawley-Boevey vertex of  $\GammaO^{\tilde{\Bw}}$.
	\item Delete all components of the pre-image  of $\Gamma$ not
	connected to the CB vertex.
\end{enumerate} 

\subsection{Cherednik algebras}
\label{sec:cherednik-algebras}

Throughout this section, we assume that $\Gamma$ is the Jordan quiver,
with dimension vector $v=n, w=\ell>0$.  
As in the introduction, we let $\CB(n,\ell)=\CB(G,N)$
where $G=GL(n)$ and
$N=\Hom(\C^n,\C^n)\oplus \Hom(\C^n, \C^\ell)$.  The case
where $\ell=0$ is somewhat different and requires slightly different
methods.  

Our aim in this subsection is to prove Theorem \ref{th:Cherednik1}.
In order to do this, let us unpack the results of the previous
sections in this case.  Our flavour consists of a $\ell$-tuple
$\varphi_1,\dots, \varphi_\ell$ corresponding to torus of
$GL(\C^{\ell})$, and a single scalar $\mathbf{t}\in \C$, which
corresponds to the $\C^\times$ action scaling the loop in $\Gamma$.

As discussed in Section \ref{sec:intro-cherednik}, Kodera and Nakajima have defined an isomorphism of algebras $\operatorname{KN}\colon\mathsf{H}({n,\ell})\to  \CB(n,\ell)$. This isomorphism matches the usual Euler grading on $\mathsf{H}({n,\ell})$ to the grading on $\CB(n,\ell)$ induced by the isomorphism $\pi_1(GL_n)\cong \Z$, so it induces an equivalence  $\mathsf{KN}\colon \Oof{\mathsf{H}({n,\ell})}\to \Oof{\CB(n,\ell)}$.  This isomorphism relates the flavour parameters to the usual parameters of the Cherednik
algebra by formulas given in \cite[Th. 1.1]{KoNa} (where $\varphi_r$ is
denoted $z_r$). We let $\mathsf{H}_{\varphi}({n,\ell})$ be the spherical Cherednik algebra specialized at these parameters.  Since we will want to compare with the results of \cite{WebRou}, we note that in that paper, the parameter $\mathbf{t}$ is denoted by $k$, and the parameters $\varphi_r$ by $s_r$.

Consider the graph with vertices $ \C/\Z $ with an edge from $ [z] \to [z + \ft] $, for all $ [z] \in \C/\Z$.  This is a union of $e$-cycles if $\mathbf{t}=m/e$ is
rational, and a union of $A_{\infty}$-components if $\mathbf{t}\not\in
\mathbb{Q}.$  For simplicity, we'll avoid the case where $\mathbf t
\in \Z$, so this graph has no edge loops.  

Given a  $\widehat{W}$ orbit $ \mathscr S \subset \C^n $, we defined  a graph $\GammaInt$ in Definition \ref{def:GammaS}.  From the definition, we see that $ \GammaInt $ is always a subgraph of this $ \C/\Z $-graph.

As discussed in Section \ref{se:nonint} above, we have a decomposition of the induction and
restriction functors $\ind_{\bullet,c},\res_{\bullet,c}$ corresponding to the single vertex of
the quiver (defined using the Hamiltonian reduction approach).  On category $\cO $, this gives a $ \fg_U $ action (by Proposition \ref{prop:non-integral irrep}) where $ c $ ranges over the vertices of $U=\GammaO$.  The analogous  decomposition of the
Bezrukavnikov-Etingof functors into $i$-induction/restriction functors $\res_{\mathsf{BE},i}$
is given by Shan \cite[Def. 4.2]{ShanCrystal}.

By Theorem \ref{thm:divpowers}, we have that
$\Oof{\CB_\varphi(n,\ell)}\cong { \sfKLR}^{\varphi}_{n}\mmod$ where $n$ gives the number
of corporeal strands.  By Lemma \ref{lem:real-reduction}, we have
that ${ \sfKLR}^{\varphi}_{n}$ is isomorphic to a flavoured KLRW
algebra over $\GammaInt$, which we can think of a flavoured KLRW
algebra for $U$ by the inclusion of $\GammaInt$ as a subgraph.  Finally, we find that by Theorem
\ref{thm:flavour-weight}, we have a quotient functor
$\Wei_!\colon {T}^{\vartheta}_{n}\mmod \to \Oof{\CB_\varphi(n,\ell)}$ from the category of modules over a weighted KLRW algebra ${T}^{\vartheta}_{n}$ for $U$, where the subscript $n$ indicates that we have a total of $ n $ corporeal strands (possibly of different labels).  

Let $\mathsf{\tilde{H}}_{\varphi}({n,\ell})$ be the full Cherednik algebra for
$G(\ell, 1,n)$. 
In \cite{Webalt}, the second-named author defined  an equivalence of categories $\mathsf{W}\colon \Oof{\mathsf{\tilde{H}}_{\varphi}({n,\ell})}\to \tilde{T}^{\vartheta}_{n}\mmod$ using a similar method to  Theorem \ref{thm:wKLR}.   
\begin{Lemma}\label{lem:KN-compat} We have an isomorphism of functors
	$\Wei_!(-) \cong \mathsf{KN}(e \mathsf{W}^{-1}(-))$
\end{Lemma}
\begin{proof}
	To see this, we note that that the three functors involved are uniquely characterized by matching three polynomial representations:
	\begin{enumerate}
		\item the polynomial representation of $\mathsf{\tilde{H}}({n,\ell})$ defined in \cite[(2.17-22)]{Webalt}, extending the representation of the spherical part given in \cite[Thm 1.5]{KoNa}.  
		\item the GKLO representation of $\CB(n,\ell)$.
		\item the polynomial representation of $\tilde{T}^{\vartheta}_{n}$, defined in \cite[Prop. 2.7]{WebwKLR}.
	\end{enumerate}
	The Kodera-Nakajima isomorphism can be defined as the unique one matching (1) and (2), the functor $\Wei$ is uniquely defined by matching (2) and (3) and from the definition 
	\cite[Lemma 3.12]{Webalt}, we see that the functor $\mathsf{W}$ is uniquely characterized by matching (1) and (3).  This shows the desired compatibility. \end{proof}

However, at the moment, we do not know how to prove the compatibility
of our induction and restriction
functors with Bezrukavnikov and Etingof's under the equivalence of
categories induced by the Kodera-Nakajima isomorphism.  Instead, we have to use a potentially different equivalence of categories, which was explicitly constructed with this property in mind.

In \cite[Th. 4.8]{WebRou}, the second-named author constructed a
potentially different equivalence
$\mathfrak{W}\colon \Oof{\mathsf{\tilde{H}}_{\varphi}({n,\ell})}\to
{T}^{\vartheta}_{n}\mmod$ using the method of uniqueness of 1-faithful
covers. Since this equivalence is constructed using
\cite[Th. 2.3]{WebRou},  it is partly characterized by the fact that
it intertwines the Bezrukavnikov-Etingof induction and
restriction functors with the induction and restriction functors
for weighted KLRW algebras induced by the inclusion $\tilde{T}^{\vartheta}_{n-1}\hookrightarrow \tilde{T}^{\vartheta}_{n}$.  Note, this equivalence depends on a
choice of isomorphism between the Hecke algebra of $G(\ell,1,n)$
(which appears here as endomorphisms of the KZ functor) and the
cyclotomic KLR algebra inside ${T}^{\vartheta}_{n}$ used in the
application of \cite[Th. 2.3]{WebRou}, so it is not unique.  We expect that if this isomorphism is chosen correctly, then it will induce an isomorphism of functors $\mathfrak{W}\cong \mathsf{W}$, but there are a number of variations on this isomorphism possible, and it is a difficult calculation to check if any of them is right.

\begin{proof}[Proof of Theorem \ref{th:Cherednik1}]
	From the discussion above, we have functors:
	\[\tikz[->,thick]{
		\matrix[row sep=12mm,column sep=25mm,ampersand replacement=\&]{
			\node (a) {$\Oof{\CB_{\varphi}(n,\ell)}$}; \&  \node (d) {$ \tilde{T}^{\vartheta}_{n}\mmod$};\& \node (g)
			{$  \Oof{\mathsf{\tilde{H}}_{\varphi}({n,\ell})}$};\\
			\node (b)
			{$\Oof{\CB_{\varphi}({n-1},\ell)}$}; \& \node (e)
			{$ \tilde{T}^{\vartheta}_{n-1}\mmod$}; \& \node (f)
			{$  \Oof{\mathsf{\tilde{H}}_{\varphi}({n-1,\ell})}$};\\
		};
		\draw (a) -- (b) node[left,midway]{$\res_{\bullet}$};
		\draw (d) -- (a) node[above,midway]{$\Wei_!$} ;
		\draw (e) -- (b) node[below,midway]{$\Wei_!$};
		\draw (d) -- (e) node[left,midway]{$\Res$};
		\draw (g) -- (f) node[left,midway]{$\res_{\mathsf{BE}}$};
		\draw (g) -- (d) node[above,midway]{$\mathfrak{W}$};
		\draw (f) -- (e) node[below,midway]{$\mathfrak{W}$};
	}\]
	with both squares commuting.  Thus, to complete the proof we need only
	show that the induced quotient functor $\Wei_!\circ
	\mathfrak{W}\colon \Oof{\mathsf{\tilde{H}}_{\varphi}({n,\ell})}\to \Oof{\CB_{\varphi}(n,\ell)}$ kills precisely the aspherical modules, i.e. those killed by the symmetrizing idempotent $e$. 
	
	We have a natural labelling of simple modules by multipartitions:
	\begin{itemize}
		\item in the category $\Oof{\mathsf{\tilde{H}}_{\varphi}({n,\ell})}$, we have standard modules $\Delta(\mu)$ induced from simple representations of $G(\ell, 1, n)$, and every simple $L(\mu)$ is a quotient of a unique such module.
		\item in the category $\tilde{T}^{\vartheta}_{n}\mmod$, this labeling is induced by the cellular structure of \cite[Th. 4.11]{WebRou}.
	\end{itemize}
	
	The equivalence $\mathfrak{W}$ intertwines these labelings by construction.  On the other hand, we calculate directly in \cite[Lemma 3.17]{Webalt} how  the Dunkl-Opdam subalgebra acts on the highest weight space of $ L(\mu)$;   this shows that  $\mathsf{W}$ preserves these labelings as well and hence $\mathfrak{W}(L(\mu))\cong\mathsf{W}(L(\mu))$ for all $\mu$.
	
	Thus, by Lemma \ref{lem:KN-compat}, we have
	that \[\Wei_!\big(\mathfrak W(L(\mu))\big)\cong \mathsf{KN}(eL(\mu)),\] so we have that
	$\Wei_! (\mathfrak{W}(L(\mu)))=0$ if and
	only if $eL(\mu)=0$.  This shows that $\Wei_!\circ \mathfrak{W} $
	factors through the usual quotient functor to
	$\Oof{\mathsf{{H}}_{\varphi}({n,\ell})}$, and induces an
	equivalence of that category to $\Oof{\CB_{\varphi}(n,\ell)}$
	intertwining the restriction and induction functors, completing the proof.
\end{proof}

\bibliography{./monbib}
\bibliographystyle{amsalpha}

\end{document}